\documentclass{article}
\usepackage{graphicx} 
\usepackage{graphicx} 
\usepackage[utf8]{inputenc}
\usepackage[T1]{fontenc}
\usepackage{mathptmx}
\usepackage{etoolbox}
\RequirePackage{amsthm,amsmath,amsfonts,amssymb,mathtools}
\usepackage{nicefrac}  
\usepackage{enumerate}

\renewcommand\Re{\mathop{{\rm Re}}}
\renewcommand\Im{\mathop{{\rm Im}}}

\newcommand{\Cb}{{\mathbb C}}

\newcommand{\Hb}{{\mathbb H}}
\newcommand{\Ib}{{\mathbb I}}

\newcommand{\Nb}{{\mathbb N}}

\newcommand{\Rb}{{\mathbb R}}

\newcommand{\Zb}{{\mathbb Z}}
\newcommand{\calA}{\mathcal{A}}

\newcommand{\calD}{\mathcal{D}}

\newcommand{\calF}{\mathcal{F}}

\newcommand{\calJ}{\mathcal{J}}
\newcommand{\calK}{\mathcal{K}}
\newcommand{\calL}{\mathcal{L}}
\newcommand{\calM}{\mathcal{M}}
\newcommand{\calN}{\mathcal{N}}

\newcommand{\calR}{\mathcal{R}}
\newcommand{\calS}{\mathcal{S}}

\def\colvec[#1,#2]{\begin{bmatrix} #1 \\ #2 \end{bmatrix}}
\def\rowvec[#1,#2]{\begin{bmatrix} #1 & #2 \end{bmatrix}}
\def\ip<#1,#2>{\left\langle #1,#2 \right\rangle}

\newcommand{\Cos}{\mathop{\mathrm{Cos}}}
\newcommand{\cosec}{\mathop{\mathrm{cosec}}}

\newcommand{\diag}{\mathop{\mathrm{diag}}}

\newcommand{\HR}{H\kern-2pt R}

\newcommand{\LIM}{\mathop{\mathrm{LIM}}}
\newcommand{\nullity}{\mathop{\mathrm{null}}}

\newcommand{\spann}{\mathop{\mathrm{span}}}
\newcommand{\spec}{\mathop{\mathrm{spec}}}

\newcommand{\trace}{\mathop{\mathrm{trace}}}

\newcommand*{\id}{{\mathrm{id}}}
\newcommand*{\R}{\mathbb R}

\newcommand*{\pa}{{\partial}}
\newcommand*{\rd}{{\mathrm{d}}}

\newtheorem{thm}{Theorem}[section] 
\newtheorem{lem}[thm]{Lemma}
\newtheorem{prop}[thm]{Proposition}

\theoremstyle{definition}
\newtheorem{defn}[thm]{Definition}
\theoremstyle{remark}
\newtheorem{rem}[thm]{Remark}
\newtheorem{ex}[thm]{Example}

\numberwithin{equation}{section}

\begin{document}

\title{Linear systems, determinants and solutions of the Kadomtsev-Petviashvili equation}
\author{Gordon Blower${}^a$ and Simon J. Malham${}^b$}

\maketitle
{${}^a$ School of Mathematical Sciences, Lancaster
University, Lancaster LA1 4YF, United Kingdom, and}
{${}^b$Department of Mathematics, Heriot-Watt University, EH14, 4AS Scotland, UK}



\begin{abstract} 
Let $(-A,B,C)$ be a linear system in continuous time $t>0$ with input and output space ${\mathbb C}$ and state space $H$. The scattering (or impulse response) functions $\phi_{(x)}(t)=Ce^{-(t+2x)A}B$ determines a Hankel integral operator $\Gamma_{\phi_{(x)}}$; if $\Gamma_{\phi_{(x)}}$ is trace class, then the Fredholm determinant $\tau (x)=\det (I+\Gamma_{\phi_{(x)}})$ determines the tau function of $(-A,B,C)$. The paper establishes properties of algebras including $R_x = \int_x^\infty e^{-tA}BCe^{-tA}\,dt$ on $H$, and obtains solutions of the Kadomtsev-Petviashvili PDE. P\"oppe's semi-additive operators are identified with orbits of a shift action on integral kernels, and P\"oppe's bracket operation is expressed in terms of the Fedosov product. The paper shows that the Fredholm determinant $\det (I+R_x)$ gives an effective method for numerical computation of solutions of $KP$.
\end{abstract}

\par
AMS Classification 2020: 47B35, 47A48, 35Q53, 65M99\par

 \section{ Introduction}
The Kadomtsev-Petviashvili equation is
\begin{equation}\label{KPII} {\frac{3\beta^2}{4}}{\frac{\partial^2u}{\partial y^2}}+{\frac{\partial}{\partial z}}
\Bigl( \alpha{\frac{\partial u}{\partial t}}+{\frac{1}{4}}{\frac{\partial^3u}{\partial z^3}}-{\frac{3}{2}}u{\frac{\partial u}{\partial z}}\Bigr)=0.\end{equation}
There are many significant applications in physics and algebraic geometry, for which we refer the reader to [6] and [9]. For $\beta^2>0,$ we have $KPII$; whereas for $\beta^2<0$, we have $KPI$. The existence theory is different in these two cases; see [3, 4, 5]. The purpose of this paper is to produce solutions as Fredholm determinants of certain families of operators. Mulase showed that $KP$ is completely integrable in the sense of Frobenius; however, his process is algebraic except for the computation of infinitely many recursive indefinite integrals 
\cite[p. 66]{M0}, so does not furnish explicit solutions. See also \cite{MR}.\par
We obtain solutions of this equation by a method associated with Gelfand, Levitan and Marchenko, by expressing $u$ in terms of the Fredholm determinant of a certain integral operator. As in previous papers \cite{BD}, \cite{Blower1}, \cite{BlowerNew}, we introduce this determinant indirectly from a family of continuous-time linear systems and related operators. \par
The method for solving the nonlinear evolution equation splits into a forward problem, a linear evolution, and an inverse problem.\begin{enumerate}[(i)]
\item The potential $u$ is part of a linear ODE which generates spectral data, including a scattering function $\phi$. We express $\phi$ as the impulse response function of a linear system $(-A,B,C)$ with state space $H$.
\item
 The scattering data evolves according to a linear ODE, which we obtain by evolving the linear systems through a family $(-A, B(t), C(t))$. 
 \item From the $(-A, B(t), C(t))$, we recover potentials $u(\cdot ,t)$ from various determinant formulas, especially the Gelfand-Levitan equation (\ref{GLM}). Our solution of (\ref{GLM}) features a family of linear operators $R_x$ on $H$, which satisfy a Lyapunov equation (\ref{Lyap}) and algebraic identities in Propositions \ref{propdiffring} and \ref{propdiffringKP}. An aspect of (i) is producing a suitable $(-A,B,C)$ for a given $\phi$. In section \ref{S:exrem}, we achieve this explicitly for a class of $\phi$ that occurs in differential equations.
\end{enumerate}
\vskip.1in

\noindent {\bf Definition} Let $H$ be a complex separable Hilbert space, which we regard as the state space, and $H_0$ a complex finite-dimensional Hilbert space, used as the input and output space. Let ${\mathcal{L}}(H)$ be the space of bounded linear operators on $H$ with the operator norm, which contains the space ${\mathcal{L}}^2(H)$ of Hilbert-Schmidt operators as an ideal. The ideal of trace-class operators is $\calL^1(H)=\{ \Phi\Psi: \Phi, \Psi\in \calL^2(H)\}$.  A continuous-time linear system is a triple $(-A,B,C)$ where
\begin{enumerate}[(i)]
\item $-A$ is the generator of a strongly continuous semigroup $(e^{-tA})_{t>0}$ on $H$, which is bounded so $\Vert e^{-tA}\Vert_{{\mathcal{L}}(H)}\leq M$ for some $M>0$ and all $t>0$; then the domain ${\mathcal{D}}(A)$ of $A$ is a dense linear subspace of $H$ which is itself a Hilbert space for the norm $\Vert h\Vert_{{\mathcal{D}}(A)}=(\Vert h\Vert^2+\Vert Ah\Vert^2)^{1/2}$;
\item $B:{{H_0}}\rightarrow {\mathcal{D}}(A)$ is a bounded linear operator;
\item $C:{\mathcal{D}}(A)\rightarrow {{H_0}}$ is a bounded linear operator. [Alternatively, one can take $B:{\mathbb{C}}\rightarrow H$ and $C:H\rightarrow {\mathbb{C}}$ bounded.]
\end{enumerate}
\vskip.1in
\indent Then one defines the impulse response function $\phi: (0, \infty)\rightarrow {\mathcal{L}}({H_0})$ by $\phi (t)=Ce^{-tA}B$, Suppose that $\int_0^\infty t\Vert\phi (t)\Vert^2_{\calL^2(H_0)}\, dt$ converges; then  the Hankel integral operator 
\begin{equation}\Gamma_\phi f(z)=\int_0^\infty \phi (z+\zeta )f(\zeta )d\zeta \qquad (f\in L^2((0, \infty ); H_0 ))\end{equation}
defines a Hilbert-Schmidt operator on $L^2((0, \infty ); H_0).$ Then we take the family of linear systems $(-A, e^{-xA}B, e^{-xA}C)$ depending upon parameter $x>0$, and consider the corresponding impulse response functions $\phi_{(x)}(t)=Ce^{-(2x+t)A}B$, and the Hankel integral operators with kernels $Ce^{-(2x+z+\zeta )A}B$. See \cite{MPT}.\par
\indent The fundamental operator in continuous-time linear systems is the right translation operator $S_t:f(x)\mapsto f(x-t)$ on $L^2(0, \infty )$, which gives rise to a semigroup $(S_t)_{t\geq 0}$ of isometries. In section \ref{S:semiadditive}, we establish the fundamental properties of the operators $\Phi\mapsto S_t^\dagger \Phi S_t$ and $\Phi\mapsto S_t\Phi S_t^\dagger$ on the Hilbert-Schmidt operators on $L^2(0, \infty )$. Theorem \ref{sigmathrm} includes P\"oppe's fundamental identity on products of Hankel operators in terms of almost derivations on algebras of operators; see \cite{McK}, \cite[\S 5]{PKP} and \cite{PSG}. This basic concept from Hochschild theory motivates the special tools from P\"oppe's theory, notably the bracket operation. We prefer to express similar concepts in terms of operators on the state space of the linear system.\par
\indent In \cite{Blower1}, we considered the family of operators
\begin{equation}\label{Roperator}R_x=\int_x^\infty e^{-tA}BCe^{-tA}\, dt\end{equation}
on $H$, which gives a solution of Lyapunov's equation
\begin{equation}\label{Lyap}{\frac{dR_x}{dx}}=-AR_x-R_xA, \quad \Bigl({\frac{dR_x}{dx}}\Bigr)_{x=0}=-BC.\end{equation} 
As a consequence of this identity, the $R_x$ have a remarkable algebraic structure which is reflected in the tau function $\tau (x)=\det (I+R_x)$. In \cite{Blower1} Propositions 2.2 and 2.3, we gave sufficient conditions for $R_x$ to be trace class. In \cite{BlowerNew} we introduced a differential ring of operators on $H$, and used this to produce solutions of $KdV$. In the current paper, we introduce families of linear systems, and thereby solve $KP$. \par
\indent The tau function generalizes the classical notion of a theta function, as follows. Mulase \cite[Theorem 6.1, Corollary]{M} showed that an abelian variety defined over $\Cb$ is the Jacobian variety of a certain algebraic curve if and only if it can be the orbit of a $KP$ dynamical system, The orbits of the $KP$ dynamical system are linear flows in the Jacobian. This is clear for one-soliton solutions, as in (\ref{eq:onesoliton}). Shiota \cite{Shio} gives a rigorous and detailed account of how $KP$ solves Schottky's problem.\par

 \indent A significant case arises when $H_0=\Cb$, and much information is captured by the potential $q(x)=-2{\frac{d^2}{dx^2}} \log \tau (x)$. Gelfand and Levitan considered the Fourier transform of the spectral measure, and related this to wave equations. In this paper, we use a similar idea, except that the operators in wave equation have matrix coefficients and the operators are not necessarily self-adjoint, so there is no spectral measure in the usual sense. To obtain substitutes, we have a preliminary section \ref{S:spectra} which introduces relevant notions of spectrum via functional calculus for operator cosine families, as in \cite{Gold}.  \par
 \indent For reasons discussed in \cite[3.239]{Mum}, $KP$ includes as special cases several significant differential equations in mathematical physics. We proceed from the simplest cases towards the general, so that the solutions are as explicit as possible. In section \ref{S:ZS}, we obtain solutions of the Zakharov-Shabat system, and in section \ref{S:SpectraZS}, describe the related spectral theory.\par
\indent 
 In section \ref{S:KP}, we introduce linear systems with infinite-dimensional state spaces and formulate conditions that ensure the $R_x$ operators are trace class, so that the necessary tau functions exist. In section \ref{S:diffring}, we introduce the differential ring structure that is essential for solving $KP$. Equipped with this algebra, we proceed to obtain solutions for $KP$ in section \ref{S:solutionKP}.

In Section~\ref{sec:NumericalSimulations}, utilising the linear systems approach for the $KP$ equation we have presented herein,
we implement three numerical methods for computing solutions to the $KP$ equation.
The first two methods are based on numerically solving the GLM equation for given scattering data,
while the third method is based on computing the $\tau$ function via the Fredholm determinant for the $KP$ equation.
A fourth numerical method based on a direct exponential time-stepping pseudo-spectral scheme is also implemented for comparison purposes.
The scattering data represents the solution to the linearised $KP$ equation evaluated at any time $t>0$.
In the case of the GLM-based methods, for any such given scattering data, we solve the linear integral GLM equation using both a Riemann Rule, and Clenshaw--Curtis quadrature based on Chebyshev polynomial approximation; see Clenshaw and Curtis~\cite{CC}. 
This generates an approximate solution to the $KP$ equation at that time $t>0$. 
In the case of the $\tau$ function based method, we approximate the Fredholm determinant associated with the scattering data,
using the Nystr\"om--Clenshaw--Curtis method developed by Bornemann~\cite{Bornemann}.
Indeed, Bornemann's use of Clenshaw--Curtis quadrature to evaluate such determinants inspired our GLM approximation method using this quadrature.   
We implement all four numerical methods in the case of scattering data corresponding to a two-soliton interaction scenario.
Such scattering data are analytic, and we observe exponential convergence in both methods based on Clenshaw--Curtis quadrature,
as outlined by Bornemann~\cite[p.~892]{Bornemann}.

There has been much recent interest in the use of linear systems and direct linearisation methods to solve integrable systems and the $KP$ equation in particular.
The approach we adopt herein is closely related to the methods developed by P\"oppe, see P\"oppe~\cite{PSG,PKdV,PKP}, P\"oppe and Sattinger~\cite{PoppeSattinger}, Bauhardt and P\"oppe~\cite{BauPop},
McKean~\cite{McK}, and by Nijhoff, see  Nijhoff Quispel, Van Der Linden, Capel~\cite{NQVC}, Nijhoff~\cite{Nijhoff}, Nijhoff and Capel~\cite{NC}, Fu and Nijhoff~\cite{FN1,FN2,FN3}
and Fu~\cite{Fu}. Also see Dyson~\cite{Dyson}, Santini, Ablowitz and Fokas~\cite{SAF}, Mulase~\cite{M0,M}, Pelinovsky~\cite{Pelinovsky} and Kodama~\cite{Kodama}. 
For the development of the linear systems approach in this context and more background, see Blower~\cite{B-LS}, Blower and Doust~\cite{BD} and Blower and Newsham~\cite{BlowerNew}.
For more details on P\"oppe's approach, see Blower and Malham~\cite{BMal1,BMal2,BM-KP-descents}, as well as Doikou, Malham and Stylianidis~\cite{DMS}, Malham~\cite{MNLS,MKdV}
and Doikou, Malham, Stylianidis and Wiese~\cite{DMSWc}.

\indent The main conclusion is that Fredholm determinants and the Gelfand-Levitan equation are highly effective methods for numerical integration of $KP$.
\par


\section{P\"oppe's bracket for semi-additive operators}\label{S:semiadditive}

\indent Our method extends that of P\"oppe \cite{PKdV}, who realised that solutions of $KP$ are given by tau functions, namely Fredholm determinants of suitable families of integral operators. P\"oppe considers semi-additive operators, which one can define as families of integral operators $\Phi_{(t)}$ on $L^2((0, \infty ), {\mathbb{C}})$ that depend upon complex parameters $(t)=(t_j)_{j=1}^\infty$ so that $\Phi$ has kernel $\phi (x+t_1, y +t_1; t_2, t_3, \dots )$ and 
\begin{equation}\Phi_{(t)} f(x)=\int_0^\infty \phi (x+t_1, y+t_1, t_2, \dots ) f(y)\,dy \qquad (f\in L^2((0, \infty ); {\mathbb{C}} )).\end{equation}
\noindent In this section, we choose $t_1=t>0$, fix and suppress $t_2, t_3, \dots $ and regard $(x,y)$ as the primary variables.\par
\indent P\"oppe \cite{PKdV} used Fredholm determinants of semi-additive operators to solve the $KP$ equation and additive operators to solve the $KdV$ equation. There were steps towards a universal tau function theorem that would generate solutions to PDE from the tau functions of suitable integral operators. The $\tau$ function satisfies differential equations relating to (\ref{GLM}). In this section, we make the basic results precise.\par
\indent  P\"oppe's calculations involve certain operator formulas which reduce integral operators to finite-rank operators, and can conveniently be expressed in terminology from \cite[Section 5]{CuQu} or \cite[Lemma 3.6]{BK} used to describe the Fedosov product. The origins of this idea in geometric quantization date back to Lichn\'erowicz. Let $\calL$ be a unital complex algebra with ideal $\calF$; introduce the algebra $\calM$ with ideal $\calJ$ by
\[ \calJ =\Biggl\{ \begin{bmatrix} 0&f\\ 0&0\end{bmatrix} : f\in \calF\Biggr\}\lhd \calM =\Biggl\{ \begin{bmatrix} a&b\\ 0&a\end{bmatrix} : a,b\in\calL\Biggr\}\]
so that $\calJ^2=0$. Let $\pi : \calM\rightarrow\calM/\calJ$ be the canonical homomorphism. 

\begin{lem}\label{Delta} For a linear map $\partial :\calL\rightarrow\calL$ with $\partial (1)=0$, the following are equivalent:
\begin{enumerate}[(i)]
\item the linear map 
\begin{equation}\label{rho} \rho :\calL\rightarrow\calM: a\mapsto \begin{bmatrix} a&\partial a\\ 0&a\end{bmatrix} \qquad (a\in \calL )\end{equation}
is a homomorphism modulo $\calJ$; that is, $\pi\circ \rho: \calL\rightarrow \calM/\calJ$ is a homomorphism;
\item $\omega :\calL\otimes\calL\rightarrow\calM :$ $\omega (a,b)=\rho (ab)-\rho (a)\rho (b)$ takes values in $\calJ$;
\item $\varpi :\calL\otimes\calL\rightarrow \calL:$ $\varpi (a,b)=\partial (ab)-(\partial a)b-a\partial b$ takes values in $\calF$.
\end{enumerate}
\end{lem}
\begin{proof} Here $\rho$ is linear with $\rho (1)=0$ by the assumptions on $\partial$. The equivalence of (i) and (ii) is clear. Also
\[ \omega (a,b)=\begin{bmatrix} 0&\varpi (a,b)\\ 0&0\end{bmatrix}\]
so (ii) and (iii) are equivalent. Whereas $\partial: \calL\rightarrow\calL$ is not necessarily a derivation,
here $\calL/\calF$ is a $\calL$-bimodule and $\calL\rightarrow \calL/\calF :$ $a\mapsto \partial (a)+\calF$ is a derivation.\end{proof}

 Let $(\Omega^{ev}\calL , \circ)$ be the space of noncommutative differential forms on $\calL$ that have even order, which forms an algebra for Fedosov's product $a\circ b=ab-dadb$; for a detailed discussion, see \cite{CuQu}. The following statements should be self-explanatory, with indexing to match \cite[(15)]{CuQu2}.  

\begin{prop}\label{cocycles} Suppose that $\partial$ satisfies the conditions of Lemma \ref{Delta}, and let $\tau :\calF\rightarrow \Cb$ be a linear functional. \begin{enumerate} [(i)] \item 
Then there exists a trilinear map
\[ \varphi_2(a_0,a_1,a_2)=\varphi_2(a_0da_1da_2)=\tau  \bigl(a_0\varpi (a_1,a_2)\bigr)\qquad (a_0, a_1, a_2\in\calL).\]
\item Suppose further that $\partial$ takes values in $\calF$. Then $\pi\circ\rho =I$ on $\calL$ and there is a linear functional $\varphi_0 :\calL\rightarrow \Cb$ with $\varphi_0 (1)=0$, given by
\[\varphi_0(a_0)=\tau (\partial a_0)\qquad (a_0\in\calL ).\] 
\end{enumerate}\end{prop}
\begin{proof} (i) Given a linear map $\rho :\calL\rightarrow\calM$ with $\rho (1)=1$, there exists a unique homomorphism $\rho_*:\Omega^{ev}\calL\rightarrow \calM$ such that $\rho_*(a)=\rho(a)$ for all $a\in\calL$ by the universal property of $(\Omega^{ev}\calL, \circ )$; see \cite[Proposition 5.1]{CuQu2} and \cite[Proposition 2.1]{CuQu}. For $\rho$ as in (\ref{rho}), the matrix expression for $\rho (a_0)\omega (a_1,a_2)$ is zero except for the entry in the top-right corner $a_0\varpi (a_1,a_2)\in \calF$, to which we apply $\tau$.\par
(ii) When $\partial$ takes values in $\calF$, the conditions of Lemma \ref{Delta}(iii) are obviously satisfied, and we can define $\varphi_0(a_0)=\tau (\partial a_0).$
\end{proof}
 Proposition \ref{cocycles}(i) applies in Theorem \ref{sigmathrm}(iii). Clearly, $\partial :\calL\rightarrow\calL$ is a derivation if and only if $\rho$ is an algebra homomorphism, as in Theorem \ref{sigmathrm}(i). Also, $\partial  (\calL )\subset \calF $ if and only if $\pi\circ\rho =I$; this situation may be compared with Theorem \ref{sigmathrm}(iv). 

\begin{defn}\label{defpoppebracket}\begin{enumerate}[(i)]
\item Let $\calL^2(L^2((0,\infty );\Cb))$ be the space of Hilbert-Schmidt integral operators
on $L^2((0, \infty );\Cb )$, 
and $\calF$ the ideal of finite-rank operators on $L^2((0, \infty ); \Cb)$.\par
\item Suppose that $P\in \calL^2(L^2(0, \infty ); \Cb )$ has kernel $p(x,y)$. We define the P\"oppe bracket by $[P]_{x,y} =p(x,y)$ for $x,y\in (0, \infty )$.
\item Let $(S_t)_{t\geq 0}$ be the strongly continuous one-parameter semigroup of isometric shift operators on $L^2((0, \infty ); \Cb )$ given by $S_tf(x)=f(x-t) \Ib_{(0, \infty )}(x-t)$ for $x,t>0$ and $f\in L^2((0, \infty ); \Cb )$, otherwise known as the right-translation operators. 
\end{enumerate}
\end{defn}
 The following includes the fundamental identities of P\"oppe in terms of translation semigroups; see \cite[p. 622]{PKP}.\par

\begin{thm}\label{sigmathrm} \begin{enumerate}[(i)] 
\item Let $\sigma_t^\sharp (\Phi )=S_t\Phi S_t^\dagger.$ The $(\sigma^\sharp_t)_{t\geq 0}$ gives a strongly continuous one-parameter semigroup of isometric algebra homomorphisms on $\calL^2(L^2((0,\infty );\Cb))$;
the infinitesimal generator $\partial^\sharp$ of $(\sigma_t^\sharp )_{t\geq 0}$ is a derivation and its domain is an algebra.
\item Let $\sigma_t (\Phi )=S_t^\dagger \Phi S_t.$ Then $(\sigma_t)_{t\geq 0}$ gives a strongly continuous one-parameter semigroup of contractions on $\calL^2(L^2((0,\infty );\Cb))$ which preserves the upper-triangular form of the Gelfand-Levitan-Marchenko equation
\begin{equation}\label{GLM} \phi (x,y)+T(x,y) +\int_x^\infty T(x,z)\phi (z,y)\, dz=0\qquad (0<x<y).\end{equation}
\item Let $\partial $ be the infinitesimal generator of the semigroup $(\sigma_t)_{t\geq 0}$ and let $\Phi, \Psi$ and $\Phi\Psi$ belong to the domain of $\partial$. Then the cocycle 
\begin{equation} \label{seconddiff}\varpi (\Phi ,\Psi )=\partial (\Phi\Psi) -(\partial\Phi )\Psi -\Phi (\partial\Psi)\end{equation}
satisfies
\begin{equation}\bigl[\Upsilon\varpi (\Phi ,\Psi )\Lambda\bigr]_{x,y}=[\Upsilon\Phi ]_{x,0}[\Psi \Lambda]_{0,y}\qquad (\Upsilon, \Lambda\in \calL^2((L^2(0,\infty ); \Cb ))).\end{equation}
\item Let $G_n=\Gamma_{\phi_1}\dots \Gamma_{\phi_{2n}}$ be a product of an even number of Hankel operators. Then $\partial(G_n)$ has finite rank. 
\end{enumerate}
\end{thm}
\begin{proof} (i) The standard inner product on $\calL^2(L^2((0,\infty );\Cb))$ is $\langle \Phi ,\Psi \rangle =\trace (\Phi \Psi^\dagger )$, and one easily checks that $\langle \sigma_t(\Phi ), \Psi\rangle =\langle \Phi , \sigma_t^\sharp \Psi \rangle$. In Proposition 1 of \cite{BMal1}, we established that $(\sigma_t)_{t\geq 0}$ is a strongly continuous contraction semigroup, and the corresponding properties of $(\sigma_t^\sharp )_{t\geq 0}$ are obtained by passing to the adjoint semigroup and invoking \cite[Theorem 6.18]{EBD}. Each $\sigma_t^\sharp$ is an isometry since we have $S_t^\dagger S_t=I$, hence 
\[\langle\sigma_t^\sharp \Phi , \sigma_t^\sharp\Phi\rangle=\trace ( S_t\Phi S_t^\dagger (S_t\Phi S_t^\dagger )^\dagger) =\trace (S_t\Phi \Phi^\dagger S_t^\dagger )=\trace (S_t^\dagger S_t\Phi\Phi^\dagger )=\langle\Phi ,\Phi\rangle.\]    
Likewise, $\sigma_t^\sharp$ is an algebra homomorphism since 
\begin{equation}\label{homomorphism}\sigma_t^\sharp (\Phi\Psi)=S_t\Phi\Psi S_t^\dagger=S_t\Phi S_t^\dagger S_t\Psi S_t^\dagger =\sigma_t^\sharp (\Phi )\sigma_t^\sharp (\Psi ).\end{equation}
We introduce $\calA_0=\{ aI+ \Phi: a\in \Cb , \Phi\in\calL^2(L^2((0,\infty );\Cb))$, which is a Banach algebra
with norm $\Vert aI+\Phi\Vert =\vert a\vert+ \Vert\Phi\Vert_{\calL^2}$ on which $aI+\Phi \mapsto aI+S_t\Phi S_t^\dagger$ gives a strongly continuous semigroup, extending $\sigma_t^\sharp$. Then by differentiating (\ref{homomorphism}) at $t=0+$, we find $\partial^\sharp (\Phi\Psi )=(\partial^\sharp\Phi)\Psi +\Phi \partial^\sharp \Psi$ on the algebra 
\begin{equation}\label{diffsharp}\calD (\partial^\sharp )=\bigl\{ aI+\Phi: a\in\Cb ;\Phi, \partial^\sharp\Phi\in\calL^2(L^2((0,\infty );\Cb ))\bigr\}.\end{equation}  
\indent (ii) In Proposition 1 of \cite{BMal1}, we established the analytical properties of $(\sigma_t)_{t\geq 0}$. It is straightforward to show that if $\Phi$ is a Hilbert-Schmidt integral operator with kernel $\phi (z, \zeta )$, then $S_t^\dagger \Phi S_t$ is the integral operator that has kernel $\phi (z+t, \zeta +t)$. Also $\calL=\{ aI+\Phi: a\in \Cb , \Phi\in \calL^2\}$ is an algebra on which $\sigma_t$ operates, and $\calL$ contains the ideal $\calF$ of finite rank operators. The simple substitution $(x,y)\mapsto (x+t,y+t)$ preserves the form of equation (\ref{GLM}).\par

(iii) The Hochschild cocycle relation
\begin{equation} \Upsilon\varpi (\Phi ,\Psi )-\varpi (\Upsilon\Phi ,\Psi )+\varpi (\Upsilon, \Phi \Psi )-\varpi (\Upsilon ,\Phi )\Psi =0\end{equation}
is a direct consequence of the definition of $\varpi$.\par
\indent Let $\Phi, \Psi\in\calL^2(L^2((0,\infty );\Cb))$, and let $P=\Phi\Psi$ have kernel $p(x,y)$ so $\sigma_t(P)$ has kernel $p(x+t, y+t)$ where $p(x,x)$ determines an element of $L^1((0, \infty ); \Cb)$. Then by Lebesgue's differentiation theorem, $h^{-1}\int_0^h \vert p(x+t,x+t)-p(x,x)\vert dt\rightarrow 0$ as $h\rightarrow 0+$, for almost all $x\in (0, \infty )$.  Also, observe that $P$ determines a trace-class operator. We have $[P]_{x,x}=[\sigma_x(P)]_{0,0}$, hence 
$(d/dx)[P]_{x,x}=[\partial\sigma_x(P)]_{0,0}$. In Proposition \ref{propdiffringKP}, we provide another expression for this diagonal derivative. \par
\indent In contrast to the situation of (i) $\sigma_t$ is not an algebra homomorphism for $t>0$ since $S_tS_t^\dagger$ corresponds to multiplication by $\Ib_{[t, \infty )}$, and $\partial$ is not a derivation in $\calL^2$; however, $\partial $ is a derivation modulo the finite-rank operators. Indeed, the discrepancy $\varpi (\Phi ,\Psi )$ is given by the integral operator with kernel
\begin{align}\label{seconddiffkernel}\int_0^\infty \Bigl({\frac{\partial\phi}{\partial x}}(x,z)\psi (z,y)& +\phi (x,z){\frac{\partial \psi }{\partial y}} (z,y)\Bigr) dz-\int_0^\infty \Bigl({\frac{\partial\phi}{\partial x}}(x,z)+{\frac{\partial \phi }{\partial z}}(x,z)\Bigr)\psi (z,y)\, dz\nonumber\\
&-\int_0^\infty \phi (x,z)\Bigl({\frac{\partial\psi}{\partial z}}(z,y)\psi (z,y) +{\frac{\partial \psi }{\partial y}} (z,y)\Bigr) dz\nonumber\\
&=-\int_0^\infty \Bigl( {\frac{\partial \phi }{\partial z}}(x,z)\psi (z,y)+\phi (x,z){\frac{\partial\psi}{\partial z}}(z,y)\psi (z,y)\Bigr) dz\nonumber\\
&=\phi (x,0)\psi (0,y).\end{align}
The result follows when we apply $\Upsilon $ on the left and $\Lambda$ on the right. Compare \cite[Lemma 1]{BMal1} and \cite[\S 5]{PKP}. 

(iv) The proof is by induction on $n$. For $n=1$, we observe that $G_1=\Gamma_{\phi_1}\Gamma_{\phi_2}$ gives $\partial G_1$ which has kernel $-\phi_1(x)\phi_2(x)$. Then $G_{n+1}=G_n(\Gamma_{\phi_{2n-1}}\Gamma_{\phi_{2n}})$ gives 
\begin{align} \partial G_{n+1}=&\big( \partial (G_n(\Gamma_{\phi_{2n-1}}\Gamma_{\phi_{2n}}))-
(\partial G_n)(\Gamma_{\phi_{2n-1}}\Gamma_{\phi_{2n}})-G_n\partial (\Gamma_{\phi_{2n-1}}\Gamma_{\phi_{2n}})\nonumber\\
&\quad +\partial( G_n)(\Gamma_{\phi_{2n-1}}\Gamma_{\phi_{2n}})+G_n\partial (\Gamma_{\phi_{2n-1}}\Gamma_{\phi_{2n}})\end{align}
where the first term on the right-hand side is in $\calF$ by (\ref{seconddiff}) and (\ref{seconddiffkernel}), while the final two terms are in $\calF$ by the induction hypothesis. See also \cite[Section 3.5]{McK}.

\end{proof}
\begin{ex} (i) In \cite[Proposition 1(vi)]{BMal1}, we identified $\varphi_0(\Phi )=\trace (\partial\Phi )=-[\Phi ]_{0,0}$.\par
(ii) In the context of Theorem \ref{sigmathrm}(i) and (iv), let $\Phi\in \calL^2((0, \infty ); \Cb))$ have kernel $\phi (x,y)$, and let $\Psi=\Gamma_{\psi_1}\Gamma_{\psi_2}$ be a product of Hankel operators, such that $\Phi\in \calD(\partial^\sharp )$. Then the usual trace formula on $\calL^2((0, \infty ); \Cb)$ gives
\[ \trace\bigl( \Phi\partial^\sharp \Psi \bigr)=\int_0^\infty\int_0^\infty \phi (x,y)\psi_1(y)\psi_2(x)\, dxdy.\]
\end{ex}

\begin{rem} The natural development of Theorem \ref{sigmathrm}(i) follows the route of \cite[Section 12]{CuQu2}. Let $\calD$ be a complex unital algebra, let $\calM$ be an $\calD$-bimodule, and $V$ a complex vector space. Let the commutator subspace of $\calM$ be $[M, \calD] =\spann\{ am-ma: a\in\calD, m\in\calM\}$. A trace $\tau: \calM\rightarrow V$ is a linear map $\tau:\calM\rightarrow V$ such that $\tau\vert [\calM, \calD]=0$. Given a derivation $\partial^\sharp: \calD\rightarrow \calM$, there exists an $\calD$-bimodule map $\Omega^1\calD\rightarrow \calM$ such that $\Phi d\Psi\mapsto \Phi \partial^\sharp\Psi$; thus a trace $\tau:\calM\rightarrow V$ gives a trace $\varphi :\Omega^1\calA\rightarrow V:$ $\varphi (\Phi d\Psi)=\tau (\Phi\partial^\sharp \Psi).$\par
\indent Next we form matrices with entries in these spaces and extend the algebraic structures in the natural way. We can amplify $\varphi$ to $\varphi_n:M_{n\times n}(\Omega^1\calD )\rightarrow V$ by $\varphi_n(A\otimes \Phi d\Psi)=\trace (A)\varphi (\Phi d\Psi )$, where $\trace :M_{n\times n}(\Cb)\rightarrow \Cb$ is the usual trace. Then we let $GL_n(\calD)$ be the multiplicative group consisting of invertible elements of $M_{n\times n}(\calD )$. Writing $\{ A,B\}=ABA^{-1}B^{-1}$ for the multiplicative commutator, we can introduce the normal subgroup $\{ GL_n(\calD), GL_n(\calD)\}$ that is generated by the multiplicative commutators. Using the trace property of $\tau$ repeatedly, we find
\[ \varphi_n \bigl((\Phi\Psi)^{-1}d(\Phi \Psi )\bigr)=\varphi_n (\Phi^{-1}d\Phi)+\varphi_n (\Psi^{-1}d\Psi ),\]
and 
\[ \varphi_n \bigl(\{ \Phi , \Psi \}^{-1}d\{\Phi , \Psi \}\bigr)= 0\qquad (\Phi, \Psi\in GL_n(\calD)).\]
Hence $\Phi\mapsto  \varphi (\Phi^{-1}d\Phi )$ induces a group homomorphism $GL_n(\calD) /\{ GL_n(\calD ), GL_n(\calD )\}\rightarrow V$. There are numerous closely related determinant and trace formulas involving this idea, as discussed in \cite{BW}, \par

\end{rem}

\begin{defn}\label{semiadditivetau}\begin{enumerate}[(i)]
\item A semi-additive kernel is the family of kernels from the orbit $(\sigma_t(\Phi ) )_{t\geq 0}$ of some $\Phi\in \calL^2(L^2((0, \infty ); \Cb))$.
\item
The tau function of a trace-class kernel $P$ is $\tau (x)=\det (I+\sigma_x(P)).$
\end{enumerate}
\end{defn}

Suppose that $\psi \in L^2 ((0, \infty ); H')$ and $\xi \in L^2 ((0, \infty );H)$, and $\phi (z, \zeta )=\psi (z)\xi (\zeta )$. Then with $R_x=\int_x^\infty \xi (z)\psi (z)dz$, we have a family of bounded linear operators $I+R_x$ $(x>0)$ on $H$ which are invertible for all $x>x_0$ for some $x_0>0$. Then we can introduce a kernel
\[ T(x,y)=-\psi (x)\bigl( I+R_x)^{-1}\xi (y) \qquad (x_0<x<y).\]

\begin{lem}\label{taudiff} Then (\ref{GLM}) holds and 
\[ T(x,x)={\frac{d}{dx}}\log \tau (x)\qquad (x>x_0).\]
\end{lem}
\begin{proof} The Gelfand-Levitan-Marchenko equation follows by a direct substitution. Then 
\begin{align} {\frac{d}{dx}}\log\det (I+R_x)&={\hbox{trace}}\Bigl( (I+R_x)^{-1}{\frac{dR_x}{dx}}\Bigr)\nonumber\\
&=-\trace\Bigl( (I+R_x)^{-1}\xi (x)\psi (x)\Bigr)\nonumber\\
&=-\psi (x)(I+R_x)^{-1}\xi (x)\nonumber\\
&=T(x,x).\end{align}
\end{proof}
 In previous papers \cite{Blower1}, \cite{BlowerNew} , we have developed P\"oppe's approach, starting from linear systems and forming differential rings. 
In Definition \ref{diffring} and Proposition \ref{defdiffringKP}, we introduce another bracket operation which facilitates calculation of $[K]_{x,x}$ and its derivatives, so we can compute $\tau$.
Our results on differential rings are equivalent in some cases to P\"oppe's results on semi-additive operators, although we regard our approach as more natural, more closely aligned with algebraic formalism that is used in differential Galois theory and elsewhere, and less dependent on isolated ingenious identities. \par

By \cite[Lemma 5.1]{Blower1}, the solution $T(x,y)$ to (\ref{GLM}) satisfies 
\begin{equation}\label{GLMPDE}{\frac{\partial^2T}{\partial y^2}}-{\frac{\partial^2 T}{\partial x^2}}=2\Bigl({\frac{dT}{dx}}(x,x)\Bigr) T(x,y)\qquad (0<x<y)\end{equation}
which is a wave equation, albeit with a potential $-2{\frac{d}{dx}}T(x,x)$ that is typically a matrix, not necessarily self-adjoint. In the next section \ref{S:spectra}, we develop a functional calculus for this context.

\section{Spectra of linear systems}\label{S:spectra}

Let $q\in C_b([0, \infty );\Rb )$ with $\int_0^\infty x\vert q(x)\vert dx<\infty$ and consider the differential operator $L=-{\frac{d^2}{dx^2}}+q$, 
which is self-adjoint for suitable boundary conditions. Hence for $f\in L^2(0, \infty );\Cb )$, we can define 
$\cos (t\sqrt{L})$ so that $u(t,x)=\cos (t\sqrt{L})f(x)$ satisfies the wave equation 
\begin{align} {\frac{\partial^2}{\partial t^2}}u(x,t)&-{\frac{\partial^2}{\partial x^2}}u(t,x)+q(x)u(t,x)=0;\nonumber
\\u(0,x)&=f(x);\nonumber\\
{\frac{\partial u(0,x)}{\partial t}}&=0.\end{align}
Gelfand and Levitan used this as the foundation of their spectral theory of second-order differential operators. In the current paper, we make a modest extension of their theory to deal with potential $q$ that are not necessarily real, so $L$ is not necessarily self-adjoint, but the wave equation is still useful. In Example \ref{ZSex} we consider examples of differential equations which in Section \ref{S:ZS} we analyze in terms of linear systems. \par
\indent We review some notions of spectral theory.
\begin{defn}\label{spectraldefn}
For a closed operator $\Delta$ with dense domain $\calD (\Delta )$ in Hilbert space $H$, let $\rho (\Delta )$ be the resolvent $\rho (\Delta )= \{ \lambda \in \Cb: \exists (\lambda I-\Delta )^{-1}\in \calL (H)\}$ and let the spectrum be $\sigma (\Delta )=\Cb\setminus \rho (\Delta )$, a change to the notation from section \ref{S:semiadditive}. 
The spectral bound $s(\Delta )=\sup\{ \Re \lambda : \lambda\in \sigma (\Delta )\}$. Then the approximate point spectrum is 
\[\sigma_{ap} (\Delta )= \{ \lambda \in \Cb : (\lambda I-\Delta )\calD (\Delta )\, \hbox{not closed}\}\cup
\{\lambda \in \Cb : \lambda I-\Delta\, {\hbox{ not injective}}\},
\] 
where the final set gives the point spectrum, namely the set of eigenvalues.\end{defn} 

Consider the operator
\begin{equation}\label{Deltaop} \Delta=\begin{bmatrix} 0&I\\ -L_1&0\end{bmatrix}\qquad \begin{matrix} H^1\\ L^2\end{matrix}\end{equation}
and suppose that $\Delta$ generates a strongly continuous semigroup $e^{t\Delta}$ such that $\Vert e^{t\Delta}\Vert \leq Me^{\omega_0t}$ for all $t\geq 0$. Then $s(\Delta )\leq \omega_0$. The topological boundary of $\sigma (\Delta )$ and the approximate point spectrum of $\Delta$ are related by $\partial \sigma (\Delta )\subseteq \sigma_{ap} (\Delta )$, 
and $e^{t\sigma_{ap}(\Delta )}\subseteq \sigma_{ap} (e^{t\Delta })$ for all $t\geq 0$. We defer discussion of the point spectrum until Proposition \ref{pointspecprop}. \par
\indent Let $\calR_0$ the subalgebra of $\calL (H)$ that generated by the set of $r(\Delta )$, where $r$ is a complex proper rational function which is holomorphic on $\sigma (\Delta )$, and let $\calR$ be the norm closure of $\calR_0$ in $\calL (H)$. Given a complex and commutative Banach algebra $\calA$ and a homomorphism $\theta:\calA\rightarrow \calR$, there is an induced map $\spec m(\calR)\rightarrow \spec m(\cal A)$ between the corresponding maximal ideal spaces given by $\varphi\mapsto \varphi\circ \theta$ for every multiplicative linear functional $\varphi :\calR\rightarrow\Cb$ that corresponds to a maximal ideal of $\calR$. A useful choice of $\calA$ is obtained from cosine families. 
\begin{defn}\label{defcosine} A cosine family on $L^2$ is a family $(\Cos (t))_{t\in \Rb}\subset \calL (L^2)$ such that\begin{enumerate}[(i)] 
\item $t\mapsto \Cos (t)f$ is continuous $\Rb \rightarrow L^2$ for all $f\in L^2$;
\item $\Cos (s+t)+\Cos (s-t)=2\Cos (s)\Cos (t)$ for all $s,t\in \Rb$; 
\item $\Cos (0)=I$, and the generator is $-{\frac{d^2}{dt^2}}\Cos (t)$; see \cite{Gold}.
\end{enumerate}
\end{defn}

We have two cases to consider for $\Delta$ and $\omega_0$ as in (\ref{Deltaop}). First, suppose $\omega_0=0$, so the cosine family $(\cos (t\sqrt{L_1})$ is uniformly bounded on $L^2$. Then there exist an invertible $U\in {\calL}(L^2, H)$ and a self-adjoint and non-negative $K\in \calL(H)$ such that $\cos (t\sqrt{L_1})=U\cos (t\sqrt {K})U^{-1}$, and in particular 
$\sigma (L)=\sigma (K)\subseteq [0, \infty )$.\par
\indent Now consider $\omega_0>0$. For $\omega>\omega_0>0$, we introduce the horizontal strip $\calS_{\omega}=\{ \lambda\in \Cb :\vert \Im \lambda \vert < \omega\}$ with closure $cl(\calS_\omega )$.

\begin{lem}\label{lemcosine} For $\omega>\omega_0$, let $\calA_\omega$ be the 
algebra of continuous functions $f: cl( \calS_\omega )\rightarrow \Cb$ such that $f$ is holomorphic on $\calS_\omega$, such that $f(z)=f(-z)$ and $f(z)=O(1/z^2)$ as $z\rightarrow\infty$ for $z\in cl(\calS_{\omega})$.\begin{enumerate} [(i)] \item 
Then $\calA_\omega$ is a Banach algebra for the norm $\Vert f\Vert_{(\omega )}=\sup_{s\in cl(\calS_\omega)} (1+\vert z\vert^2)\vert f(z)\vert$.
\item There is a bounded homomorphism $\calA_\omega\rightarrow\calL (L^2)$ defined by
\begin{equation}\label{froot} f(\sqrt{L_1})=\int_{-\infty}^\infty \hat f(k)\cos (k\sqrt{L_1})\, {\frac{dk}{2\pi}}.\end{equation}
\item For $\zeta^2>\omega_0$, the operator $\zeta^2I+L_1$ is invertible with 
\[ \zeta (\zeta^2I+L_1)^{-1}\Psi_0(x)=\int_0^\infty e^{-\zeta t}\cos (t\sqrt{L_1})\Psi_0(x)\, dt\qquad (\zeta>\sqrt {\omega_0}).\]
\end{enumerate}
\end{lem}
\begin{proof} (i) One checks that $\Vert fg\Vert_{(\omega )}\leq \Vert f\Vert_{(\omega )}\Vert g\Vert_{(\omega )}$. Completeness of $\calA_\omega$ follows from Morera's theorem.\par
(ii) Here we have $\hat f(k)=\hat f(-k)$. Using Cauchy's theorem, one can show that 
\[ \hat f(k)=\int_{-\infty}^\infty f(x)e^{-ikx}\, dx=e^{-k\omega} \int_{-\infty}^\infty f(x-i\omega )e^{-ikx}\, dx\]
so there exists $M_0>0$ such that $\vert\hat f(k)\vert\leq M_0e^{-\omega k}$ for $k>0$, so the integral (\ref{froot}) converges. Then one checks that the map $f\mapsto f(\sqrt{L_1})$ is multiplicative. \par
(iii) The function $f(z)=\zeta /(\zeta^2+z^2)$ is an element of $\calA_\omega$.\par
\end{proof}
\begin{rem} We are led to consider functions that are holomorphic on a strip due to the following example from \cite[page 239]{BloH}. Consider 
\[ \varphi_\lambda (x)={\frac{\cos \lambda x}{\cosh x}}\qquad (x\in ([0, \infty ); \lambda \in \calS_1)\]
so that $\lambda\mapsto \varphi_\lambda (x)$ is holomorphic in $\calS_1$, and $\vert\varphi_\lambda (x)\vert\leq 1.$ Given $f\in L^1( (0, \infty ); \cosh^2x\, dx; \Cb )$, we introduce
\[ \hat f(\lambda )=\int_0^\infty f(x)\varphi_\lambda (x)\cosh^2x\, dx\qquad (\lambda\in \calS_1),\]
and observe that $\hat f(\lambda )$ is holomorphic on $\calS_1$. With 
\[ Mf(x)=-{\frac{d^2f}{dx^2}}-2\tanh x{\frac{df}{dx}}-f(x)\]
we find that $M=M^\dagger$ in $L^2 ((0,\infty ); \cosh^2x\, dx; \Cb )$ with $M\geq 0$ and $M\varphi_\lambda =\lambda^2\varphi_\lambda$ for all $\lambda\in \calS_1$. There is a cosine family $(\cos (t\sqrt{M}))_{t\in\Rb}$, and the Kunze-Stein phenomenon \cite[[p. 124]{BloH}
applies to the operators 
\[ T_Mf=\int_0^\infty f(t)\cos (t\sqrt{M})\, dt.\]
For $1\leq p<2$ and 
$f\in L^p([0, \infty ); \cosh^2 x\, dx; \Cb )$, the function $\hat f(\lambda )$ is holomorphic on $\calS_{(2-p)/p}$.
\end{rem} 
Let $L_0$ be the self-adjoint differential operator which is densely defined in $H$ with $L_0\geq 0$.
Then we consider the Cauchy problem for the symmetric hyperbolic system
\begin{equation} \begin{bmatrix} 0&-I\\ I&0\end{bmatrix} {\frac{d}{dt}} \begin{bmatrix} \psi \\ \phi \end{bmatrix} = \begin{bmatrix} L_0&0\\ 0&I\end{bmatrix} \begin{bmatrix} \psi \\ \phi \end{bmatrix}\end{equation}
with initial condition
\begin{equation} \begin{bmatrix} \psi \\ \phi \end{bmatrix}_{t=0}=\begin{bmatrix} f\\ g\end{bmatrix}\qquad (g\in L^2((0, \infty ); \mathbb{C}), f\in H^1((0, \infty ); \mathbb{C}))\end{equation}
Then there exists a unique solution
\begin{equation} \begin{bmatrix} \psi \\ \phi \end{bmatrix}= \begin{bmatrix} \cos t\sqrt{L_0}& {\frac{\sin t\sqrt{L_0}}{\sqrt{L_0}}}\\ 
-\sqrt{L_0}\sin t\sqrt{L_0}& \cos t\sqrt{L_0}\end{bmatrix}
 \begin{bmatrix} f\\ g\end{bmatrix}.\end{equation}

To determine the eigenvalues of $I-{\frac{d^2}{dx^2}}+U^\dagger$, it is convenient to work with its inverse operator. In Proposition \ref{pointspecprop} we obtain a criterion for eigenvalues involving Pincus's principal function \cite{CP}, which requires the following Lemma.  
\begin{lem}\label{lemPincus} Suppose that $U\in L^\infty (\Rb; M_{n_0\times n_0}(\Cb ))$ has $\Vert U\Vert_\infty <1$
and $U^\dagger -U=V_1V_2$ where 
\[V_1,V_2\in C_b^2(\Rb ;M_{n_0\times n_0}(\Cb ))\cap L^2(\Rb ;M_{n_0\times n_0}(\Cb )).\]
Let $G_U=(I-{\frac{d^2}{dx^2}}+U)^{-1}$. Then $G_U$ is a bounded linear operator which is almost normal in the sense that the additive commutator with its adjoint satisfies
\[ [G_U^\dagger , G_U]\in \calL^1 (L^2(\Rb ;M_{n_0\times n_0}(\Cb ))).\]
\end{lem}
\begin{proof} By Fourier analysis, one finds that the operator $-{\frac{d^2}{dx^2}}+I$ is invertible with inverse 
$G_0=(-{\frac{d^2}{dx^2}}+I)^{-1}$ with integral kernel $2^{-1}e^{-\vert x-y\vert}$ on $L^2(\Rb ;\Cb)$, so $\Vert G_0\Vert\leq 1$. Hence $I-{\frac{d^2}{dx^2}}+V_1V_2$ is invertible with inverse 
\[ G_U=\Bigl(I-{\frac{d^2}{dx^2}}+U\Bigr)^{-1}=G_0(I+UG_0)^{-1};\]
and $G_U^\dagger =G_{U^\dagger}$.
Then we have
\[\bigl[G_U^\dagger ,G_U\bigr]=G_U^\dagger G_U\Bigl[ {\frac{d^2}{dx^2}}, U^\dagger -U\Bigr]G_UG_U^\dagger ,\]
in which 
\[ \Bigl[ {\frac{d^2}{dx^2}}, V_1V_2\Bigr]={\frac{d^2V_1}{dx^2}}V_2+2{\frac{dV_1}{dx}}{\frac{dV_2}{dx}}+V_1{\frac{d^2V_2}{dx^2}} +2{\frac{dV_1}{dx}} V_2{\frac{d}{dx}}+2V_1{\frac{dV_2}{dx}}{\frac{d}{dx}}.\]
By considering their kernels as integral operators, we find that the operators of the form $G_0{\frac{d^2V_1}{dx^2}}$ and $V_2G_0$ are Hilbert-Schmidt, while $G_0d/dx$ is bounded. Hence $[G_U^\dagger , G_U]$ is a sum of products of Hilbert-Schmidt operators, hence is of trace class.
\end{proof}

\indent From Lemma \ref{lemPincus}, we deduce that $G_U=X+iY$ where $X,Y$ are bounded and self-adjoint operators in $L^2$ such that $2i[X,Y]=[G_U^\dagger ,G_U]$ is trace class; one says that $(X,Y)$ are an almost commuting pair of self-adjoint operators. With nonzero  $\lambda\in \rho (X)$ and $\ell \in \rho (Y)$, we have $W=\log (I-X/\lambda )$ and $Z=\log (I-Y/\ell )$ such that $[W,Z]\in \calL^1(H)$; then $\det (e^We^Ze^{-W}e^{-Z})=\exp(\trace [W,Z])$.  
Pincus has shown that there exists a compactly supported $g\in L^1(\Rb\times \Rb ; \Rb )$ such that 
\[ \det \Bigl( (Y-\ell I)(X-\lambda I)(Y-\ell I)^{-1}(X-\lambda I)^{-1}\Bigr) =\exp\Bigl( {\frac{1}{2\pi i}}\iint_{\Rb\times\Rb}g(y,x) {\frac{dydx}{(y-\ell )(x-\lambda )}}\Bigr).\]

We introduce the numerical range of $G_U^\dagger$ as
\begin{align}\mathcal{W}(G_U^\dagger )&=\{ \langle G_U^\dagger f,f\rangle : f\in H, \Vert f\Vert =1\}\nonumber\\
&=\Biggl\{ {\frac{\int_{-\infty}^\infty ( \Vert f(x)\Vert^2+\Vert f'(x)\Vert^2+\langle f(x), U^\dagger (x)f(x)\rangle) dx}
{\int_{-\infty}^\infty \Vert f(x)-f''(x)+U^\dagger (x)f(x)\Vert^2dx}}:\nonumber\\
&\qquad f\in \calD (I+\Delta +U^\dagger ); f\neq 0\Biggr\}.\end{align}
Since $U-U^\dagger$ is typically non-zero and skew, there is no reason to presume that $\mathcal{W}(G_U^\dagger )$ is contained in $\Rb$.

\begin{prop}\label{pointspecprop}\begin{enumerate}[(i)] 
\item The spectrum of $G_U^\dagger$ satisfies $[0,1]\subseteq \sigma (G_U^\dagger )\subseteq \mathcal{W}(G_U^\dagger )$;
\item any $\lambda\in \sigma (G_U^\dagger )\setminus [0,1]$ is an eigenvalue of finite multiplicity.
\item $\lambda $ is an eigenvalue of $G_U^\dagger$ if and only if
\[ \iint_{\Rb\times \Rb} {\frac{1-g(y,x)}{\vert x+iy-\lambda \vert^2}} dydx<\infty .\]
\end{enumerate}
\end{prop}
\begin{proof} (i) See \cite{BonDun}. As in the Lemma \ref{lemPincus}, we have $G_U-G_0\in \calL^2(H)$, so $G_U$ is a Hilbert-Schmidt perturbation of the self-adjoint operator $G_0$, hence their Weyl spectra are equal 
\begin{align}\sigma_{W}(G_U)&=\cap \{ \sigma (G_U+K): K\in \calK (H)\}\nonumber\\
&=\sigma (G_0)=[0,1].\end{align}
One checks that $\mathcal{W}(G_U^\dagger )$ is a compact nonempty set and $\Vert (\lambda I-G_U^\dagger )f\Vert \geq {\hbox{dist}} (\lambda , \mathcal{W}(G_U^\dagger ))$ for all $f\in H$ with $\Vert f\vert =1$, so the spectrum is contained in the numerical range; see \cite{BonDun}.\par 
(ii) Weyl showed that any  $\lambda\in \sigma (G_U)\setminus \sigma_W(G_U)$ gives an eigenvalue of finite multiplicity.\par
(iii) This is Carey and Pincus's criterion for eigenvalues; see \cite{CP}.
\end{proof}

Let $\calR_0$ be the subalgebra of $\calL (H)$ that is generated by the set of $r(\Delta^\dagger)$, where $\Delta^\dagger =I-{\frac{d^2}{dx^2}}+U^\dagger$ and $r$ is a complex and proper rational function which is holomorphic on $\sigma (\Delta )$; then let $\calR$ be the norm closure of $\calR_0$ in $\calL (H)$.
Also, if $\lambda $ is a non-zero eigenvalue of $G_U^\dagger$, then $\lambda^{-1}$ is an eigenvalue of $\Delta^\dagger$, so $\lambda^{-1}\in \spec m(\calR )$.

\begin{ex}\label{ZSex} (i) Consider the system
\begin{equation}\label{KP1} {\frac{d}{dx}}\Psi (x;k)=\begin{bmatrix} -ik&q(x)\\-\bar q(x)&ik\end{bmatrix} \Psi (x;k)\end{equation}
which gives the second-order matrix system
\begin{equation}\label{KP2} -{\frac{d^2}{dx^2}}\Psi (x;k) +\begin{bmatrix} -\vert q(x)\vert^2 &q'(x)\\ -\bar q'(x)&-\vert q(x)\vert^2\end{bmatrix}\Psi (x;k)=k^2\Psi (x;k);\end{equation}
Let $U$ be the matrix exhibited in (\ref{KP2}). Then $U-U^\dagger$ is a skew matrix given by the off-diagonal terms. One can select $q'$ to satisfy the hypotheses of Lemma \ref{lemPincus}. See \cite{ZS}.
(ii) Consider
\[ J{\frac{d\Psi}{dx}}(x;\lambda )=\Omega (x; \lambda )\Psi (x;\lambda )\qquad J=\begin{bmatrix}0&-1\\ 1&0\end{bmatrix} \]
where $\Psi (x,\lambda )$ is a real $2\times 2$ matrix such that $\Omega (x; \lambda )=\Omega (x;\lambda )^\top$ and $\trace \Omega (x,\lambda )$ is constant for fixed $\lambda$. Then one can check that 
\[ U=J{\frac{d\Omega}{dx}}-J\Omega J\Omega\]
is a real symmetric matrix.
\end{ex}
Let $\phi : (0, \infty )\rightarrow M_{n\times n}(\Cb )$. We say that $\phi$ is of Floquet type if there exist $\zeta \in \Cb$ with $\Re\zeta >0$ and $F: \Rb \rightarrow M_{n\times n}(\Cb ) $ which is once continuously differentiable and $2\pi$-periodic such that $\phi (t)=e^{-\zeta t}F(t)$.
\begin{lem} Let $\phi$ be a Floquet function. Then there exists a linear system $(-A,B,C)$ such that $\phi$ is impulse response function and $I+R_t$ gives a convergent determinant.\end{lem}

\begin{proof} An infinite determinant converges provided that the product of the diagonal entries converges absolutely and the sum of the off-diagonal entries converges absolutely; see \cite[2.81]{WW}. For notational simplicity, we suppose that $F$ is scalar-valued with Fourier expansion
$F(x)=\sum_{n=-\infty}^\infty a_ne^{inx}$. We choose $3/2<r<2$ and $q=4/r$ and $\beta =(3-r)/2$ such that 
\[\sum_n \vert a_n\vert^{r/2}\leq \Bigl(\sum_n \vert a_n\vert^2(1+\vert n\vert)^{\beta q}\Bigr)^{1/q}\Bigl(\sum_n(1+\vert n\vert )^{-\beta q/(q-1)}\Bigr)^{(q-1)/q)}\]
converges. Let $H=\ell^2(\Zb ;\Cb )$ and consider $\calD (A)=\{ (\xi_m)_{m=-\infty} ^\infty \in H : (m\xi_m)\in H\}$ 
\begin{align}\label{linsysFloquet} A: \calD (A)\rightarrow H:& \quad (\xi_m)\mapsto ((\zeta -im)\xi_m)\nonumber\\
B: \Cb \mapsto H:&\quad b\mapsto (\vert a_m\vert^{1/2} b)_{m=-\infty}^\infty\nonumber\\
C: H\rightarrow \Cb : &\quad (\xi_m)\mapsto \sum_m {\frac{a_m\xi_m}{\sqrt{\vert a_m\vert}}}.\end{align}
Then we have $\phi (t)=Ce^{-tA}B=\sum_m a_m e^{-t(\zeta -im)}=e^{-\zeta t}F(t)$, and we can represent $R_t$ as the following  matrix with respect to the standard basis of $\ell^2(\Zb; \Cb )$:
\begin{align}\label{Hilldet} R_t&=\int_t^\infty e^{-uA}BCe^{-uA}\, du\nonumber\\
&=\Biggl[\int_t^\infty e^{-(2\lambda -im-in)u} {\frac{ a_m\sqrt{\vert a_n\vert}}{\sqrt{\vert a_m\vert}}}du \Biggr]_{n,m=-\infty}^\infty\nonumber\\
&=\Biggl[ {\frac{e^{-(2\lambda -im-in)t}}{2\zeta -im-in} }  {\frac{a_m\sqrt{\vert a_n\vert}}{\sqrt{\vert a_m\vert }}}\Biggr]_{m,n=-\infty}^\infty \end{align}
By Young's convolution inequality, we have
\begin{equation}\label{Youngconvo} \sum_{m,n} {\frac{\sqrt{\vert a_n\vert}\sqrt{\vert a_m\vert}} {\vert 2\zeta -im-in\vert}}\leq \bigl\Vert (\sqrt{\vert a_n\vert})\bigr\Vert^2_{\ell^r}\Bigl\Vert \Bigl({\frac{1}{\vert 2\zeta -im\vert}}\Bigr) \Bigr\Vert_{\ell^p} \end{equation}
where $1/p=2-2/r<1$. Hence the matrix that represents $R_t$ has absolutely summable entries, as stated.
\end{proof}
\begin{ex} Let $U_1\in M_{n\times n}(\Cb )$ and let $U_0:\Rb \rightarrow M_{n\times n}(\Cb )$ be continuous and $2\pi$ periodic. Then the differential equation
\begin{equation}\label{secondperiodicODE} {\frac{d^2\Phi }{dt^2}}=(\lambda U_1+U_0(t))\Phi (t)\end{equation}
gives rise to a first-order periodic system
\begin{equation}\label{firstperiodicODE} {\frac{d\Psi}{dt}}=\begin{bmatrix} 0&I_n\\ \lambda U_1+U_0(t)&0\end{bmatrix}\Psi\end{equation}
where $\Psi (t;\lambda )\in M_{2n\times 2n}(\Cb )$ is holomorphic in $\lambda$. Then $\det \Psi (t; \lambda )$ is constant in $t$, and for $\Psi (0;\lambda )\neq 0$, we have a $2n$-sheeted cover of $\Cb$ given by 
\[ \bigl\{ (\lambda , \rho ): \det \bigl( \rho \Psi (0; \lambda )-\Psi (2\pi ; \lambda )\bigr) =0\bigr\}.\]
Suppose that $\lambda$ is not a branch point, so that there are truly $(2n)$ distinct choices of $\rho$. 
Then we write the modular matrix as $M(\lambda )= \Psi (0; \lambda)^{-1}\Psi (2\pi ;\lambda )$.  
\end{ex}
\begin{prop}\label{Floquetprop} \begin{enumerate}[(i)]\item Suppose that $M(\lambda )$ for has distinct eigenvalues, not all of them unimodular, for some $\lambda\in \Cb$. Then there exists $(-A,B,C)$ such that $I+R_t$ has a convergent determinant and impulse response function $\phi$.
\item Suppose that $U_0(x)=\sum_{k=-\ell}^\ell A_ke^{ikx}$ is a trigonometric polynomial such that $A_\ell$ and $A_{-\ell}$ are invertible, and $(a_n)_{n=-\ell}^\ell$ are given. Then the full sequence $(a_n)_{n=-\infty}^\infty$ is determined by a recurrence relation.\end{enumerate}\end{prop}

\begin{proof} (i) We introduce an invertible $V\in M_{2n\times 2n}(\Cb )$ and $D=\diag (\rho_1, \dots , \rho_{2n})$ such that $M=VDV^{-1}$. The product of the eigenvalues is $\prod_{j=1}^{2n}\rho_j=\det M(\lambda )=1$, so either $\vert\rho_m\vert=1$ for all $m$, or there exists $j$ such that $\vert\rho_j\vert<1$, as in the hypothesis.  
With $\zeta_j=-(2\pi )^{-1}\log\rho_j$ and $D(x)=\diag ( e^{-\zeta_jx})$, we have a $2\pi$-periodic $G(x)=\Psi (x)VD(-x)$. In particular, let $E_j$ be the diagonal matrix that has $1$ in place $j$ and zeros elsewhere; then let $C_j=VE_jV^{-1}$ so that $\Psi_j(x)=\Psi (x)C_j$ gives a solution of the ODE with $\Psi_j(x)=e^{-\zeta_jx}G(x)E_jV^{-1}$. Then we write the $(2n)\times (2n)$ matrix in terms of four $n\times n$ blocks
\[ \Phi_j(x)=\begin{bmatrix} \Phi_j(x) & \Xi_j(x) \\ \Phi'_j(x)& \Xi'_j(x)\end{bmatrix}\]
where $\Phi_j(x)= e^{-\zeta_jx}F_j(x)$ has $F_j:\Rb\rightarrow M_{n\times n}(\Cb )$ is continuous and $2\pi$-periodic and
${\frac{d^2\Phi_j}{dt^2}}=(\lambda U_1+U_0(t))\Phi_j(t)$. For $\vert\rho_j\vert<1$, we obtain a Floquet solution.\par
(ii) We consider the recurrence relation
\begin{equation}\label{recurrencerelation} \zeta_j^2a_n-2in\zeta_ja_n-n^2a_n-\lambda a_n-\sum_{k=-\ell}^\ell A_ka_{n-k}=0,\end{equation}
which we can solve for $a_{n+\ell}$ since $A_{-\ell}$ is invertible, giving , with
\begin{equation} B_\ell (n)=-A_{-\ell}^{-1}\bigl(\zeta_j^2 -2i\zeta_jn-n^2-\lambda-A_0\bigr)\end{equation}
the related matrix version
\begin{equation}\label{matrixrecurrencerelation} \begin{bmatrix} a_{n+\ell }\\ \vdots \\ a_{n+1}\\ \vdots \\ a_{n-\ell +1}\end{bmatrix}= \begin{bmatrix} -A_{-\ell}^{-1}A_{-\ell +1}& \dots & B_\ell (n) &\dots & -A_{-\ell}^{-1}A_\ell\\
1&0&\dots &\dots &0\\ 0&1&0&\dots &\vdots\\  \vdots &{}&\ddots&\ddots &\vdots\\ 0&\dots & 0&1&0\end{bmatrix} 
\begin{bmatrix} a_{n+\ell-1} \\ \vdots \\ a_{n}\\ \vdots \\ a_{n-\ell }\end{bmatrix}.\end{equation}
This tells us how to progress one step forwards along the sequence $(a_n)_{-\infty}^\infty$; while a corresponding matrix relation involving $A_\ell^{-1}$ tells us how to take one step backwards. Then the complete sequence $(a_n)$ gives the entries of $R_t$ in (\ref{Hilldet}).
\end{proof}

\begin{rem} If $U_0(x)$ is a trigonometric polynomial with matrix coefficients, then (\ref{secondperiodicODE}) becomes a variant of Mathieu's equation. This occurs for an instance of the differential equation (\ref{secondperiodicODE}) that is used in the theory of graphene to analyze spinors and determine the band structure of their energy levels  \cite[(12)]{LN}].
See \cite[page 135]{R} for instances in orbital motion where the hypotheses of Proposition \ref{Floquetprop} are not satisfied.\end{rem}

\section{Zakharov-Shabat system}\label{S:ZS}
The Zakharov-Shabat system was introduced in \cite{ZS}, and expressed in terms of operators in \cite{BauPop}. The following definition is suggested by  \cite{JPP}. 
\begin{defn}\label{defadmissible}  Let $(-A,B,C)$ be a linear system with input and output space $H_0$ and state space $H$. A bounded linear operator $C:\calD (A)\rightarrow H_0$ is admissible for $e^{-tA}$ if $Ce^{-tA}\xi$ belongs to $L^2((0, \infty ); H_0)$ for all $\xi\in H$ and there exists $K_C(A)$ such that 
\begin{equation}\label{admissible} \int_0^\infty \Vert Ce^{-tA} \xi \Vert^2_{H_0}dt\leq K_C(A)^2\Vert \xi\Vert^2_{H}\qquad (\xi\in H).\end{equation}
\end{defn}

\begin{lem}\label{lemGramians} Suppose for the remainder of this section that $C$ is admissible for $e^{-tA}$ and $B^\dagger$ is admissible for $e^{-tA^\dagger}$. Then\begin{enumerate}[(i)]
\item the operators $R_x:H\rightarrow H$
\[ R_x\xi =\int_x^\infty e^{-tA}BCe^{-tA}\xi dt\qquad (x>0; \xi\in H)\]
are bounded,
\item with $\phi (x)=Ce^{-xA}B$, the Hankel operator $\Gamma_\phi :L^2((0, \infty ); H_0)\rightarrow L^2((0, \infty ); H_0)$, defined by
\[ \Gamma_\phi f(x)=\int_0^\infty \phi (x+y)f(y)\, dy\]
is also bounded.
\item Suppose that 
\begin{equation}\int_0^\infty (\Vert Ce^{-tA}\Vert^2_{\calL^2(H,H_0)}+\Vert B^\dagger e^{-tA^\dagger}\Vert^2_{\calL^2(H,H_0)})dt\end{equation}
converges. Then $R_x$ and $\Gamma_\phi$ are trace class.
\end{enumerate}
\end{lem}
\begin{proof} The hypothesis is equivalent to the statement that 
 the observability Gramian $Q_x$ and the observability Gramian $M_x$ are bounded linear operators on $H$, as defined by the weakly convergent integrals
\begin{equation}\label{QL} Q_x=\int_x^\infty e^{-tA^\dagger}C^\dagger Ce^{-tA}\,dt, \qquad M_x=\int_x^\infty e^{-tA}BB^\dagger e^{-tA^\dagger} dt.\end{equation}
Then the observability operators
$\Theta_x:L^2((0, \infty ):H_0)\rightarrow H$ and the controllability operators $\Xi_x: L^2((0, \infty ); H_0)\rightarrow H$ given by 
\[ \Theta_xf=\int_x^\infty e^{-tA^\dagger} C^\dagger f(t)\, dt, \quad \Xi_xf=\int_x^\infty e^{-tA}Bf(t)\, dt\]
are bounded, since $Q_x=\Theta_x\Theta_x^\dagger$ and $M_x=\Xi_x\Xi_x^\dagger$. Here $Q_x$ and $M_x$ are self-adjoint and non-negative.\par
 (i) Hence the operator $R_x=\Xi_x\Theta_x^\dagger$ is bounded on $H$.\par
(ii) Also $\Theta_x^\dagger\Xi_x$ is bounded on $L^2((0, \infty );H_0)$, so the Hankel operator $\Gamma_\phi$ is bounded. \par
(iii) Here $Q_x$ and $M_x$ are trace class, so $\Theta_x$ and $\Xi_x$ are Hilbert-Schmidt; hence $R_x$ and $\Gamma_\phi$ are trace class, with 
\[ \Vert R_x\Vert_{\calL^1}\leq \Vert \Xi_x\Vert_{\calL^2}\Vert \Theta_x^\dagger \Vert_{\calL^2}, \quad \Vert\Gamma_{\phi}\Vert_{\calL^1}\leq \Vert \Xi_0\Vert_{\calL^2}\Vert\Theta_0\Vert_{\calL^2}.\]
\end{proof}

\indent The hypothesis for Lemma \ref{lemGramians}(iii) is somewhat coarse. In Proposition \ref{propGelfandLevitan}(iv) and Proposition \ref{propPKintegral}, we only need $Q_x$ and $M_x$ to be Hilbert-Schmidt, and in section \ref{S:exrem} we give examples where this occurs under milder hypotheses. In the abstract, we express our results in terms of Hankels, instead of the less familiar $R_x$.\par 
\indent Now let $H_0=\Cb$. From the linear system 
\[\bigl(\hat A, \hat B, \hat C\bigr)= \Biggl( \begin{bmatrix} -A^\dagger& 0\\ 0&-A\end{bmatrix}, \begin{bmatrix} C^\dagger &0\\ 0&B\end{bmatrix} , \begin{bmatrix} 0&C\\ \lambda B^\dagger &0\end{bmatrix}\Biggr)\]
with input and output space $\mathbb{C}^2$ and state space $H\oplus H$, we introduce $\phi (x)=Ce^{-xA}B$ 
and 
\[\Phi (x)=\begin{bmatrix} 0&\phi (x)\\ \lambda\overline{\phi (x)}&0\end{bmatrix},\qquad  \hat R_x=\begin{bmatrix} 0&M_x\\ \lambda Q_x&0\end{bmatrix}\]
so
\[ \hat F_x=(I+\hat R_x)^{-1}=\begin{bmatrix} (I-\lambda Q_xM_x)^{-1}& -(I-\lambda Q_xM_x)^{-1}Q_x\\ -\lambda (I-\lambda M_xQ_x)^{-1}M_x& (I-\lambda M_xQ_x)^{-1}\end{bmatrix}.\]
We note that $M_x, Q_x\geq 0$ as operators, so $\sigma (M_xQ_x)\setminus \{ 0\}=\sigma (M_x^{1/2}Q_xM_x^{1/2})\setminus \{ 0\}\subset (0, \infty )$, hence $I-\lambda M_xQ_x$ is invertible for all $\lambda\in \Cb \setminus (0, \infty )$

\begin{defn}\label{defdiffring} Let $\mathcal{A}$ be the complex algebra formed by linear combinations of products of $I,\hat A,\hat F$, and let ${\frac{d}{dx}}$ be a derivation on $\mathcal{A}$ such that 
\begin{equation}{\frac{d\hat A}{dx}}=0,\end{equation}
\begin{equation}\label{Fderivative}{\frac{d\hat F}{dx}}=\hat A\hat F+\hat F\hat A-2\hat F\hat A\hat F.\end{equation}
We also introduce the associative product $\ast$ on ${\mathcal{A}}$ by
\begin{equation} Y\ast Z=Y(\hat A\hat F+\hat F\hat A-2\hat F\hat A\hat F)Z\end{equation}
and the differential expressions
\begin{equation}D_xY=(\hat A-2\hat A\hat F)Y+{\frac{dY}{dx}} +Y(\hat A-2\hat F\hat A).\end{equation} 
 Then we let
\begin{equation} \lfloor Y\rfloor = \hat Ce^{-x\hat A}\hat F_xY\hat F_xe^{-x\hat A}\hat B\qquad (x>0).\end{equation}
\end{defn}
This bracket operation is related to that of P\"oppe \cite[p 622]{PKdV} and Definition \ref{defpoppebracket} (ii), and our identities that correspond to his Theorem 3.1 are summarized in the following Proposition.

\begin{prop}\label{propdiffring} There is a homomorphism of differential rings 
\[ \lfloor \cdot \rfloor :\bigl(\mathcal{A}, \ast , D_x \bigr)\rightarrow \bigl( C^\infty ((0, \infty ); M_{2\times 2}(\mathbb{C})), \cdot ,d/dx\bigr) \]
such that
\begin{equation}\label{diffring} {\frac{d}{dx}}\lfloor Y\rfloor =\lfloor D_xY\rfloor, \qquad \lfloor Y\ast Z\rfloor =\lfloor Y\rfloor \lfloor Z\rfloor .\end{equation}
\end{prop}
\begin{proof} Using the Lyapunov identity
\begin{equation}\label{Lyapunovidentity} {\frac{d\hat R_x}{dx}}=-\hat A \hat R_x-\hat R_x\hat A,\end{equation}
one verifies (\ref{Fderivative}). We can also write
\[ {\frac{d\hat R_x}{dx}} =-e^{-x\hat A}\hat B\hat Ce^{-x\hat A};\]
given this, it is straightforward to verify (\ref{diffring}) by a direct calculation as in \cite[Lemma 4.1]{BlowerNew}.
\end{proof}

Now we let $T(x,y)=-\hat Ce^{-x\hat A}(I+\hat R_x)^{-1}e^{-y\hat A}\hat B$, or more explicitly
\[ T(x,y)=\begin{bmatrix} \lambda Ce^{-xA}(I-\lambda M_xQ_x)^{-1} M_xe^{-yA^\dagger} C^\dagger & - Ce^{-xA}(I-\lambda M_xQ_x)^{-1}e^{-yA}B\\
-\lambda B^\dagger e^{-xA^\dagger} (I-\lambda Q_xM_x)^{-1}e^{-yA^\dagger}C^\dagger & \lambda B^\dagger e^{-xA^\dagger} (I-\lambda Q_xM_x)^{-1}Q_xe^{-yA}B\end{bmatrix} .\]
There exists $x_0$ such that $\Vert Q_xM_x\Vert_{\mathcal{L}(H)}<1$ for all $x>x_0$ since $Q_x, M_x\rightarrow 0$ as $x\rightarrow\infty $, so $T$ is well defined.

\begin{prop}\label{propGelfandLevitan}\begin{enumerate}[(i)] \item Then the Gelfand-Levitan equation is satisfied
\[ 0=\Phi (x+y)+T(x,y)+\int_x^\infty T(x,z) \Phi (z+y)\, dz\qquad (x_0<x<y).\]
\item  The nonlinear differential equation is satisfied 
\begin{equation}\label{PDEforT} {\frac{\partial^2T}{\partial x^2}} -{\frac{\partial^2T}{\partial y^2}}=\Bigl(-2{\frac{d}{dx}}T(x,x)\Bigr) \, T(x,y).\end{equation}
\item Suppose that $\hat C\hat A $ and $\hat A\hat B$ are bounded linear operators, and let $\Delta$ be the differential operator
\[ \Delta \Psi (x)=-{\frac{\partial^2\Psi }{\partial x^2}}+\Bigl(-2{\frac{d}{dx}}T(x,x)\Bigr) \Psi (x).\]
Then $(\cos (t\sqrt \Delta ))_{t\geq 0}$ is a cosine family on $L^2((x_0, \infty ); \Cb^2)$ such that 
\[\Vert \cos (t\sqrt{\Delta })\Vert_{\calL (L^2)}\leq Me^{\omega t}\qquad (t\geq 0).\]
\item  Suppose that $Q_x$ and $M_x$ are Hilbert-Schmidt operators on $H$. Then the Fredholm determinant satisfies
\begin{equation} \trace T(x,x)={\frac{d}{dx}}\log\det (I-\lambda Q_xM_x).
\end{equation}
\end{enumerate}
\end{prop}
\begin{proof} (i) This is a direct computation by substituting the functions in terms of $(-\hat A,\hat B, \hat C)$.\par
(ii) This follows from (i) by the uniqueness of solutions.\par
(iii) 
 Let $U(x)=-2{\frac{d}{dx}}T(x,x)$. We have
\begin{align} {\frac{d}{dx}}T(x,x)&=\hat C\hat Ae^{-x\hat A}(I+\hat R_x)^{-1}e^{-x\hat A}\hat B+\hat Ce^{-x\hat A}(I+\hat R_x)^{-1}e^{-x\hat A}\hat B\nonumber\\
&\qquad -\hat Ce^{-x\hat A}(I+\hat R_x)^{-1}e^{-x\hat A}\hat B\hat Ce^{-x\hat A}(I+\hat R_x)^{-1}e^{-x\hat A}\hat B,\end{align}
where each term is uniformly bounded for $x>x_0$; hence $U(x)\in L^\infty (\Rb ;M_{2\times 2}(\Cb ))$.\par
\indent For all $\Psi_0\in L^2(\Rb; M_{2\times 2} (\Cb ))$, there 
exists a unique solution to 
\begin{align} {\frac{\partial^2}{\partial t^2}}\Psi (t,x)-{\frac{\partial^2}{\partial x^2}}\Psi (t,x) +U(x)\Psi (x,t)&=0\nonumber\\
\Psi (0,x)&=\Psi_0(x)\nonumber\\
{\frac{\partial}{\partial t}}\Psi (0,x)&=0.\end{align}
When $U=0$, the unique solution is given by D'Alembert's solution 
\[\Psi (t,x)=(1/2)(\Psi_0(x-t)+\Psi_0(x+t)).\] 
By a standard result of operator semigroup theory \cite[theorem 8.5]{Gold}, one can perturb $L_0=-{\frac{\partial^2}{\partial x^2}}$ by adding the bounded operator of multiplication by $U$, and obtain a generator $L_1$ of a cosine family $\cos t\sqrt{L_1}$ such that $\Vert \cos (t\sqrt{L_1})\Vert\leq Me^{\omega t}$ for some $M,\omega\geq 0$ and all $t\in \Rb$.\par
(iv) By hypothesis, $Q_xM_x\in \mathcal{L}^1(H)$, so the determinant is well defined. To prove this identity, we rearrange operators in the following expression
\begin{align}{\hbox{trace}}\, & T(x,x)\nonumber\\
&=\lambda Ce^{-xA}(I-\lambda M_xQ_x)^{-1}+\lambda B^\dagger e^{-xA^\dagger} (I-\lambda Q_xM_x)^{-1}Q_xe^{-xA}B\nonumber\\
&=\lambda{\hbox{trace}}\bigl( (I-Q_xM_x)^{-1}M_xe^{-xA^\dagger}C^\dagger Ce^{-xA}+(I-\lambda Q_xM_x)^{-1}Q_xe^{-xA}BB^\dagger e^{-xA^\dagger}\bigr)\nonumber\\
&=-\lambda {\hbox{trace}}\Bigl( (I-\lambda M_xQ_x)^{-1}M_x{\frac{dQ_x}{dx}}+(I-\lambda Q_xM_x)^{-1}Q_x{\frac{dM_x}{dx}}\Bigr)\nonumber\\
&= -\lambda {\hbox{trace}}\Bigl( M_x(I-\lambda Q_xM_x)^{-1}{\frac{dQ_x}{dx}}+(I-\lambda Q_xM_x)^{-1}Q_x{\frac{dM_x}{dx}}\Bigr)\nonumber\\
&=-\lambda {\hbox{trace}}\Bigl((I-\lambda Q_xM_x)^{-1}\Bigl({\frac{dQ_x}{dx}}M_x+Q_x{\frac{dM_x}{dx}}\Bigr)\Bigr)\nonumber\\
&={\frac{d}{dx}}\log\det (I-\lambda Q_xM_x).\end{align}
\end{proof}
Note that we can write
\[ T(x,y)=\begin{bmatrix}Y(x,y)&V(x,y)\\ \lambda X(x,y)&W(x,y)\end{bmatrix}\]
where $X(x,x)=\overline{V(x,x)}$. 
Within $M_{2\times 2} (\mathbb{C})$ we note two real-linear subspaces that have real dimension four, namely the quaternions $\Hb$ and the $2\times 2$ complex Hermitian matrices. These arise as follows. \par

\begin{prop}\label{propSchrodinger} Suppose that $CA^j$ is admissible for $(e^{-tA})_{t\geq 0}$ and $B^\dagger (A^\dagger )^j$ is admissible for $(e^{-tA^\dagger })_{t\geq 0}$ for $j=0, 1, \dots.$ 
\begin{enumerate} [(i)]
\item Then the matrix function $\Psi$, defined by an integral that converges in the $L^2$ sense, 
\[ \Psi (x)=\begin{bmatrix} e^{ikx}&0\\ 0&e^{-ikx}\end{bmatrix}+\int_x^\infty T(x,y)\begin{bmatrix} e^{iky}&0\\ 0&e^{-iky}\end{bmatrix}dy,\]
satisfies Schr\"odinger's equation with a matrix potential $U(x)=-2{\frac{d}{dx}}T(x,x)$,  
\[ -{\frac{d^2}{dx^2}} \Psi (x)+ \Bigl( -2{\frac{d}{dx}}T(x,x)\Bigr)\Psi (x)=k^2\Psi (x).\]
\item Let $\lambda =1$; then $U(x)$ is a Hermitian matrix.
\item Let $\lambda =-1$; then $U(x)$ is a quaternion, and the differential ring generated by $U$ consists of quaternions.
\end{enumerate}
\end{prop}
\begin{proof} (i) First, we check convergence of the integral defining $\Psi$. 
By hypothesis, there exists $K_{B^\dagger}(A^\dagger )$ such that 
\[ \int_0^\infty \Vert B^\dagger e^{-tA^\dagger}\xi \Vert^2_{H_0}dt\leq K_{B^\dagger}(A^\dagger )^2\Vert \xi \Vert^2_H \qquad (\xi\in H)\]
Let $RHP=\{ z\in \Cb :\Im z>0\}$ and $H^2(RHP: H_0 )$ be the Hardy space of holomorphic $f:RHP\rightarrow H_0$ such that $\sup_{x>0} \int_{-\infty}^\infty \Vert f(x+iy)\Vert^2_{H_0}\, dy<\infty $. Then by the Paley-Wiener theorem,
\[\int_x^\infty e^{-st} B^\dagger e^{-tA^\dagger }\xi dt=e^{-sx}B^\dagger (sI+A^\dagger )^{-1} e^{-xA^\dagger}\xi\qquad (\Re s>0)\]
gives a function in the Hardy space $H^2(RHP; H_0)$. Choosing $\xi_0\in H_0$ and $\xi =(I+R_x^\dagger )^{-1}e^{-xA^\dagger} C^\dagger \xi_0$, we have
\[ e^{-sx}B^\dagger (sI+A^\dagger)^{-1}e^{-xA^\dagger} (I+R_x^\dagger )^{-1} e^{-xA^\dagger} C^\dagger \xi_0\in H^2(RHP; H_0).\]
for $s=\varepsilon +ik$ and the boundary values with $\varepsilon\rightarrow 0+$ give
\[ \Bigl( \int_x^\infty T(x,y)e^{iky}\, dy\Bigr)^\dagger \xi_0.\]
Then the differential equation follows from (\ref{PDEforT}).\par
Suppose that $C,CA, CA^2$ are admissible for $(e^{-tA})$ and $B^\dagger, B^\dagger A^\dagger B^\dagger (A^\dagger )^2$ are admissible for $(e^{-tA^\dagger})$. Then $U$ is twice differentiable, with the following derivatives
\begin{align} U&=-4\bigl\lfloor A\bigr\rfloor ,\nonumber\\
{\frac{dU}{dx}}&=-8\bigl\lfloor A(I-2F)A\big\rfloor ,\\
{\frac{d^2U}{dx^2}}&=-16\bigl\lfloor A(I-2F)A(I-2F)A\bigr\rfloor +16\bigl\lfloor A(AF+FA-2FAF)A\bigr\rfloor\end{align}
which includes terms such as 
\[ 16\bigl\lfloor A^2FA\bigr\rfloor =16Ce^{-xA}\Bigl( \bigl([F,A^2]+A^2F\bigr) F\bigl( [A,F]+FA\bigr)\Bigr)e^{-xA}B\]
where $[F,A^2]=-F[R,A^2]F$, and $A^2R$ and $RA^2$ are bounded operators. By induction, one can prove that $D_x^nA=2^{n}A^{n+1}+p_n(A,F)$ where 
$p_n(A,F)$ is a polynomial in the noncommuting variables $A$ and $F$ such that $p_n(A,F)$ has degree less than or equal to $n+1$ in the variable $A$ and monomials in $p_n(A,F)$ have factors $I,A, \dots , A^n$ but not $A^{n+1}$. \par

\indent (ii) For $\lambda =1$, we have $T(x,x)=T(x,x)^\dagger$, and $-2{\frac{d}{dx}}T(x,x)=-2{\frac{d}{dx}}T(x,x)^\dagger$. Furthermore, we can extract $Y(x,y)$ and $V(x,y)$ from $T(x,y)$ and solve the Gelfand-Levitan equation 
\begin{align} 0=&\begin{bmatrix} Y(x,y)&V(x,y)\\ \overline{V}(x,y)&\overline{Y}(x,y)\end{bmatrix} +\begin{bmatrix} 0&\phi (x+y)\\ \overline{\phi} (x+y) &0\end{bmatrix}\nonumber\\
&\quad +\int_x^\infty 
\begin{bmatrix} Y(x,z)&V(x,z)\\ \overline{V}(x,z)&\overline{Y}(x,z)\end{bmatrix}\begin{bmatrix} 0&\phi (z+y)\\ \overline{\phi} (z+y) &0\end{bmatrix}\, dz.\end{align}
Here $T(x,y)$ and $\Phi (x+y)$ can be expressed as real linear combinations of the matrices\par
\[\sigma_0=\begin{bmatrix} 1&0\\ 0&1\end{bmatrix},\quad  
\sigma_1=\begin{bmatrix} 0&1\\ 1&0\end{bmatrix}, \quad
\sigma_2=\begin{bmatrix} 0&-i\\ i&0\end{bmatrix}\quad
i\sigma_3 =\begin{bmatrix} i&0\\ 0&-i\end{bmatrix},\quad
\]
where $T(x,x)$ and $\Phi (x+y)$ are hermitian.\par
(iii) For $\lambda =-1$, we have quaternions since we can extract $Y(x,y),V(x,y)$ from $T(x,y)$ and obtain a solution of the Gelfand-Levitan equation of the form
\begin{align} 0=&\begin{bmatrix} Y(x,y)&V(x,y)\\ -\overline{V}(x,y)&\overline{Y}(x,y)\end{bmatrix} +\begin{bmatrix} 0&\phi (x+y)\\ -\overline{\phi} (x+y) &0\end{bmatrix}\nonumber\\
&\quad +\int_x^\infty 
\begin{bmatrix} Y(x,z)&V(x,z)\\ -\overline{V}(x,z)&\overline{Y}(x,z)\end{bmatrix}\begin{bmatrix} 0&\phi (z+y)\\ -\overline{\phi} (z+y)&0\end{bmatrix}\, dz.\end{align}
so that the integral equation may be expressed in terms of quaternion units
\[ \Phi (x+y)=\Re \Phi (x+y)\begin{bmatrix} 0&1\\ -1&0\end{bmatrix} +\Im \Phi (x+y)\begin{bmatrix} 0&i\\ i&0\end{bmatrix}\]
\begin{align} T(x,y)&=\Re Y(x,y) \begin{bmatrix} 1&0\\ 0&1\end{bmatrix}+\Im  Y(x,y) \begin{bmatrix} i&0\\ 0&-i\end{bmatrix}\nonumber\\
&+\Re V(x,y) \begin{bmatrix} 0&1\\ -1&0\end{bmatrix}+\Im V(x,y) \begin{bmatrix} 0&i\\ i&0\end{bmatrix}\end{align}
Also $U(x)=-2{\frac{d}{dx}} T(x,x)$ is a quaternion; hence the noncommutative differential ring generated by $U(x)$ with differential $d/dx$ is a subring of $C^\infty ((0, \infty );\Hb )$. We have quaternions
\[ U=-4\lfloor \hat A\rfloor,\quad  {\frac{d^jU}{dx^j}}=-4\lfloor D^j_x\hat A\rfloor ,\quad  U^j=(-4)^j \lfloor \hat A\ast\dots \ast \hat A\rfloor  .\]
\end{proof}

\section{Spectra of Hankel operators associated with ZS}\label{S:SpectraZS}

Under the hypotheses of (ii) and (iii) of Proposition \ref{propSchrodinger}, we consider
\[(ii)\quad  \Phi (x)=\begin{bmatrix} 0&\phi (x)\\ \overline{\phi (x)}&0\end{bmatrix} ,\qquad (iii) \quad \Psi (x)=\begin{bmatrix} 0&i\phi (x)\\ -i\overline{\phi (x)} &0\end{bmatrix}.\]
By the spectral theorem \cite[p 177]{BS}, a bounded and self-adjoint operator $\Gamma$ on a separable Hilbert space $H$ has a resolution $H=\int^\oplus H(t)\mu (dt)$ where $\mu$ is a positive Radon measure on $\Rb$ called the scalar spectral measure or maximal spectral type and $\Gamma h(t)=th(t)$ where $h(t)\in H(t)$. The spectral multiplicity is $\nu :\Rb\rightarrow \Nb\cup \{ \infty \}$ with $\nu (t)=\dim H(t)$. In particular, let $\nullity (\Gamma )=\{ \xi\in H:\Gamma \xi =0\}$.

\begin{prop}\label{propmult} Then $\Gamma_\Phi$ gives a bounded and self-adjoint operator on $L^2((0, \infty ); \Cb^{2\times 1})$ with scalar spectral measure $\mu$. Let $\mu=\mu_a+\mu_s$ be the Lebesgue decomposition into an absolutely continuous measure $\mu_a$ and a singular measure $\mu_s$ with respect to Lebesgue measure. Then the spectral multiplicity function $\nu$ satisfies
\begin{enumerate}[(i)]
\item either $\nullity (\Gamma_\Phi )=\{ 0\}$, or $\dim \nullity (\Gamma_\Phi )=\infty$;
\item $\Gamma_\Phi $ is not invertible;
\item $\nu (t)=\nu (-t)$ for all $t>0$.
\end{enumerate}
A similar statement holds for $\Gamma_\Psi$ with obvious changes.
\end{prop}
\begin{proof} Recall the Laguerre polynomials $L_n(t)={\frac{e^t}{n!}} {\frac{d^n}{dt^n}}(t^ne^{-t})$, so $(e^{-t/2}L_n(t))_{n=0}^\infty$ gives a complete orthonormal basis of $L^2((0, \infty ); \Cb )$. We introduce 
\[\gamma_n=\sqrt{2}\int_0^\infty \Phi (x)L_n(2x)e^{-x}\, dx,\]
which are $2\times2$ self-adjoint matrices for $n=0, 1, \dots$. Then the Hankel integral operator $\Gamma_\Phi$ on $L^2((0, \infty ); \Cb^{2\times 1})$ is unitarily equivalent to the self-adjoint block-Hankel matrix $\Gamma=[\gamma_{n+m}]_{n,m=0}^\infty$ on $\ell^2(\Cb^{2\times 1})$. Now we apply Theorem 2 of \cite{MPT}, which gives (i), (ii) and the balance conditions for multiplicity of absolutely continuous and singular spectra 
\begin{align}\label{numult}\vert\nu (t)-\nu (-t)\vert \leq 4&\quad {\hbox{for}}\quad \mu_a\quad{\hbox {almost all}}\quad t;\nonumber\\
 \vert \nu (t)-\nu (-t)\vert\leq 2&\quad {\hbox{for}}\quad \mu_s\quad {\hbox{almost all}}\quad t.\end{align}
\indent To sharpen this, we suppose that $\eta\in L^2((0, \infty ); \Cb )$ is an eigenvector of $\Gamma_\phi^\dagger \Gamma_\phi$, so $\Gamma_\phi^\dagger\Gamma_\phi \eta =s^2\eta$ for some $s>0$; then 
\begin{equation}\label{Schmidtpair} \begin{bmatrix} 0&\Gamma_\phi \\ \Gamma_\phi^\dagger&0\end{bmatrix} \begin{bmatrix}s^{-1}\Gamma_\phi \eta \\ \eta \end{bmatrix} =s  \begin{bmatrix}s^{-1}\Gamma_\phi \eta \\ \eta \end{bmatrix}; \end{equation}
likewise, when we replace $s$ by $-s$, so eigenvalues arise in pairs of opposite sign. Hence $\nu (t)=\nu (-t)$ for all $y$ in the support of the discrete part of $\mu_s$, namely the set of eigenvalues.\par
\indent Since $\Gamma_\Phi$ is self-adjoint, its spectrum coincides with $\sigma_{ap}(\Gamma_\Phi)$; likewise for $\Gamma_\phi^\dagger\Gamma_\phi$. Let $(\eta_j)$ be an approximate eigenvector for $\Gamma_\phi^\dagger\Gamma_\phi$ with approximate eigenvalue $\lambda>0$, so $\Vert \eta_j\Vert =1$ and $\Vert \Gamma_\phi^\dagger \Gamma_\phi \eta_j-\lambda\eta_j\Vert\rightarrow 0$ as $j\rightarrow\infty$; then $[\Gamma_\phi \eta_j /\sqrt{\lambda}; \eta_j ]^\top$ gives an approximate eigenvector of $\Gamma_\Phi$ corresponding to approximate eigenvalue $\sqrt{\lambda}$ as in (\ref{Schmidtpair}); conversely, approximate eigenvectors for $\Gamma_\Phi$ with approximate eigenvalue $s>0$ gives an approximate eigenvector for $\Gamma_\phi^\dagger\Gamma_\phi$ with approximate eigenvalue $s^2$. This suggests that 
\[ \nu_{\Gamma_\phi^\dagger \Gamma_\phi} (\lambda )=\nu_{\Gamma_\Phi}\bigl(\sqrt{\lambda}\bigr)=\nu_{\Gamma_\Phi}\bigl(-\sqrt{\lambda }\bigr)\qquad (\lambda >0),\]
as in (iii). To make this precise, we introduce the Banach space ultrapower $(H)_\calL$, which is essentially the quotient space $\ell^\infty (\Nb; H)/\calN_\calL$ where $\calN_\calL= \{ (\xi_j)\in \ell^\infty (\Nb ;H): \LIM  \Vert \xi_j\Vert =0\}$ and $\LIM$ is a Banach limit on $\Nb$. Then by a simple case of \cite[Theorem 3.3(ii)]{Hein}, $(H)_\calL$ is a Banach space that satisfies the parallelogram law, hence is a Hilbert space. Let $(\Gamma_\phi^\dagger\Gamma_\phi )$ be the bounded linear operator on $(H)_\calL$ that is determined by $(\xi_j)\mapsto (\Gamma_\phi^\dagger\Gamma_\phi \xi_j)$ for $(\xi_j)\in \ell^\infty (\Nb ;H)$. Then by the Gram-Schmidt process, one shows that  $\nu_{\Gamma_\phi^\dagger\Gamma_\phi}(\lambda)=\dim\{ \eta\in (H)_\calL: (\Gamma_\phi^\dagger\Gamma_\phi)\eta =\lambda\eta\}$.\par
\indent We have
\[ \Gamma_\Phi^2=\begin{bmatrix}\Gamma_\phi \Gamma_\phi^\dagger &0\\ 0&\Gamma_\phi^\dagger\Gamma_\phi\end{bmatrix}\]
where $\sigma (\Gamma_\phi\Gamma_\phi^\dagger )=\sigma (\Gamma_\phi^\dagger\Gamma_\phi )$ and there exist spectral families $K=\int^\oplus K(t)\mu (dt)$ for $\Gamma_\phi\Gamma_\phi^\dagger$ and $G=\int^\oplus G(t)\omega (dt)$ for $\Gamma_\phi^\dagger \Gamma_\phi$ where $\mu$ and $\omega$ are positive measures on $\sigma (\Gamma_\phi^\dagger\Gamma_\phi )$. Then by the Radon-Nikodym theorem, we can introduce $\lambda =\mu+\omega$ and $\lambda$-measurable functions $k,g$ such that $k,g\geq 0$ and $k+g=1$ such that $\mu (dt)=k(t)\lambda (dt)$ and $\omega (dt)=g(t)\lambda (dt)$. Then let $H(t)=K(t)\oplus G(t)$ with the inner product
\[ \Bigl\langle \begin{bmatrix} k_1(t)\\ g_1(t)\end{bmatrix},\begin{bmatrix} k_2(t)\\ g_2(t)\end{bmatrix}\Bigr\rangle_{H(t)}=k(t)\langle k_1(t), k_2(t)\rangle_{K(t)}+g(t)\langle g_1(t), g_2(t)\rangle_{G(t)}.\]
Then $H=\int^\oplus H(t)\lambda (dt)$ contains $K$ and $G$ as orthogonal subspaces, and $\Gamma_\Phi^2$ is unitarily equivalent to multiplication by $t$ on $H$. 
\end{proof}
In the case of a nonnegative and self-adjoint compact Hankel operator, nonzero eigenvalues are all simple. This is in contrast to Proposition \ref{propmult}, as we have not proposed a bound on $\nu (t)$. The decay rate of singular numbers is reflected in the asymptotics of the Fredholm determinant, as follows.   

\begin{prop}\label{asymprop}Suppose that $\Gamma_\phi$ is Hilbert-Schmidt, and let $s_0^2\geq s_1^2\geq \dots$ be the eigenvalues of $\Gamma_\phi^\dagger\Gamma_\phi$, listed according to multiplicity. Let $n(t)=\sharp \{ j: ts_j^2\geq 1\}$.  
\begin{enumerate}[(i)]\item
Then $\det (I+ix\Gamma_\phi )$ is an entire and even function of $x$.
\item There exist $\alpha, \beta >0$ such that $n(t)\sim \alpha t^\beta$ as $t\rightarrow\infty$, if and only if 
\begin{equation}\label{asymp} \log\det (I+ix\Gamma_\Phi )\sim \pi\alpha \cosec (\pi\beta ) x^{2\beta }\qquad (x\in (0, \infty ), x\rightarrow\infty ).\end{equation}
\end{enumerate}
\end{prop}

\begin{proof} (i) 
We have a standard summation formula
\begin{align}\label{summation} \log\det (I+ix\Gamma_\Phi )&=\log\det (I+x^2\Gamma_\phi^\dagger \Gamma_\phi)\nonumber\\
&=\log\prod_{j=0}^\infty (1+x^2s_j^2)\nonumber\\
&=x^2\int_0^\infty {\frac {n(t)\, dt}{t(t+x^2)}}\qquad (x^2\in \Cb\setminus (0, \infty )).\end{align}
The product converges, hence we have an entire function.\par
\indent (ii) If $n(t)=\alpha t^\beta$, then one can substitute $t=x^2\tan\theta$ and reduce the integral to $\alpha \Gamma (\beta )\Gamma (1-\beta )x^{2\beta}$, where here $\Gamma$ is Euler's gamma function. By an approximation argument from \cite[p 271]{T}, one obtains a corresponding asymptotic formula when $n(t)\sim \alpha t^{\beta}$ and $x\rightarrow\infty$ through real values. The converse also holds, by a Tauberian theorem due to Valiron \cite[(58), p.237]{V} . See also \cite[Theorem 6.1]{Blower0} for conditions on $\phi$ that ensure rapid convergence of $(s_j^2)_{j=0}^\infty$. \par

\end{proof}
\section{Integral equations relating to KP}\label{S:KP}

 We start by introducing families of linear systems and related operators, and obtain a determinant formula which we then express in more classical terms. Let $(-A, B(\zeta ),C)$ be a one-parameter family of continuous-time linear systems with state space $H$ and input and output space $H_1$. Let $\Phi (z,\zeta )=Ce^{-zA}B(\zeta )$ be the impulse response function and $R_x=\int_x^\infty B(\zeta )Ce^{-\zeta A}\, d\zeta$ as in (\ref{Roperator})  
these are our basic operator functions, from which we introduce various auxiliary functions. Let $\Theta_x:L^2((0,\infty );H_1)\rightarrow H$ and $\Xi_x: L^2((0,\infty );H_1)\rightarrow H$ be defined by 
\begin{equation} \Theta_xh=\int_x^\infty e^{-\zeta A^\dagger}C^\dagger h(\zeta )d\zeta,\end{equation}
and
\begin{equation} \Xi_x h=\int_x^\infty B(\zeta )h(\zeta ) \, d\zeta\qquad (h\in L^2((0, \infty ); H_1)) \end{equation}
and one computes $\Theta^\dagger_x :H\rightarrow L^2((0, \infty ); H_1)$ 
\begin{equation}\Theta^\dagger_x\alpha ={\mathbb I}_{(x, \infty )}(\zeta ) Ce^{-\zeta A}\alpha\qquad (\alpha\in H)\end{equation} 
Hence
\begin{align}\Theta^\dagger_x \Xi_x h&={\mathbb I}_{(x,\infty )}(z )Ce^{-zA} \int_x^\infty B(\zeta )h(\zeta )\, d\zeta \nonumber\\
&=\int_0^\infty {\mathbb I}_{(x,\infty )}(z)\Phi (z,\zeta ){\mathbb I}_{(x,\infty )}(\zeta )h(\zeta )d\zeta \qquad (h\in L^2((0, \infty );H_1),\end{align}
so $\Theta^\dagger_0\Xi_0$ is the integral operator on $L^2((0, \infty );H_1)$ that has kernel $\Phi (z, \zeta )$; while 
\begin{align}\Xi_x\Theta^\dagger_x\alpha &=\int_x^\infty 
{\mathbb I}_{(x,\infty)}(\zeta )B(\zeta )Ce^{-\zeta A}\alpha \, d\zeta \nonumber\\
&=R_x\alpha \qquad (\alpha\in H),\end{align}
which is our other basic operator. Next, we let
\begin{equation} K(z, \zeta )=-Ce^{-zA}(I+R_z)^{-1}B(\zeta )\qquad (0<z<\zeta )\end{equation}
which determines a Volterra-type integral operator 
\begin{equation} h\mapsto h(z)+\int_z^\infty K(z, \zeta )h(\zeta )\, d\zeta \qquad (h\in L^2((0, \infty );H_1).\end{equation}
\vskip.1in
\noindent 
\begin{prop}\label{propPKintegral} Suppose that $\Theta_0$ and $\Xi_0$ are Hilbert-Schmidt, and $\Vert R_x\Vert <1$ for all $x>0$. 
\begin{enumerate}[(i)]
\item Then the Gelfand-Levitan equation \begin{equation} \Phi (z, \zeta )+K(z, \zeta )+\int_z^\infty K(z, \eta )\Phi (\eta , \zeta ) d\eta =0\end{equation}
is satisfied;
\item the operator $\Theta_x^\dagger\Xi_x$ is trace class and may be expressed as the integral operator on $L^2((0, \infty );H_1)$ that has kernel $\Phi (x+z, x+\zeta )$;
\item for $H_1=\mathbb{C}$, the Fredholm determinant satisfies
\begin{equation}\label{diagdiff} {\frac{d}{dx}}\log\det (I+\Theta^\dagger_x \Xi_x)=K(x,x)\qquad (x>0).\end{equation}
\end{enumerate}
\end{prop}
\begin{proof} (i) Here $I+R_x$ is invertible in $\mathcal{L}(H)$, so $K(z, \zeta )$ is defined. Then one verifies the Gelfand-Levitan equation by substitution. See \cite[Lemma 5.1]{Blower1}.\par
(ii) The operator $\Theta_0^\dagger\Xi_0$ is a product of Hilbert-Schmidt operators, hence trace class, and has compression $\Theta_x^\dagger\Xi_x$, which is also trace class.\par
(iii) We have 
\begin{equation} \det (I+\Theta_x^\dagger \Xi_x)=\det (I+\Xi_x\Theta^\dagger_x)=\det(I+R_x)\end{equation}
so 
\begin{align}{\frac{d}{dx}}\log\det (I+\Theta_x^\dagger\Xi_x)&={\frac{d}{dx}}{\hbox{trace}}\log (I+R_x)\nonumber\\
&={\hbox{trace}} \Bigl( (I+R_x)^{-1}{\frac{d}{dx}} R_x\Bigr)\nonumber\\
&=-{\hbox{trace}}\bigl( (I+R_x)^{-1} B(x)Ce^{-xA}\bigr)\nonumber\\
&=-{\hbox{trace}}\bigl( Ce^{-xA}(I+R_x)^{-1}B(x)\bigr)\nonumber\\
&={\hbox{trace}}\, K(x,x).\end{align}
For $H_1=\mathbb{C}$, we have $K(x,x)={\hbox{trace}}\, K(x,x)$. 
\end{proof}

Squares of Hankel operators are themselves associated with tau functions, and on account of 
Theorem \ref{sigmathrm}(iv) can be simpler to use in computations than Hankel operators. For a specific example relating to Painlev\'e equation and concentric $KdV$, see \cite[p. 172]{DrazJ}. In \cite[Section 4]{BD}, we discussed a differential ring related to the Hastings-McLeod solution of Painlev\'e II and obtained the following determinant from a Hankel square operator in \cite[Proposition 4.1(ii)]{BD} and Example \ref{Airy}. The values of $\det (I\pm R_x)$ are equivalent data to the values of $\det (I-R_x^2)$ and $\det ((I+R_x)(I-R_x)^{-1})$ for $R_x\in\calL^1$ such that $I-R_x$ is invertible.
\begin{prop} Let $(-A,B,C)$ be admissible, and introduce
\[ F(x,z)=\int_0^\infty \phi (x+y)\phi (y+z)\, dy.\]
Then the integral equation
\begin{equation}\label{GLMsquared} K(x,z)+\kappa F(x,z)+\kappa \int_x^\infty K(x,y)F(y,z)\, dy=0\end{equation}
has solution
\[ K(x,z)=-\kappa Ce^{-xA}(I+\kappa R_xR_0)^{-1}R_0e^{-zA}B,\]
where 
\[ K(x,x)={\frac{d}{dx}}\log\det (I+\kappa R_{x/2}^2).\]
\end{prop}
\begin{proof}
One can verify (\ref{GLMsquared}), starting from the formula $F(x,z)=Ce^{-xA}R_0e^{-zA}B$. The determinant formula 
\[ K(x,x)={\frac{d}{dx}}\log\det (I+\kappa R_xR_0)\]
follows as in Lemma \ref{taudiff}, and we have
\begin{align}\det \bigl(I+\kappa R_xR_0\bigr)&=\det \bigl(I+e^{-xA}R_0e^{-xA}R_0\bigr)\nonumber\\
&=\det \bigl(I+\kappa e^{-xA/2}R_0e^{-xA}R_0e^{-xA/2}\bigr)\nonumber\\
&=\det \bigl(I+\kappa R_{x/2}^2\bigr).\end{align}
\end{proof}

\indent

\begin{ex}\label{exBessel} Consider the differential operator 
\begin{equation}L_0\psi =-{\frac{d^2\psi }{dx^2}}+\beta \psi\qquad  (\psi \in C_c^\infty ((0, \infty ); \mathbb{C}))\end{equation} 
which is semibounded for $\beta \in \mathbb{R}$ and positive for all $\beta>0$. Suppose for simplicity that $\beta\geq 0$. 

 Here we choose the continuous-time linear systems $(-A,B(\zeta ),C)$ with state space $H=L^2((0, \infty ); \mathbb{C}^{2\times 1})$ and input and output space $H_1=\mathbb{C}$ with operators
 \begin{align} -A&=\begin{bmatrix}0&-I\\ L_0&0\end{bmatrix}, \nonumber\\
 B(\zeta )\alpha &=\begin{bmatrix} f(t+\zeta )\\ g(t+\zeta )\end{bmatrix}\alpha ,\nonumber\\
C:\begin{bmatrix} h_1(t)\\ h_2(t)\end{bmatrix} &\mapsto h_1(0)\qquad (h_1\in H^1((0,\infty );\mathbb{C})).\end{align}
 Note that $f(t)\mapsto f(t+\zeta )$ is the backward shift operator, which is strongly continuous and unitary on $L^2(\mathbb{R}; \mathbb{C})$ and strongly continuous and coisometric on $L^2((0, \infty ); \mathbb{C})$. Then $CB(\zeta )=f(\zeta )$, and the impulse response function is 
 \begin{equation}\Phi (z, \zeta )=Ce^{-zA}B(\zeta )= \psi (z,\zeta ).\end{equation}
\end{ex}

 Next we provide an explicit expression for this solution in classical style. We use spatial coordinates $(z, \zeta )\in \mathbb{R}^2$, variables $x,y\in\mathbb{R}$, a spectral parameter $\kappa\in {\mathbb{R}}\cup i{\mathbb{R}}$ for an operator in $(z, \zeta )$, and a Fourier transform variable $\omega$ which is dual to $y$. 

We show how to obtain $\psi (\zeta ,z;\kappa )$ that is a solution of the system
\begin{align}\label{system} {\frac{\partial^2\psi}{\partial \zeta ^2}}-{\frac{\partial^2\psi }{\partial z^2}} &+\kappa^2\psi =0\nonumber\\
\psi (\zeta ,0;\kappa )&=f(\zeta ;\kappa )\nonumber\\
{\frac{\partial\psi }{\partial z}}(\zeta ,0;\kappa )&=g(\zeta ;\kappa ).\end{align}
Let $J_n$ be Bessel's function of the first kind of order $n\in \Zb$, defined by 
\begin{equation}\label{Besselfunction} J_n(x)=\int_0^{\pi} \cos\bigl(n\theta -x\sin\theta \bigr) {\frac{d\theta}{\pi}}\qquad (x\in\Cb) \end{equation}
and let $K_0$ be the modified Bessel function of the second kind 
\begin{equation}K_0(w)=\int_0^\infty e^{-w\cosh t}\, dt\qquad (w\in \Cb )\end{equation}
such that $K_0(w)=J_0(iw)$. 
\vskip.1in
\begin{lem}\label{lemRiemannVolterra} Let \begin{align}\psi (\zeta ,z;\kappa )=&{\frac{1}{2}}\Bigl( f(\zeta +z;\kappa )+f(\zeta -z;\kappa )+{\frac{1}{2}}\int_{\zeta -z}^{\zeta +z} g(\xi ;\kappa )J_0\bigl(\kappa \sqrt{(\xi -\zeta )^2-z^2}\bigr)\, d\xi\nonumber\\
{}&\quad +\kappa^2z\int_{\zeta -z}^{\zeta +z} f(\xi ;\kappa ){\frac{J_0'( \kappa \sqrt{ (\xi -\zeta )^2-z^2})}{\sqrt{ (\xi -\zeta )^2-z^2}}} \, d\xi .\end{align}
\begin{enumerate}[(i)]
\item Then $\psi$ satisfies the system (\ref{system}).
\item For $\beta=0$, the differential equation in (\ref{system}) reduces to the wave equation, and the solution is given by
\begin{equation} \psi (\zeta ,z;0) ={\frac{1}{2}} \Bigl( f(\zeta -z;0)+f(\zeta +z;0)\Bigr) +{\frac{1}{2}}\int_{\zeta -z}^{\zeta +z} g(\xi ;0)\, d\xi .\end{equation}
\end{enumerate}
\end{lem}
\begin{proof} (i) Observe that $G(x,y)=J_0(\sqrt{(x-\xi )(y-\eta )})$ satisfies
\begin{equation} {\frac{\partial^2G}{\partial x\partial y}}=-{\frac{1}{4}} G(x,y).\end{equation}
Then by the Riemann-Volterra method for the wave equation, as in \cite{S}  page 226, a solution to the preceding system is given by (\ref{system}).
To change from $\beta =\kappa^2>0$ to $\beta=\kappa^2<0$, we replace
\begin{equation}J_0\bigl(\kappa\sqrt{(\xi -\zeta )^2-z^2}\bigr)=K_0\bigl(\kappa\sqrt{z^2-(\xi -\zeta )^2}\bigr) .\end{equation}

To interpret this formula geometrically, we make a change of variables $z=u+v$ and $\xi-\zeta=2\sqrt{uv}\cos t$. Then $z^2-(\xi-\zeta )^2= u^2+v^2-2uv\cos t$, as in the cosine formula for plane trigonometry, so 
\begin{equation} J_0(\lambda u)J_0(\lambda v)=\int_0^{2\pi} J_0\bigl(\lambda \sqrt{u^2+v^2-2uv\cos t}\Bigr) {\frac{dt}{2\pi}}.\end{equation}
(ii) For $\beta =0,$ we have the wave equation, and the stated solution is a particular case of D'Alembert's formula.
\end{proof}
We return to general $\beta\in \mathbb{R}$, and observe that
\begin{equation} \Phi (z, \zeta y;\kappa)=\psi (z, \zeta ;\kappa )e^{-\kappa^2y}\end{equation}
as a function of $y$ is the Fourier transform in the $\omega$-variable of 
\begin{equation}\check{\Phi} (z,\zeta, \omega; \kappa )={\frac{\kappa^2}{\pi}}{\frac{\psi (z,\zeta; \kappa)}{ \omega^2+\kappa^4}}\end{equation}
Then the Gelfand-Levitan equation
\begin{equation}0=\check{\Phi}(z, \zeta ,\omega ;\kappa)+\check{K}(z, \zeta ,\omega ;\kappa )+\int_z^\infty \int_{-\infty}^\infty \check{K}(z, \eta ,\omega-\nu ; \kappa )\check{\Phi} (\eta, \zeta, \nu ;\kappa )d\nu d\eta \end{equation}
has Fourier transform 
\begin{equation}0={\Phi}(z, \zeta ,y;\kappa )+{K}(z, \zeta ,y;\kappa )+\int_z^\infty {K}(z, \eta ,y; \kappa ){\Phi} (\eta, \zeta, y ;\kappa )d\eta , \end{equation}
and we can solve this as above.
Suppose that $c:(a_1, a_2)\rightarrow \mathbb{C}$ is an integrable function and 
\begin{equation}\Phi (z,\zeta ,y)=\int_{a_1}^{a_2} e^{\kappa^2y}c(\kappa )\psi (z, \zeta ;\kappa )\, d\kappa .\end{equation}
Here we choose the continuous-time linear systems $(-A,B(\zeta ),C)$ with state space $H=L^2((0, \infty )\times (a_1,a_2); \mathbb{C}^{2\times 1})$ and input and output space $H_1=\mathbb{C}$ with operators
 \begin{align}\label{zetasystem} -A&=\begin{bmatrix}0&-I\\ L_0&0\end{bmatrix};\nonumber\\
 B(\zeta )\alpha &=\begin{bmatrix} f(t+\zeta ;\kappa )\\ g(t+\zeta ;\kappa )\end{bmatrix}\alpha \qquad (\alpha\in \mathbb{C});\nonumber\\
 C:\begin{bmatrix} h_1(t,\kappa )\\ h_2(t,\kappa )\end{bmatrix} &\mapsto \int_{a_1}^{a_2} c(\kappa )h_1(0,\kappa )\, d\kappa\qquad (h_1\in H^1((0,\infty );L^2(a_1,a_2);\mathbb{C}))).\end{align}
\vskip.1in
\begin{prop}\label{propGelfandLevitanKP} Let $K(z,\zeta ,y)$ be the solution of the Gelfand-Levitan equation (\ref{GLMforBessel}) that corresponds (\ref{zetasystem}), and let
\begin{equation}\label{diagonalderivative} u(z;y)=-2{\frac{d}{dz}} K(z,z,y).\end{equation}
Then
\begin{equation}u(z;y)=-2{\frac{d^2}{dz^2}}\log\det (I+\Theta_z^\dagger\Xi_z)\end{equation}
 and
\begin{equation}\beta {\frac{\partial K}{\partial y}}(z,\zeta ,y)+{\frac{\partial^2K}{\partial z^2}}(z,\zeta ,y)-{\frac{\partial^2K}{\partial \zeta^2}}(z,\zeta ,y)=u(z,y) K(z,\zeta ,y).\end{equation}
 \end{prop}
 \begin{proof} The impulse response function for the linear system is 
 \begin{equation}\Phi (z, \zeta ,y)=Ce^{-zA}B(\zeta )= \int_{a_1}^{a_2}e^{\kappa^2y}\psi (z,\zeta ;\kappa )c(\kappa )\, d\kappa .\end{equation}
Hence we have 
\begin{equation} \beta {\frac{\partial\Phi}{\partial y}}(z,\zeta ,y)+{\frac{\partial^2\Phi}{\partial z^2}}(z,\zeta ,y)-{\frac{\partial^2\Phi}{\partial \zeta^2}}(z,\zeta ,y)=0\end{equation}
and
\begin{equation}\label{GLMforBessel} \Phi (z,\zeta ,y)+K(z,\zeta ,y)+\int_z^\infty K(z,\eta ,y)\Phi (\eta ,\zeta ,y)\, d\eta =0\end{equation}
which together imply that
\begin{align} -2&{\frac{d}{dz}}K(z,z,y) \Phi (z,\zeta ,y)+
\beta {\frac{\partial K}{\partial y}}(z,\zeta ,y)+{\frac{\partial^2K}{\partial z^2}}(z,\zeta ,y)-{\frac{\partial^2K}{\partial \zeta^2}}(z,\zeta ,y)\nonumber\\
&+\int_z^\infty \Bigl(\beta {\frac{\partial K}{\partial y}}(z,\eta ,y)+{\frac{\partial^2K}{\partial z^2}}(z,\eta ,y)-{\frac{\partial^2K}{\partial \eta^2}}(z,\eta ,y)\Bigr) \Phi (\eta, \zeta ,y)\, d\eta =0,\end{align}
hence
\begin{equation}\beta {\frac{\partial K}{\partial y}}(z,\zeta ,y)+{\frac{\partial^2K}{\partial z^2}}(z,\zeta ,y)-{\frac{\partial^2K}{\partial \zeta^2}}(z,\zeta ,y)=\Bigl(-2{\frac{d}{dz}}K(z,z,y) \Bigr)K(z,\zeta ,y).\end{equation}
\end{proof}
\vskip.1in
\begin{lem}\label{lemVolterra} With
\begin{equation}L_0=-\Bigl(\beta {\frac{\partial}{\partial y}}+{\frac{\partial^2}{\partial z^2}}\Bigr), \qquad 
L_u=-\Bigl(\beta {\frac{\partial}{\partial y}}+{\frac{\partial^2}{\partial z^2}}\Bigr)+u(z,y),\end{equation}
the Volterra-type operator $I+K$ satisfies 
\begin{equation} L_u(I+K)=(I+K)L_0.\end{equation}
\end{lem}
\begin{proof}
Let $\theta =(I+K)\phi$, or more explicitly
\begin{equation} \theta (z,y)=\phi (z,y)+\int_z^\infty K(z, \eta ,y)\phi (\eta ,y)\, d\eta;\end{equation}
then we have the identity

\begin{align}\Bigl(\beta {\frac{\partial}{\partial y}}+{\frac{\partial^2}{\partial z^2}}\Bigr) \theta (z,y)&=\Bigl(\beta {\frac{\partial}{\partial y}}+{\frac{\partial^2}{\partial z^2}}\Bigr) \phi (z,y)+u(z)\int_z^\infty K(z,\eta ,y)\phi (\eta ,y)\, d\eta\nonumber\\
&\quad +\int_z^\infty K(z, \eta ,y) \Bigl(\beta {\frac{\partial}{\partial y}}+{\frac{\partial^2}{\partial \eta^2}}\Bigr) \phi (\eta ,y)\, d\eta .\end{align}
so $L_u\theta =(I+K)L_0\phi$.
\end{proof}

In particular, if 
\begin{equation}-\Bigl(\beta {\frac{\partial}{\partial y}}+{\frac{\partial^2}{\partial z^2}}\Bigr) \phi (z,y)=\kappa^2\phi (z,y),\end{equation}
then
\begin{equation}-\Bigl(\beta {\frac{\partial}{\partial y}}+{\frac{\partial^2}{\partial z^2}}\Bigr) \theta  (z,y)+u(z,y)\theta (z,y)=\kappa^2\phi (z,y).\end{equation}


\section{A differential ring for KP}\label{S:diffring}

\indent From $\Phi$, we now obtain a solution of the linearized $KP$ equation in the standard form. We consider the semi-additive family of kernels
\begin{equation} p(x,y,z,\zeta )=\Phi (z+x,\zeta +x; y)\end{equation}
such that 
\begin{equation} {\frac{\partial p}{\partial x}}={\frac{\partial p }{\partial z}}+{\frac{\partial p}{\partial \zeta }},\end{equation}
and 
\begin{equation}{\frac{\partial^2 p}{\partial z^2}}-{\frac{\partial^2 p}{\partial \zeta^2}}+\beta {\frac{\partial p}{\partial y}}=0.\end{equation}
For each real $x$, there is an operator $P_x$ given by
\begin{equation} P_xh(y,z)=\int_{0}^\infty p(x,y,z,\zeta )h (y,\zeta )d\zeta \end{equation}
which is a multiplication operator as a function in the $y$ variable and an integral operator in $\zeta$.
Consider a family of linear systems $\Sigma_{(y,t)}=(-A, B(y), C(t))$ with impulse response function $\Phi (z, \zeta ; y,t)=C(t)e^{-(z+\zeta )A}B(y)$; then
\begin{equation}\label{RoperatorforKP}R_{x;y,t}=\int_x^\infty e^{-\zeta A}B(y)C(t)e^{-\zeta A}\, d\zeta \end{equation}
gives 
\begin{equation}K(z,\zeta ;y,t)=-C(t)e^{-zA}(I+R_{z;y,t})^{-1}e^{-\zeta A}B(y)\end{equation}
which satisfies
\begin{equation} \Phi (z+\zeta ;y,t)+K(z, \zeta ;y,t)+\int_z^\infty K(z, \eta ;y,t)\Phi (\eta +\zeta ;y,t)\, d\eta =0.\end{equation}

\indent We now introduce a version of the P\"oppe bracket operation, that is suited to the $KP$ equation. The starting point is Lyapunov's identity
\begin{equation}{\frac{dR_{x;y,t}}{dx}}=-AR_{x;y,t}-R_{x;y,t}A=-e^{-zA}e^{-y\Delta}BCe^{-zA}.\end{equation}
\vskip.1in
\begin{defn}\label{defdiffringKP}Let $\mathcal{A}$ be the complex algebra formed by linear combinations of products of $I,A,\Delta ,F$, and let ${\frac{d}{dz}}$ and ${\frac{d}{dy}}$ be derivations on $\mathcal{A}$ such that 
\begin{equation}{\frac{dA}{dy}}={\frac{dA}{dz}}=0, \qquad{\frac{d\Delta}{dy}}={\frac{d\Delta}{dz}}=0,\end{equation}
\begin{equation}{\frac{dF}{dx}}=AF+FA-2FAF, \qquad {\frac{dF}{dz}}=F\Delta (I-F).\end{equation}
We also introduce the associative product $\ast$ on ${\mathcal{A}}$ by
\begin{equation} X\ast Y=X(AF+FA-2FAF)Y\end{equation}
and the differential expressions
\begin{equation}D_zX=(A-2AF)X+{\frac{dX}{dz}} +X(A-2FA), \quad D_y X=\Delta (I-F)X +{\frac{dX}{dy}}-XF\Delta ;\end{equation} 
the asymmetry here is intentional. Then we let
\begin{equation} \lfloor X\rfloor = Ce^{-zA}F_zXF_ze^{-zA}e^{-y\Delta }B\qquad (y,z>0).\end{equation}
\end{defn}
\vskip.1in
\begin{prop}\label{propdiffringKP} Suppose further that $A\Delta=\Delta A.$ Then $({\mathcal{A}}, D_z,D_y, \ast )$ is a differential ring, and $\lfloor \cdot \rfloor$ is a differential ring homomorphism in the sense that 
\begin{enumerate}[(i)]
\item $\lfloor X\ast Y\rfloor =\lfloor X\rfloor \lfloor Y\rfloor$;
\item $ {\frac{d}{dz}}\lfloor X\rfloor =\lfloor D_zX\rfloor$;
\item ${\frac{d}{dy}}\lfloor X\rfloor =\lfloor D_yX\rfloor$  for all $X,Y\in {\mathcal{A}}$.
\end{enumerate}
\end{prop}
\begin{proof}
(i) This is a direct calculation as in \cite[Theorem 4.4]{BlowerNew} and the key step is at the right bracket $\rfloor$, where $\ast$ is replaced by 
\begin{equation}AF_z+F_zA-2F_zAF_z={\frac{dF_z}{dz}}=F_ze^{-zA}e^{-y\Delta}BCe^{-zA}F_z.\end{equation}
(ii) Note that at the left bracket $\lfloor$ we have
\begin{equation} {\frac{d}{dz}}Ce^{-zA}F=Ce^{-zA}\bigl( -AF+AF+FA-2FAF\bigr)=Ce^{-zA}F\bigl( A-2AF\bigr).\end{equation}
(iii) At the right bracket $\rfloor$ we have the $y$ derivative
\begin{equation}{\frac{d}{dy}} Fe^{-zA}e^{-y\Delta}B= \bigl(F\Delta (I-F)-F\Delta \bigr)e^{-zA}e^{-y\Delta}B=-F\Delta Fe^{-zA}
e^{-y\Delta}B.\end{equation}
\end{proof}
For the family of linear systems $(-A, e^{-y\Delta }B,C)$, we have
\begin{equation}R_{z,y}=\int_z^\infty e^{-\zeta A}e^{-y\Delta}BCe^{-\zeta A}\, d\zeta,\end{equation}
then we define $F_{z,y}=(I+R_{z,y})^{-1}$. Also, we have
\begin{equation}\tau (z,y)=\det (I+R_{z,y}),\end{equation}
and 
\begin{equation}{\frac{\partial}{\partial z}}\log \tau (z,y)=K(z,z,y).\end{equation}
Hence the second-order partial derivatives of $\log\tau$ satisfy
\begin{align}\label{uwbrackets} u(z,y)&=-2{\frac{\partial K}{\partial z}}(z,z,y)=-4\lfloor A\rfloor \nonumber\\
w(z,y)&={\frac{-3\beta}{2}}{\frac{\partial K}{\partial y}}(z,z,y)={\frac{-3\beta}{2}}\lfloor \Delta \rfloor ,\end{align} 
by calculations as in Proposition \ref{propdiffringKP}. With $K(x,z)=-Ce^{-xA}(I+R_x)^{-1}e^{-zA}B$, we reconcile the brackets $[\cdot ]$ with $\lfloor\cdot \rfloor$, as in
\[ u(z,y)=-2{\frac{d}{dz}}\bigl[K\bigr]_{z,z}=2{\frac{d}{dz}}\bigl\lfloor I+R\bigr\rfloor_z=-4\bigl\lfloor A\rfloor_z.\]
\par

\section{Solution of KP}\label{S:solutionKP}

Suppose that there is another parameter $t$ such that
\begin{equation}\alpha {\frac{\partial \Phi}{\partial t}}+{\frac{\partial^3\Phi}{\partial z^3}}+{\frac{\partial^3\Phi}{\partial \zeta^3}}=0.\end{equation}
For the $KP$ equation, we adjust the choice of $L_u$ to  
\begin{equation}L_0=-\Bigl(\beta {\frac{\partial}{\partial y}}+{\frac{\partial^2}{\partial z^2}}\Bigr), \qquad  L_u=-\Bigl(\beta {\frac{\partial}{\partial y}}+{\frac{\partial^2}{\partial z^2}}\Bigr)+u(z,y,t)\end{equation}
to accommodate the extra variable $t$, nevertheless, we have $L_u(I+K)=(I+K)L_0$, as in Lemma \ref{lemVolterra}.
As in (\ref{uwbrackets}), we define
\begin{equation}w(z,y,t) ={\frac{-3\beta}{2}}{\frac{\partial K}{\partial y}}(z,z,y,t);\end{equation}
and recall that 
\begin{equation}u(z,y,t)=-2{\frac{\partial K}{\partial z}}(z,z,y,t).\end{equation}
Let 
\begin{equation}M_0 \theta =\alpha{\frac{\partial \theta}{\partial t}}+{\frac{\partial^3\theta}{\partial z^3}}, \end{equation}
\begin{equation}M_u \theta =\alpha{\frac{\partial \theta}{\partial t}}+{\frac{\partial^3\theta}{\partial z^3}}-{\frac{3}{2}}u{\frac{\partial \theta}{\partial z}} -{\frac{3}{4}}{\frac{\partial u}{\partial z}}\theta +w\theta.\end{equation}
\vskip.1in
\begin{prop}\label{propKPII} The transformation $X\mapsto (I+K)^{-1}X(I+K)$ takes the commuting pair $(L_0,M_0)$ to a commuting pair $(L_u,M_u)$, and $u$ satisfies $KPII$.
\end{prop}
\begin{proof}
From the integral equation, 
\begin{align}\label{dPhidz}0=&{\frac{\partial\Phi}{\partial z}}(z,\zeta, y,t)+{\frac{\partial K}{\partial z}}(z, \zeta ,y,t)+\int_z^\infty{\frac{\partial K}{\partial z}}(z, \eta, y,t)\Phi (\eta, \zeta, y,t)\, d\eta\nonumber\\
&-K(z,z,y,t)\Phi (z,z,y,t),\end{align}
and by calculating further derivatives,  we find
\begin{align} 0=&\alpha {\frac{\partial\Phi}{\partial t}}(z,\zeta, y,t)+{\frac{\partial^3\Phi}{\partial z^3}}(z,\zeta, y,t)+ {\frac{\partial^3\Phi }{\partial \zeta^3}}(z, \zeta , y,t)\nonumber\\
&+ \alpha{\frac{\partial K}{\partial t}}(z,\zeta, y,t)+{\frac{\partial^3K}{\partial z^3}}(z,\zeta, y,t)+ {\frac{\partial^3K}{\partial \zeta^3}}(z, \zeta , y,t)\nonumber\\
&+\int_z^\infty \Bigl( \alpha{\frac{\partial K}{\partial t}}(z,\eta, y,t)+{\frac{\partial^3K}{\partial z^3}}(z,\eta, y,t)+ {\frac{\partial^3K }{\partial \eta^3}}(z, \eta , y,t)\Bigr) \Phi (\eta,\zeta, y,t)\, d\eta\nonumber\\
&+\Bigl( -2{\frac{d}{dz}}K(z,z,y,t)-{\frac{\partial K}{\partial z}}(z,z,y,t)-{\frac{\partial K}{\partial\zeta}}(z,z,y,t)\Bigr){\frac{\partial \Phi}{\partial z}}(z, \zeta ,y,t)\nonumber\\
&+\Bigl(-2{\frac{d^2K}{dz^2}}(z,z,y,t)-{\frac{d}{dz}}{\frac{\partial K}{\partial z}}(z,z,y,t)-{\frac{\partial^2K}{\partial z^2}}(z,z,y,t)+{\frac{\partial^2 K}{\partial \zeta^2}}(z,z,y,t)\Bigr)\Phi (z,\zeta ,y,t),\end{align}
where the coefficient of $\Phi$ in the final line is 
\begin{align}{\frac{-3}{2}}&\Bigl( {\frac{\partial}{\partial z}}+{\frac{\partial}{\partial \zeta}}\Bigr)^2K(z,z,y,t)-{\frac{3}{2}}\Bigl( {\frac{\partial^2K}{\partial z^2}}(z,z,y,t)-{\frac{\partial^2K}{\partial\zeta^2}}(z,z,y,t)\Bigr)\nonumber\\
&={\frac{3}{4}}{\frac{\partial u}{\partial z}}(z,y,t)+ {\frac{3\beta}{2}}{\frac{\partial K}{\partial y}}(z,z,y,t)-{\frac{3}{2}}u(z,y,t)K(z,z,y,t),\end{align}
while the coefficient of ${\frac{\partial \Phi}{\partial z}}$ in the preceding line is $(3/2)u(z,y,t)$
Then by adding $(3/2)u(z,y,t)$ times (\ref{dPhidz}), we have
\begin{align}
0=&\alpha{\frac{\partial K}{\partial t}}(z,\zeta, y,t)+{\frac{\partial^3K}{\partial z^3}}(z,\zeta, y,t)+ {\frac{\partial^3K}{\partial \zeta^3}}(z, \zeta , y,t)-{\frac{3}{2}}u(z,y,t){\frac{\partial K}{\partial z}}(z,\zeta ,y,t)\nonumber\\
&+\Bigl( {\frac{3}{4}}{\frac{\partial u}{\partial z}}(z,y,t)+ {\frac{3\beta}{2}}{\frac{\partial K}{\partial y}}(z,z,y,t)\Bigr) \Phi (z, \zeta ,y,t)\nonumber\\
&+\int_z^\infty \Bigl( \alpha{\frac{\partial K}{\partial t}}(z,\eta, y,t)+{\frac{\partial^3K}{\partial z^3}}(z,\eta, y,t)+ {\frac{\partial^3K}{\partial \eta^3}}(z, \zeta , y,t)\nonumber\\
&\qquad -{\frac{3}{2}}u(z,y,t){\frac{\partial K}{\partial z}}(z,\eta ,y,t)\Bigr) \Phi (\eta, \zeta ,y,t)\, d\eta
\end{align}
hence we have the differential equation
\begin{align}\Bigl( \alpha {\frac{\partial K}{\partial t}}+{\frac{\partial^3K}{\partial z^3}}&+{\frac{\partial^3 K}{\partial\zeta^3}}-{\frac{3}{2}} u{\frac{\partial K}{\partial z}}\Bigr)(z,\zeta, y,t)\nonumber\\
&={\frac{3}{4}}{\frac{\partial u}{\partial z}}K(z,\zeta, y,t)+{\frac{3\beta }{2}}{\frac{\partial K}{\partial y}}(z,z,y,t) K(z,\zeta ,y,t).\end{align}
The final term involves $w(z,t,y)$, which is so chosen that $M_u(I+K)=(I+K)M_0$, as one verifies by similar computations; hence 
\begin{equation}[M_u,L_u](I+K)=(I+K)[M_0,L_0]\end{equation}
where $M_0$ and $L_0$ have constant coefficients, hence
\begin{equation}[M_u,L_u]=0\end{equation}
so $L_u$ and $M_u$ commute. The condition on $w$ is that 
\begin{align} \beta {\frac{\partial w}{\partial y}}&=-{\frac{1}{4}}{\frac{\partial^3u}{\partial z^3}}+{\frac{3}{2}}u{\frac{\partial u}{\partial t}} -\alpha {\frac{\partial u}{\partial t}},\nonumber\\
{\frac{\partial w}{\partial z}}&={\frac{3\beta }{4}}{\frac{\partial u}{\partial y}}.\end{align}
The second of these follows from the choice of $w$. The equality of mixed partials with
\begin{equation}{\frac{\partial^2w}{\partial z\partial y}}={\frac{3\beta}{4}}{\frac{\partial^2u}{\partial y^2}}\end{equation}
is equivalent to the $KP$ system of partial differential equations (\ref{KPII}). We can also express this as an evolution equation
\begin{equation}  
\Bigl( \alpha{\frac{\partial u}{\partial t}}+{\frac{1}{4}}{\frac{\partial^3u}{\partial z^3}}-{\frac{3}{2}}u{\frac{\partial u}{\partial z}}\Bigr)+
\int{\frac{3\beta^2}{4}}{\frac{\partial^2u}{\partial y^2}}dz
=0;\end{equation}
indeed, this is the form in which the equation is usually solved.
\end{proof}

\section {Examples and Remarks}\label{S:exrem}

In Propositions \ref{propGelfandLevitan} and \ref{propKPII}, we use a trace-class $R_x$ that arises as the product of Hilbert-Schmidt operators. This is a more stringent hypothesis than admissibility, as in \ref{defadmissible}, so in this
section we give conditions for various linear systems to produce operators in trace ideals. Let $H=L^2(\mathbb{R}; \mathbb{C})$ and $\mathcal{D}(A)=\{ f\in H: vf(v)\in H\}$, 
let ${\mathcal{D}}(\Delta )=\{ f\in H: v^3f(v)\in H\}$ which are themselves Hilbert spaces for the appropriate graph norms; let $b,c\in H\cap L^\infty (\mathbb{R}; \mathbb{C})$; then introduce bounded linear operators
\begin{align}\label{Rtraceexample} A:\mathcal{D}(A)\rightarrow H:\quad Af(v)&=-ivf(v)\qquad (f\in \mathcal{D}(A))\nonumber\\
B:\mathbb{C}\rightarrow H:\qquad B\beta &=b(v)\beta \qquad (\beta\in {\mathbb{C}})\nonumber\\
C:H\rightarrow \mathbb{C}:\qquad Cf&=\int_{-\infty}^\infty f(v)c(v)\, {\frac{dv}{2\pi}}\qquad (f\in H)\nonumber\\
\Delta: {\mathcal{D}}(\Delta )\rightarrow H:\quad \Delta f (v)&=-iv^3f(v)\qquad (f\in {\mathcal{D}}(\Delta )).\end{align}
This example is not covered by Propositions 2.2 and 2.3 of \cite{Blower1}, so we give a special argument to show that $R_x$ exists. 
\vskip.1in
\begin{prop}\label{propHS} Consider the linear system (\ref{Rtraceexample}).\begin{enumerate}[(i)]
\item The integral $R_x=\int_x^\infty e^{-tA}BCe^{-tA}\, dt$ converges in the weak operator topology and gives a solution of Lyapunov's equation (\ref{Lyap}).
\item Suppose further that $b(\nu )/\sqrt{\nu }, c(\nu )/\sqrt{\nu}\in L^2((0, \infty ); \mathbb{C}).$ Then $R_x$ defines a Hilbert-Schmidt operator on $L^2((0, \infty ), \mathbb{C})$.
\item The impulse response function satisfies
\begin{equation}{\frac{\partial \phi}{\partial y}}(z,y)+{\frac{\partial^3\phi}{\partial z^3}}(z,y)=0,\end{equation}
in the the weak sense.
\end{enumerate}
\end{prop}
\vskip.1in
\begin{proof} (i) By the dominated convergence theorem, we have strongly continuous unitary groups $(e^{-zA})_{z\in {\mathbb{R}}}$ and $(e^{-y\Delta })_{y\in {\mathbb{R}}}$. By Plancherel's formula, we have
\begin{align} \int_0^\infty \vert Ce^{-sA}f\vert^2 ds&\leq \int_{-\infty}^\infty \Bigl\vert \int_{-\infty}^\infty c(\nu )f(\nu )e^{is\nu}{\frac{d\nu}{2\pi}} \Bigr\vert^2 ds\nonumber\\
&=\int_{-\infty}^\infty \vert c(\nu )\vert^2\vert f(\nu )\vert^2 d\nu\nonumber\\
&\leq \Vert c\Vert_{L^\infty}^2\Vert f\Vert^2_{L^2}\end{align}
and likewise
\begin{align} \int_0^\infty \vert B^\dagger e^{-sA^\dagger }f\vert^2 ds&\leq \int_{-\infty}^\infty \Bigl\vert \int_{-\infty}^\infty \bar b(\nu )f(\nu )e^{-is\nu}ds\Bigr\vert^2 d\nu\nonumber\\
&=\int_{-\infty}^\infty \vert \bar b(\nu )\vert^2\vert f(\nu )\vert^2 d\nu\nonumber\\
&\leq \Vert b\Vert_{L^\infty}^2\Vert f\Vert^2_{L^2}\end{align}
so the integrals
\begin{equation} \int_0^\infty e^{-sA^\dagger}C^\dagger Ce^{-sA}\, ds, \quad  \int_0^\infty e^{-sA}BB^\dagger e^{-sA^\dagger}\, ds\end{equation}
are convergent in the weak operator topology and define elements of $\mathcal{L}(H)$. Hence $R_x=\int_x^\infty e^{-sA}BCe^{-sA}\, ds$ is also convergent in the weak operator topology.
For $f,h\in \mathcal{D}(A)$, the function $\langle R_xf,h\rangle$ is differentiable with derivative $\langle (-AR_x-R_xA)f,h\rangle$, so Lyapunov's equation holds. (In \cite{MPT}, Lyapunov's equation is also interpreted weakly.)\par
\indent (ii) As an integral operator on $L^2((0, \infty ); \mathbb{C})$ the operator $R_x$ has kernel
\begin{equation}\label{RformulaHowland}{\frac{ c(\kappa )b(k)e^{ix(k+\kappa )}}{2\pi i(k+\kappa )}}\qquad (k, \kappa >0),\end{equation}
which is square integrable.\par
 (iii) Then the impulse response function is 
\begin{equation}\label{impulseresponse} \phi (z,y)=Ce^{-zA}e^{-y\Delta}B=\int_{-\infty}^\infty b(v)c(v) e^{ivz+iv^3y}\, {\frac{dv}{2\pi}}\qquad (y,z>0).\end{equation}
\end{proof}
\vskip.1in

\begin{ex}\label{Airy}

Alternatively, one can assume $b,c\in C^1_b({\mathbb{R}}; \mathbb{C}),$ and interpret  via integration by parts
\begin{equation}\phi (z,y)=i\int_{-\infty}^\infty {\frac{(b(v)c(v))'}{z+3v^2y}} e^{ivz+iv^3y}\, {\frac{dv}{2\pi}}-i\int_{-\infty}^\infty {\frac{ 6vy b(v)c(v)}{(z+3v^2y)^2}} e^{ivz+iv^3y}{\frac{dv}{2\pi}},\end{equation}
where these integral are absolutely convergent. In particular, one can choose $b=c=1$ and obtain the oscillatory integral 
\begin{equation}\phi (z,y)=\int_{-\infty}^\infty e^{ivz+iv^3y}{\frac{dv}{2\pi}}={\frac{1}{(3y)^{1/3}}}{\hbox{Ai}}\Bigl( {\frac{z}{(3y)^{1/3}}}\Bigr)\end{equation}
which is a scaled form of Airy's function.
\end{ex}
\begin{rem} (i) There is an existence theorem for solutions which are periodic in the spatial variables, so $(z,y)\in {\mathbb{R}}^2/2\pi{\mathbb{Z}}^2$, where $u(z,y,0)$ is specified as initial data for a Cauchy problem in $t$; see \cite{Bou}, \cite{DS}, \cite{Hadac}.\par
(ii) Using $R_x$ from (\ref{RformulaHowland}), one can readily prove the identities of \cite[article 5.7]{EMcK} for Fredholm determinant expansions. Whereas Ercolani and McKean show that the tau function satisfies identities consistent with classical theta functions; in section \ref{sec:NumericalSimulations}, we use Fredholm determinants in numerical simulations.\par
(iii) 
For the Clenshaw-Curtis numerical quadrature in section \ref{sec:NumericalSimulations}, it is more convenient to restrict $\phi (z,y)$ to $z\in [-L, L]$ for some large $L>0$, and to use the expansion of $\phi (z,y)$ in Chebyshev polynomials for $z/L\in [-1,1]$. Equivalently, one considers the Fourier cosine expansion of $\phi (L\cos\theta ,y)$ in the $\theta$ variable. Recall that the Chebyshev polynomials of the first kind are characterized  by $T_n(\cos\theta )=\cos (n\theta )$ for $n=1,2, \dots $.
From the standard expansion \cite[17.23]{WW}
\[ e^{iz\sin t}=J_0(z)+2\sum_{n=1}^\infty  \bigl( J_{2n}(z) \cos 2nt +iJ_{2n-1}(z)\sin (2n-1)t\bigr),\]
with $\theta =\pi /2-t$ and $z=L\cos\theta$, we obtain a Fourier cosine expansion
\begin{align}\phi (L\cos\theta ,y)&=\int_{-\infty}^\infty b(\nu )c(\nu )e^{i\nu^3y}J_0(L\nu ){\frac{d\nu }{2\pi}}\nonumber\\
&\quad +2\sum_{n=1}^\infty (-1)^n \int_{-\infty}^\infty b(\nu )c(\nu )e^{i\nu^3y}J_{2n}(L\nu ){\frac{d\nu }{2\pi}}\cos 2n\theta \nonumber\\
&\quad -2i
\sum_{n=1}^\infty (-1)^n \int_{-\infty}^\infty b(\nu )c(\nu )e^{i\nu^3y}J_{2n-1}(L\nu ){\frac{d\nu }{2\pi}}\cos (2n-1)\theta ,\end{align}
where the coefficients involve Bessel functions with integer indices as in (\ref{Besselfunction}).
\end{rem}
\section{Numerical simulations}\label{sec:NumericalSimulations}
We present numerical simulations of solutions to the Kadomtsev--Petviashvili (KP) equation.
We use four different numerical approaches as follows: \smallskip
(1) \emph{GLM solution using Riemann Rule approximation (GLM-RR):} We solve the linear integral Gelfand--Levitan--Marchenko (GLM) equation,
with the coefficients given by semi-additive scattering data representing solutions to the linearised $KP$ equation.
The solution to the linearised equations~\eqref{eq:linearisedKP} can be analytically advanced to  any time $t>0$ and substituted into the GLM equation, which is then solved. 
For this method we use the left-hand Riemann Rule to approximate the integral in the GLM equation,
which is then solved as a large linear system of equations at that time $t>0$;\smallskip
(2) \emph{GLM solution using Clenshaw--Curtis quadrature (GLM-CC):} This is similar to the last method except that we use Clenshaw--Curtis quadrature
to approximate the integral in the GLM equation. Clenshaw--Curtis quadrature is based on Chebyshev polynomial approximation and its use here is inspired
by the approximation method for computing Fredholm determinants developed by Bornemann~\cite{Bornemann}, as we outline next; \smallskip
(3) \emph{Fredholm determinant using Nystr\"om--Clenshaw--Curtis method (Det-CC):} The solution to the $KP$ equation is given in terms of
the second derivative, with respect to $x$, of the logarithm of the $\tau$-function, which can be expressed as the Fredholm determinant of the scattering data.
We use the Nystr\"om--Clenshaw--Curtis method developed by Bornemann~\cite{Bornemann} to approximate Fredholm determinants to very high accuracy;\smallskip
(4) \emph{Direct pseudo-spectral time-stepping approximation (FFT2-exp):} This is a direct, exponential split-step, pseudo-spectral method that utilises 
the fast Fourier transform (FFT) in both the $x$- and the $y$- directions.
It also utilises the window method outlined by Kao and Kodama~\cite{KaoKodama}, to deal with non-periodic boundary conditions.
This method is outlined in detail in Blower and Malham~\cite[App.~B]{BMal2}.\smallskip
We provide further details of these numerical approaches presently. Before doing so, let us outline the direct linearisation approach
we have outlined in the sections \ref{S:KP}, \ref{S:diffring}, \ref{S:solutionKP}, and connect that to the direct linearisation approach given in Blower and Malham~\cite{BMal2}, as well as the expression for the solution via the $\tau$-function. The $KP$ equation for the field $g=g(x,y;t)$ in potential form is given by 
\begin{equation}\label{eq:KPform}
  g_t=g_{xxx}+6g_x^2+3\pa_x^{-1}g_{yy}.
\end{equation}
This corresponds to the Kadomtsev--Petviashvili equation (\ref{KPII}) in the case $\alpha=-1/4$, $\beta=\pm1$ and $u=2\pa_xg$.
The linearised form of the $KP$ equation for $p=p(t)$ is given by
\begin{equation}\label{eq:linearisedKP}
  p_t=p_{xxx}+3\pa_x^{-1}p_{yy}.
\end{equation}
The linear integral GLM equation is of the form,
\begin{equation}\label{eq:GLM}
P=G(\id-P),
\end{equation}
for the solution operator $G$, or equivalently its kernel, $g$.
Here $P$ is the scattering operator associated with the kernel function solution $p$ solving~\eqref{eq:linearisedKP}.
We assume that $P$ is a Hilbert--Schmidt valued integral operator on $(-\infty,0]$ with kernel of the form,
\begin{equation}\label{eq:semiaddtitiveform}
p=p(z+x,\zeta+x;y,t).
\end{equation}
This is the semi-additive form first introduced by P\"oppe~\cite{PKP} with $z,\zeta\in(-\infty,0]$ the primary variables parametrising the operator $P$,
while $x,y\in\R$ and $t\geqslant0$ are regarded as additional parameters.
This form guarantees that the first of the following two constraints on $p$ is automatically satisfied:
\begin{subequations}\label{eq:constraints}
\begin{align}
  p_x&=p_z+p_{\zeta},\label{eq:xconstraint}\\
  p_y&=p_{zz}-p_{\zeta\zeta}.\label{eq:yconstraint}
\end{align}
\end{subequations}
These two constraints arise as the first two equations in the $KP$ hierarchy; see for example~\cite{BM-KP-descents}.
While our semi-additive assumption for $p$ means that \eqref{eq:xconstraint} is satisfied, we henceforth assume that $p$ satisfies~\eqref{eq:yconstraint} as well.
Using the constraints~\eqref{eq:constraints}, the linearised $KP$ equation~\eqref{eq:linearisedKP}, has the alternative formulation,
\begin{equation}\label{eq:linearisedKPalt}
  p_t=4\,\bigl(p_{zzz}-p_{\zeta\zeta\zeta}\bigr).
\end{equation}
Note that if $P$ is Hilbert--Schmidt valued on $(-\infty,0]$,
then the solution operator $G$ to the GLM equation \eqref{eq:GLM} is Hilbert--Schmidt valued on $(-\infty,0]$ as well--see Blower and Malham~\cite[Lemma~7]{BMal2}---and
the kernel $g=g(z,\zeta;x,y,t)$ of $G$ is square-integrable. Further, $g$ automatically adopts any regularity that $p$ possesses. 
The solution $g=g(z,\zeta;x,y,t)$ to the GLM equation evaluated at $z=\zeta=0$, i.e. $g=g(0,0;x,y,t)$, satisfies the $KP$ equation~\eqref{eq:KPform}.
See, for example, \cite{BMal2}. Solutions to the $KP$ equation~\eqref{eq:KPform} can thus be generated by solving the
linearised $KP$ equation~\eqref{eq:linearisedKP}, or equivalently here~\eqref{eq:linearisedKPalt}, and the following linear integral equation for $g=g(0,\zeta;x,y,t)$,
\begin{equation}\label{eq:GLMPoppe}
  p(x,\zeta+x;y,t)=g(0,\zeta;x,y,t)-\int_{-\infty}^0g(0,\xi;x,y,t)p(\xi+x,\zeta+x;y,t)\,\rd\xi.
\end{equation}
Further, we also know that $g=g(0,0;x,y,t)$ is given by the trace formula,
\begin{equation}\label{eq:traceformula}
  g(0,0;x,y,t)=\mathrm{trace}\,\bigl((\pa_{\mathfrak{l}}P)V+V(\pa_{\mathfrak{r}}P)\bigr),
\end{equation}
where $V\coloneqq(\id-P)^{-1}$; see \cite[Cor.~12]{BMal2}. 
Here $\pa_{\mathfrak{l}}G$ and $\pa_{\mathfrak{r}}G$ represent the trace-class operators with the respective kernels $\pa_zg$ and $\pa_\zeta g$.
In other words we have, $g(0,0;x,y,t)=\mathrm{trace}\,\bigl((\pa_x P)V\bigr)$ since $\pa_xP=\pa_{\mathfrak{l}}P+\pa_{\mathfrak{r}}P$ using \eqref{eq:xconstraint}.
By a standard calculation, we thus have,
\begin{equation}\label{eq:taufunctionform}
  g(0,0;x,y,t)=-\pa_x\log\mathrm{det}(\id-P).
\end{equation}
The quantity $\tau\coloneqq\mathrm{det}(\id-P)$ is known as the $\tau$-function, as in Definition \ref{semiadditivetau}(ii).
Thus in our numerical methods GLM-RR and GLM-CC outlined above, we solve the linear integral equation~\eqref{eq:GLMPoppe},
respectively using the left-hand Riemann Rule and Clenshaw--Curtis quadrature for the integral on the right-hand side.
Utilising either of these quadrature approximations, generates a linear algebraic system of equations for the solution $g$ at the nodal points $\zeta_n$.
We give more details on this procedure presently. Further, for our numerical method Det-CC,
we use the Nystr\"om--Clenshaw--Curtis method developed by Bornemann~\cite{Bornemann} to approximate the $\tau$-function Fredholm determinant.
Lastly, in this context, since $g$ solves the potential form of the $KP$ equation~\eqref{eq:KPform}, the solution to the $KP$ equation itself is $\pa_xg(0,0;x,y,t)$.

\begin{ex}[One-soliton solution]\label{ex:onesoliton}
Suppose $a$ and $b$ are real constants and $\Lambda\coloneqq a^2-b^2$ and $\Omega\coloneqq 4(a^3+b^3)$. Further, suppose $p=p(z+x,\zeta+x;y,t)$ has the form,
\begin{equation}\label{eq:groundstate}
p=-(a+b)\exp\bigl(a(z+x)+b(\zeta+x)+\lambda y+\Omega t\bigr).
\end{equation}
Then the solution to the GLM equation~\eqref{eq:GLMPoppe}, $g=g(0,\zeta;x,y,t)$, generates the following one-soliton solution to the $KP$ equation,
\begin{equation}\label{eq:onesoliton}
\pa_xg(0,0;x,y,t)=\tfrac14(a+b)^2\mathrm{sech}^2\Theta,
\end{equation}
where $\Theta\coloneqq\frac12\bigl((a+b)x+\Lambda y+\Omega t\bigr)$.
\end{ex}

Before we discuss the implementation of our numerical methods, let us relate the quantities above to those in Sections~\ref{S:KP}, \ref{S:diffring}, \ref{S:solutionKP}.
Suppose the operator $R=R(x,y,t)$, see for example (\ref{Roperator})  or (\ref{RoperatorforKP}), has the form, 
\begin{equation}\label{eq:Rform}
  R(x,y,t)\coloneqq\int_0^{\infty}\mathrm{e}^{-b(\xi+x)}\hat{B}(y)C(t)\mathrm{e}^{-a(\xi+x)}\,\rd\xi.
\end{equation}
Here we have assumed that, the operator $A$ simply represents real multiplication by the constant $a$, and the operator $B=B(y,\zeta)$ has the form,
\begin{equation}\label{eq:Bform}
  B(y,\zeta)=\mathrm{e}^{-b\zeta}\hat{B}(y),
\end{equation}
where $b$ is a real constant. Further, suppose we define the operator from $\mathbb C$ to itself, or function, $\hat{K}=\hat{K}(z,\zeta;x,y,t)$ by,
\begin{equation}\label{eq:Kformspecial}
  \hat{K}(z,\zeta;x,y,t)\coloneqq -C(t)\mathrm{e}^{-a(z+x)}\bigl(\id+R(x,y,t)\bigr)^{-1}\mathrm{e}^{-b(\zeta+x)}\hat{B}(y).
\end{equation}
It is then straightforward to verify that $\hat{p}(z+x,\zeta+x;y,t)\coloneqq C(t)\mathrm{e}^{-a(z+x)-b(\zeta+x)}\hat{B}(y)$ and $\hat{K}=\hat{K}(z,\zeta;x,y,t)$
satisfy the linear integral equation,
\begin{equation}\label{eq:GLMspecial}
  \hat{p}(z+x,\zeta+x;y,t)+\hat{K}(z,\zeta;x,y,t)+\int_0^\infty K(z,\xi;x,y,t)\,\hat{p}(\xi+x,\zeta+x;y,t)\,\rd\xi=0.
\end{equation}
If we make the change of variables,
\begin{subequations}\label{eq:transf}
\begin{align}
\xi\to-\xi,\qquad z\to-z,\qquad &\zeta\to-\zeta,\qquad x\to-x,
\intertext{and set,}
p(z+x,\zeta+x;y,t)&\coloneqq-\hat{p}(-z-x,-\zeta-x;y,t)
\intertext{and}
g(z,\zeta;x,y,t)&\coloneqq\hat{K}(-z,-\zeta;-x,y,t),
\end{align}
\end{subequations}
then we see that $p$ and $g$ satisfy the GLM equation $P=G(\id-P)$, i.e.,
\begin{equation}\label{eq:GLMinfull}
p(z+x,\zeta+x;y,t)=g(z,\zeta;x,y,t)-\int_{-\infty}^0g(z,\xi;x,y,t)p(\xi+x,\zeta+x;y,t)\rd\xi.
\end{equation}
This generates~\eqref{eq:GLMPoppe} when we set $z=0$.
\begin{ex}[One-soliton solution: reprise]\label{ex:onesolitonreprise}
  Recall the one-soliton solution we outlined in Example~\ref{ex:onesoliton}, and in particular the quantities $\Lambda$ and $\Omega$.
  In the context of the operator $R$ and kernels $\hat{p}$ and $\hat{K}$, suppose the kernel of the operator $\hat{B}(y)C(t)$ has the separable form, $\hat{b}(y,z)\,c(t,\zeta)$,
  where $z$ and $\zeta$ are the primary variables. Further suppose that $\hat{b}$ and $c$ have the respective specific forms,
  \begin{equation*}
      \hat{b}(y,z)=B_0\mathrm{e}^{-az+\Lambda y}\qquad\text{and}\qquad c(t,\zeta)=C_0\mathrm{e}^{-b\zeta+\Omega t},
  \end{equation*}
  where the constants $B_0$ and $C_0$ satisfy $B_0C_0=(a+b)^2$. Then the kernel $r=r(z,\zeta;x,y,t)$ of of the operator $R=R(x,y,t)$ has the form, 
  \begin{equation*}
     r(z,\zeta;x,y,t)=(a+b)\mathrm{e}^{-a(z+x)-b(\zeta+x)+\Lambda y+\Omega t}.
  \end{equation*}
  This kernel form matches the one-soliton semi-additive form for $p$ in \eqref{eq:groundstate}---taking into account the transformation~\eqref{eq:transf}.
\end{ex}

Let us now outline in detail the four numerical algorithms we used to compute solutions to the $KP$ equation.
For all four numerical methods, we truncate the $(x,y)\in\R^2$ domain to $[-L_x/2,L_x/2]\times[-L_y/2,L_y/2]$ for sufficiently large domain lengths $L_x>0$ and $L_y>0$.
First, we outline the simple solution method, GLM-RR, that solves the linear integral GLM equation, by using the left-hand Riemann Rule to approximate the integral therein.
Our goal is to compute $g=g(0,0;x,y,t)$.
To achieve this, for given scattering data $p$, and given $x\in[-L_x/2,L_x/2]$, $y\in[-L_y/2,L_y/2]$ and $t\geqslant0$,
we numerically solve the GLM equation~\eqref{eq:GLMPoppe} for $g=g(0,\zeta;x,y,t)$ and then set $\zeta=0$.
In practice we use $N_x$ nodal points $x_n$ in the truncated $x$-domain and $N_y$ nodal points $y_{n^\prime}$ in the truncated $y$-domain.
For each nodal point-pair $(x_n,y_{n^\prime})$ and any given fixed $t\geqslant0$, we numerically solve~\eqref{eq:GLMPoppe} as follows.
Note that $\zeta,\xi\in[-L_x/2,0]$. We use $M/2+1$ nodal points $\zeta_m$ and $\xi_{m^\prime}$ for these variables in $[-L_x/2,0]$, with separation $h=L_x/M$.
We always take $M$ to be even.
Using each of the nodal point $\zeta_m$, we generate the row vector $\widehat{P}$ of length `$M/2+1$' containing the values $p(x_n,\zeta_m+x_n;y_{n^\prime},t)$ for $m\in\{0,1,\ldots,M/2\}$.
Note that $\zeta_{M/2+1}=0$. 
Further, for each of the nodal point pairs $(\zeta_m,\xi_{m^\prime})$, we generate the matrix $\widehat{Q}$ of size $(M/2+1)\times(M/2+1)$ containing the values
$p(\xi_{m^\prime}+x_n,\zeta_m+x_n;y_{n^\prime},t)$ for $m^\prime,m\in(0,1,\ldots,M/2)$.
Suppose that $\widehat{G}$ is the row vector of unknown values $g(0,\zeta_m;x_n,y_{n^\prime},t)$ for $m\in(0,1,\ldots,M/2)$.
We then solve the linear algebraic system,
\begin{equation}\label{eq:GLMRR}
  \widehat{P}=\widehat{G}(I-h\widehat{Q}),
\end{equation}
for $\widehat{G}$, where $I=I_{M/2+1}$ is the $(M/2+1)\times(M/2+1)$ identity matrix.
The matrix product `$h\widehat{G}\widehat{Q}$' with the multiplicative scaling $h$, naturally implements the left-hand Riemann Rule. 
We solve the linear system in~\eqref{eq:GLMRR} using Gaussian elimination.
For a given $t\geqslant0$, the procedure just outlined, is carried out for each $(x_n,y_{n^\prime})$ with $n\in\{0,1,\ldots,N_x\}$ and $n^\prime\in\{0,1.\ldots,N_y\}$.
For each $(x_n,y_{n^\prime})$ we extract the final $(M/2+1)$th component of $\widehat{G}$ which represents an approximation for $g(0,0;x_n,y_{n^\prime},t)$.
This outlines the GLM-RR method.
We compute the solution to the $KP$ equation, namely $\pa_xg(0,0;x_n,y_{n^\prime},t)$, by approximating the derivative via a finite difference.

Second, we outline the GLM-CC method, inspired by the method for computing Fredholm determinants developed by Bornemann~\cite{Bornemann}.
To begin with, this method follows that of the GLM-RR method, we truncate the $(x,y)\in\R^2$ domain in precisely the same way and utilise the nodes $x_n$ and $y_{n^\prime}$ as outlined above.
The difference comes into play with the choice of the $M/2+1$ nodal points $\zeta_m$ and $\xi_{m^\prime}$ in $[-L_x/2,0]$.
Here, we choose these nodal points according to the Clenshaw--Curtis quadrature rule approximation for the integral term in~\eqref{eq:GLMPoppe}.
Indeed, using the notation,
\begin{equation*}
\rho(\xi,\zeta;x_n,y_{n^\prime},t)\coloneqq p(\xi+x_n,\zeta+x_n;y_{n^\prime},t)\quad\text{and}\quad \gamma(\zeta;x_n,y_{n^\prime},t)\coloneqq g(0,\zeta;x_n,y_{n^\prime},t),
\end{equation*}
and then suppressing the implicit $x_n$, $y_{n^\prime}$ and $t$-dependence in $\gamma$ and $\rho$, we approximate,
\begin{equation}\label{eq:CCquadrature}
  \int_{-L_x/2}^0\gamma(\xi)\rho(\xi,\zeta_m)\,\rd\xi\approx\sum_{m^\prime=0}^{M/2+1}w_{m^\prime}\gamma(\xi_{m^\prime})\rho(\xi_{m^\prime},\zeta_m).
\end{equation}
Here $\zeta_m$ and $\xi_{m^\prime}$ are chosen to be the Clenshaw--Curtis nodal points and the $w_{m}$ are the Clenshaw--Curtis quadrature weights.
Clenshaw--Curtis quadrature is based on the expansion of the integrand using Chebyshev polynomials of the first kind. 
The nodal points and weights are given explicitly in Bornemann~\cite[p.~909]{Bornemann}, including a Matlab code for generating them, which in fact, we utilised directly. 
Bornemann points out that alternatively, Gauss--Legendre quadrature could also be used, however, Clenshaw--Curtis quadrature is more efficient.
Thus, for each nodal poin- pair $(x_n,y_{n^\prime})$ and any given fixed $t\geqslant0$, we construct the row vectors $\widehat{P}$ and $\widehat{G}$ and the matrix $\widehat{Q}$,
as above, except now based on the Clenshaw--Curtis nodal points $\zeta_m$ and $\xi_{m^\prime}$, which are not uniformly distributed.
Let $W$ denote the diagonal matrix of Clenshaw--Curtis quadrature weights.
We then solve the linear algebraic system, 
\begin{equation}\label{eq:GLMCC}
  \widehat{P}=\widehat{G}(I-W\widehat{Q}),
\end{equation}
for $\widehat{G}$, again using Gaussian elimination. The Clenshaw--Curtis quadrature approximation is implicit in the matrix product $\widehat{G}W\widehat{Q}$.
As above, for a given $t\geqslant0$ and for each $(x_n,y_{n^\prime})$ with $n\in\{0,1,\ldots,N_x\}$ and $n^\prime\in\{0,1.\ldots,N_y\}$,
we solve the linear system~\eqref{eq:GLMCC} and extract the final $(M/2+1)$th component of $\widehat{G}$ which represents our approximation for $g(0,0;x_n,y_{n^\prime},t)$.
This outlines the GLM-CC method. We again use a finite difference approximation to compute $\pa_xg(0,0;x_n,y_{n^\prime},t)$. 

Third, we now outline the Det-CC solution method based on the method for computing Fredholm determinants developed by Bornemann~\cite{Bornemann}.
Again, initially this follows the set-up of the GLM-CC method, we truncate the $(x,y)\in\R^2$ domain as above and utilise the uniformly distributed nodes $x_n$ and $y_{n^\prime}$.
Our goal here is to compute the quantity `$\mathrm{det}(\id-P)$' in the formula~\eqref{eq:taufunctionform}.
Bornemann~\cite{Bornemann} provides a simple and accurate approximation formula for computing such a determinant based on Clenshaw--Curtis quadrature.
Using the same notation to that outlined above for the GLM-CC method, Bornemann~\cite[p.~890,~894]{Bornemann} suggests we compute,
\begin{equation}\label{eq:DetCC}
\mathrm{det}(\id-P)\approx\det\bigl(I-W^{1/2}\widehat{Q}W^{1/2}\bigr),
\end{equation}
where $W^{1/2}$ is the diagonal matrix of entries consisting of the square-roots of the Clenshaw--Curtis weights.
We compute this approximation for any given $t\geqslant0$ and for each $(x_n,y_{n^\prime})$ with $n\in\{0,1,\ldots,N_x\}$ and $n^\prime\in\{0,1.\ldots,N_y\}$.
This outlines the Det-CC method. The approximate solution to the $KP$ equation can be computed by approximating the partial derivative $\pa_x^2$ of this
determinant approximation using a second order central difference scheme.
A comprehensive error and performance analysis is provided in Bornemann~\cite{Bornemann}.
We remark on this in our implementation in Example~\ref{ex:twosoliton} below.
Fourth, we now outline the exponential split-step pseudo-spectral algorithm we used to directly integrate the $KP$ equations.
For convenience, suppose $A$ denotes the linear $KP$ operator, $A\coloneqq\pa_x^3+3\pa_x^{-1}\pa_y^2$.
Then the exponential split-step we use to integrate the $KP$ equation~\ref{eq:KPform} is given by,
\begin{align*}
    v_\ell&=\exp\bigl(\Delta t\,\mathcal{F}(A)\bigr)\,\hat{u}_\ell,\\
    \hat{u}_{\ell+1}&=v_\ell-\Delta t\,\mathcal{F}\bigl(\mathrm{N}(v_\ell)\bigr),
\end{align*}
where $\mathrm{N}(v)=6\pa_x(\mathcal{F}^{-1}v)^2$.
Here, $\mathcal{F}$ represents the two-dimensional Fourier transform,
and $\hat{u}_\ell$ is the two-dimensional Fourier transform of the approximate solution $u$ to the $KP$ equation at time $t_\ell\in\{0\}\cup\mathbb N$.
The quantity $\mathcal{F}(A)$ represents the Fourier transform of the operator $A$.
In practice, if $2\pi\mathrm{i}k_x/L_x$ and $2\pi\mathrm{i}k_y/L_y$ are the wavenumbers, respectively in the $x$- and $y$- directions, we set,
\begin{equation*}
\bigl(\mathcal{F}(A)\bigr)(k_x,k_y)=(2\pi\mathrm{i}k_x/L_x)^3+3\frac{(2\pi\mathrm{i}k_y/L_y)^2}{(2\pi\mathrm{i}k_x/L_x)+2\pi\delta}.
\end{equation*}
Here, following Klein and Roidot~\cite[p.~3341]{KleinRoidot}, we have approximated the Fourier transform of $\pa_x^{-1}$ by $1/(2\pi\mathrm{i}k_x/L_x+2\pi\delta)$, where $\delta=2^{-52}$.
The initial data is generated by numerically solving the GLM equation, using the GLM-CC method, for the given scattering data $p$ at time $t=0$, as outlined above.
We use a pseudo-spectral algorithm due to its simplicity and efficiency, see, for example, Klein and Saut~\cite{KleinSaut} and Grava, Klein and Pitton~\cite{GKP}.
To deal with the fact that the solutions we compute are not periodic, we use the `window method' employed by Kao and Kodama~\cite{KaoKodama}.
Precise details on this method can be found in Blower and Malham~\cite[App.~B]{BMal2}.
The FFT2-exp method we have employed, thus combines the exponential split-step pseudo-spectral algorithm with the `window method'.

\begin{figure*}
  \begin{center}
   \mbox{\includegraphics[width=7cm,height=6cm]{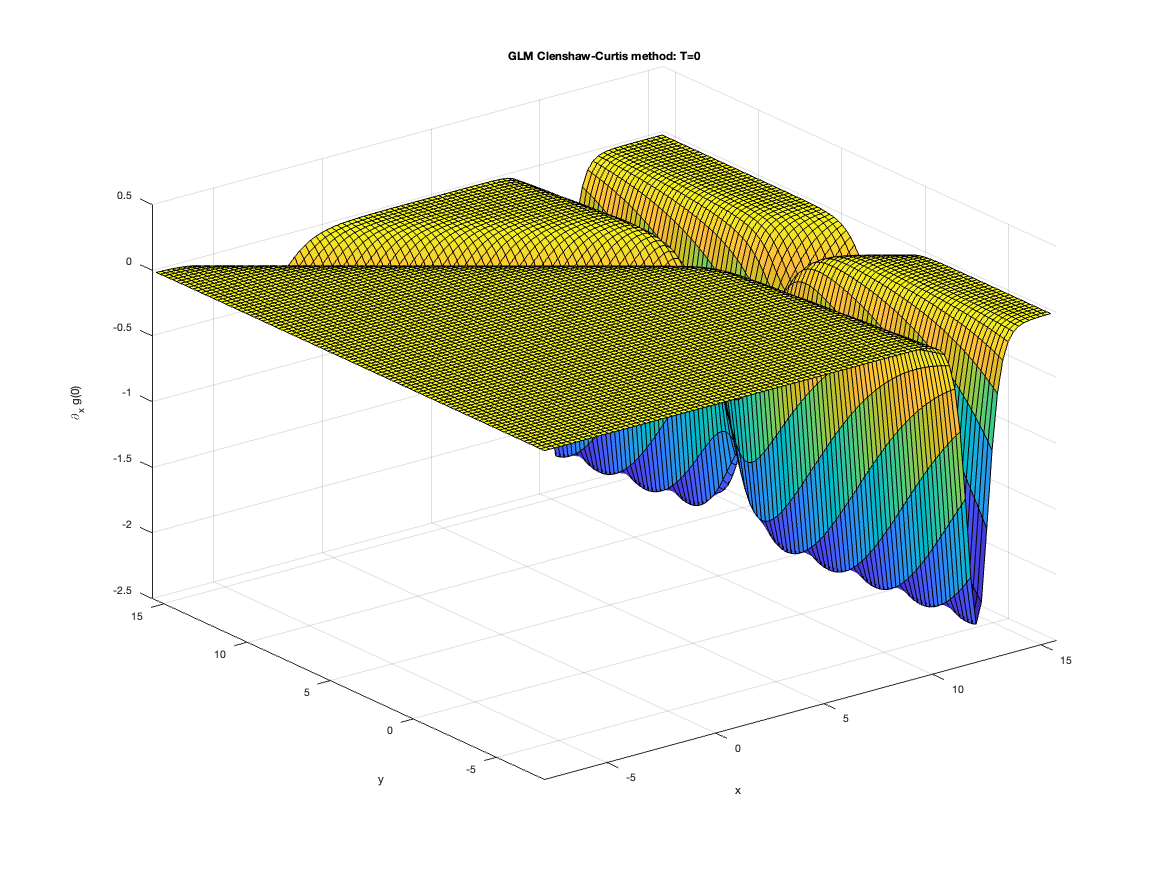}
    \includegraphics[width=7cm,height=6cm]{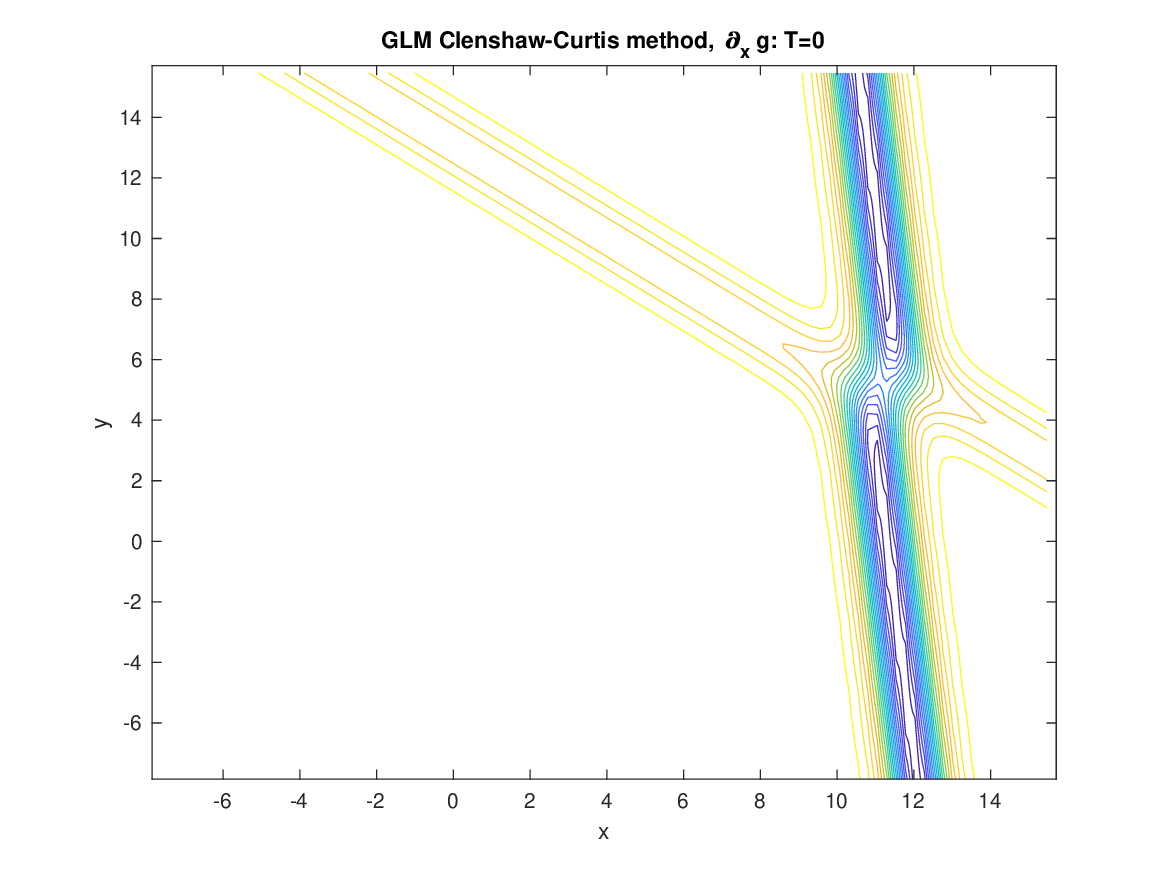}}\\ 
    \mbox{\includegraphics[width=7cm,height=6cm]{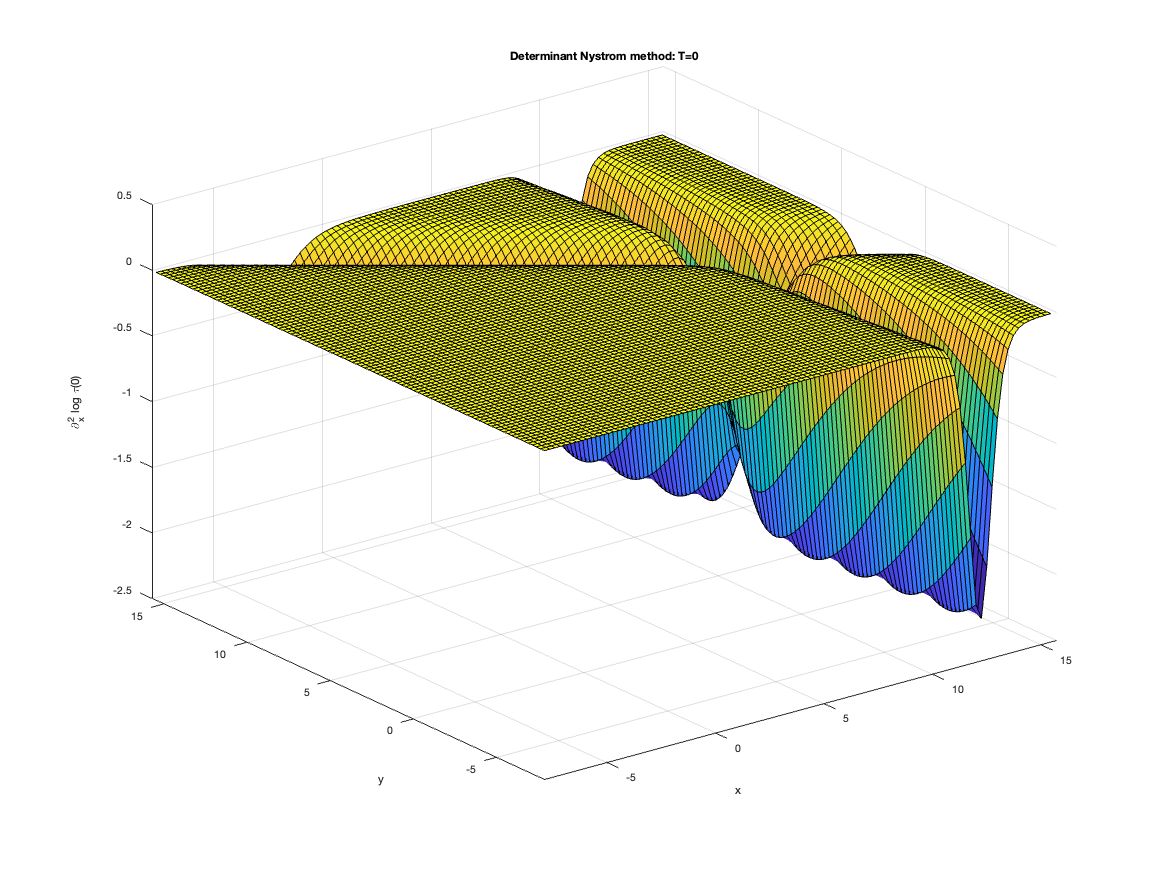}
      \includegraphics[width=7cm,height=6cm]{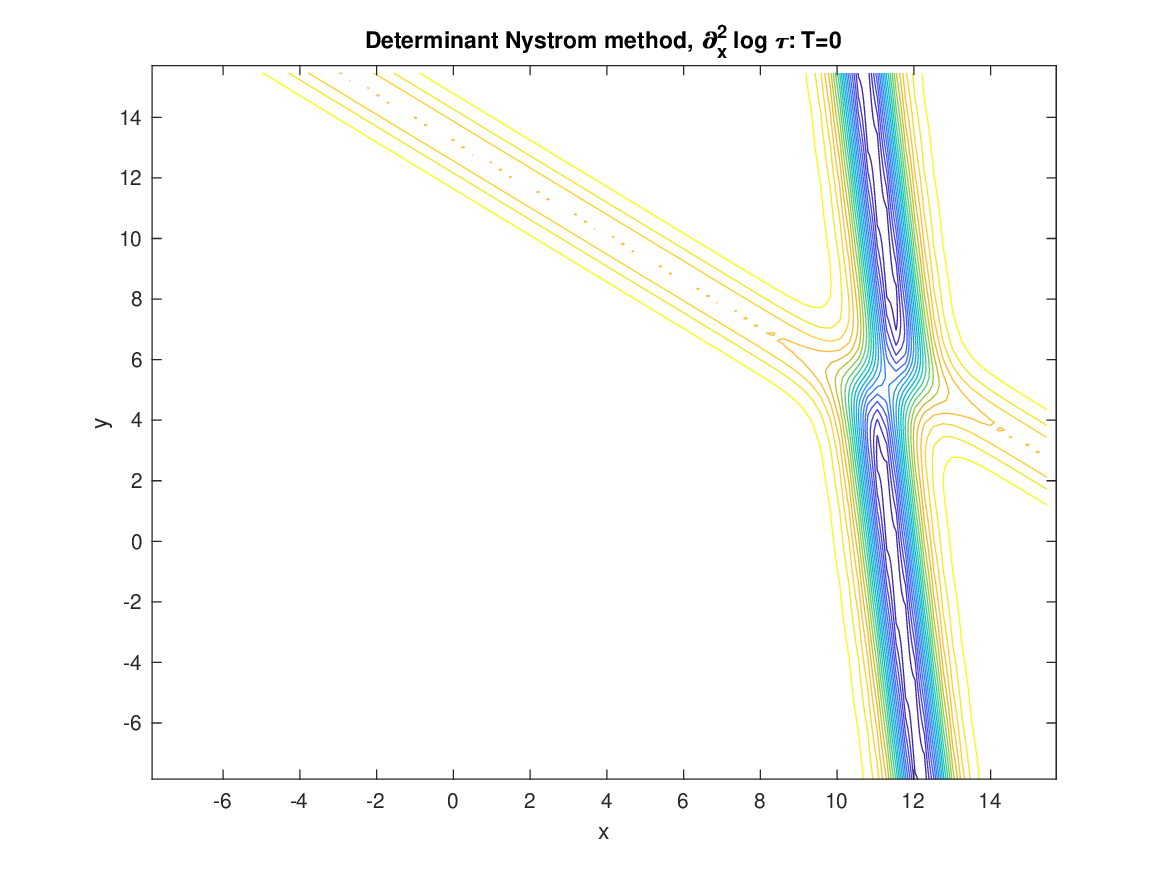}}
  \end{center}
  \caption{We plot the two-soliton interaction solution outlined in Example~\ref{ex:twosoliton} at $t=0$.
    The top panel set shows the solution computed by numerically solving the GLM equation~\eqref{eq:GLMPoppe} using Clenshaw--Curtis quadrature, i.e. the GLM-CC method.
    The bottom panel set shows the solution computed using the $\tau$-function Fredholm determinant, i.e. using the Nystr\"om--Clenshaw--Curtis method Det-CC. 
    The right-hand panels show the corresponding contour plots.}
\label{fig:initialdata}
\end{figure*}

\begin{figure*}
  \begin{center}
    \mbox{\includegraphics[width=7cm,height=6cm]{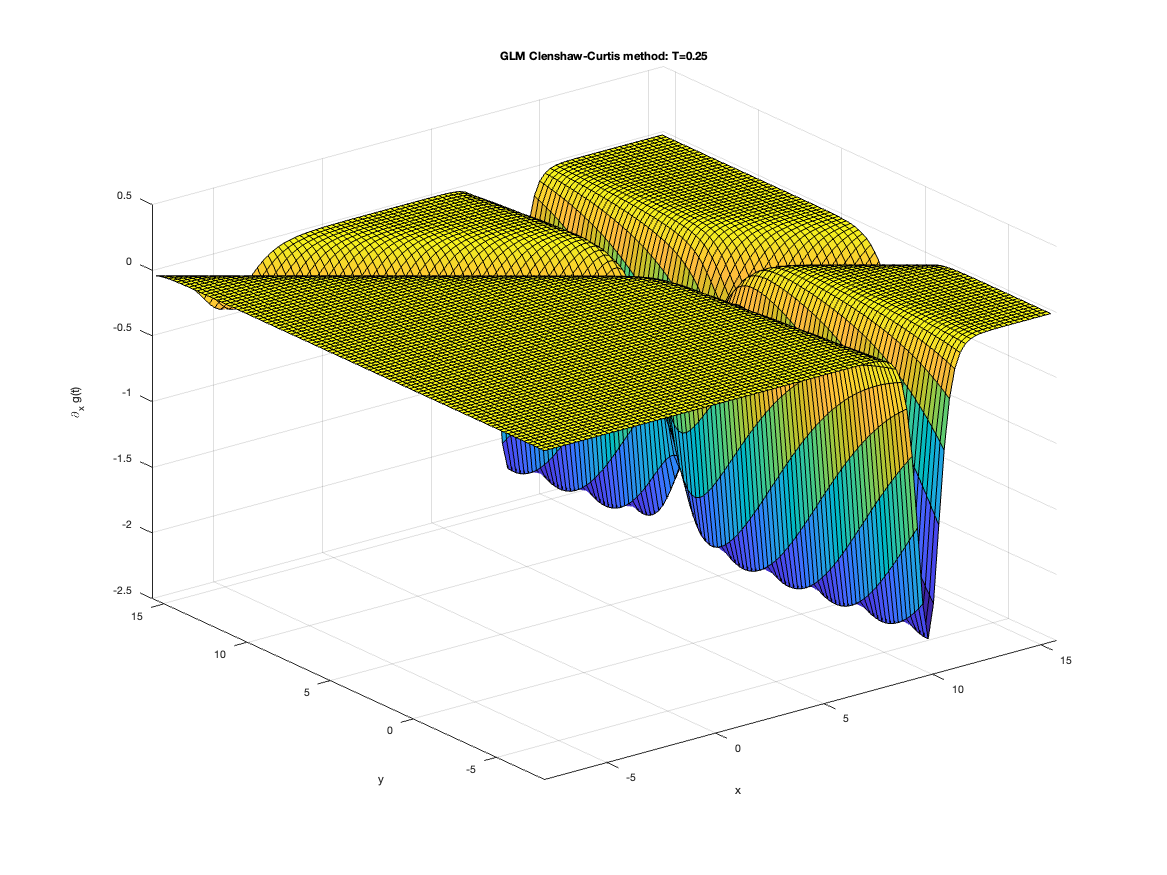}
    \includegraphics[width=7cm,height=6cm]{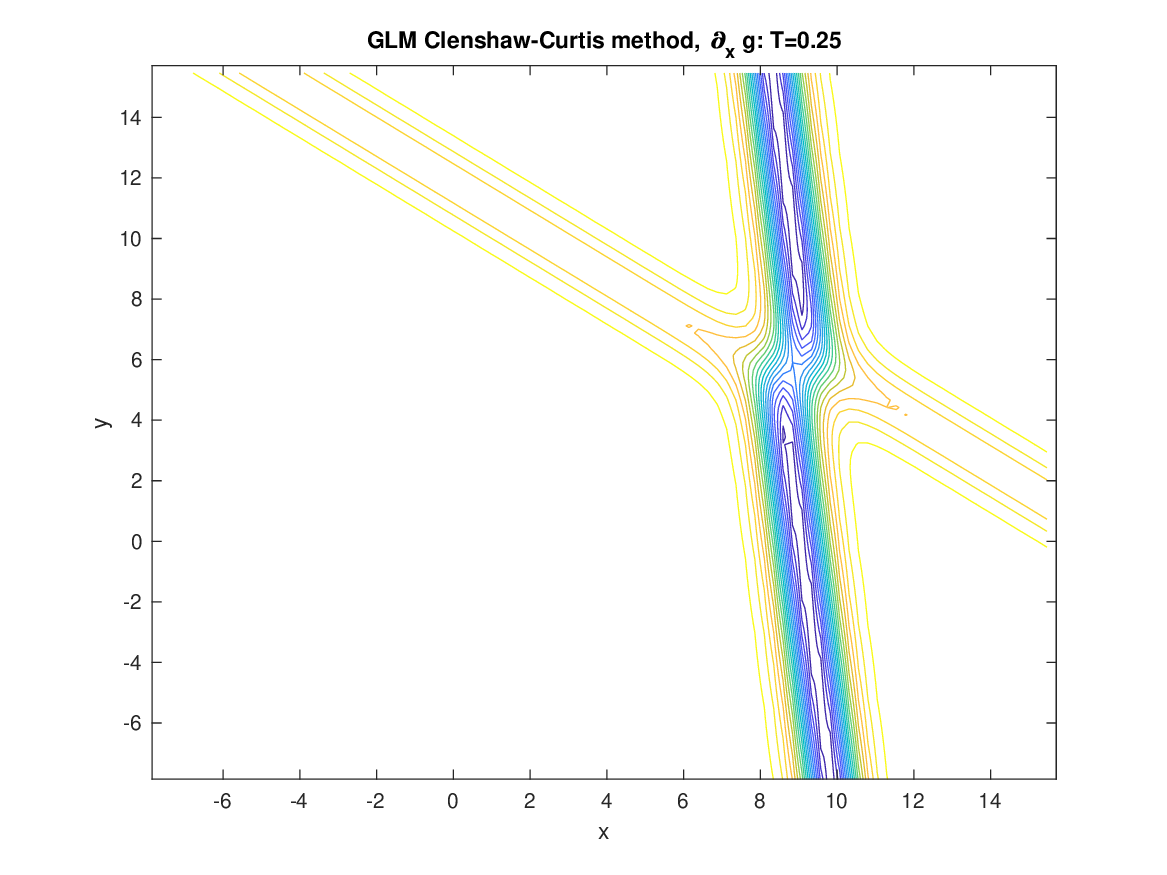}}\\ 
    \mbox{\includegraphics[width=7cm,height=6cm]{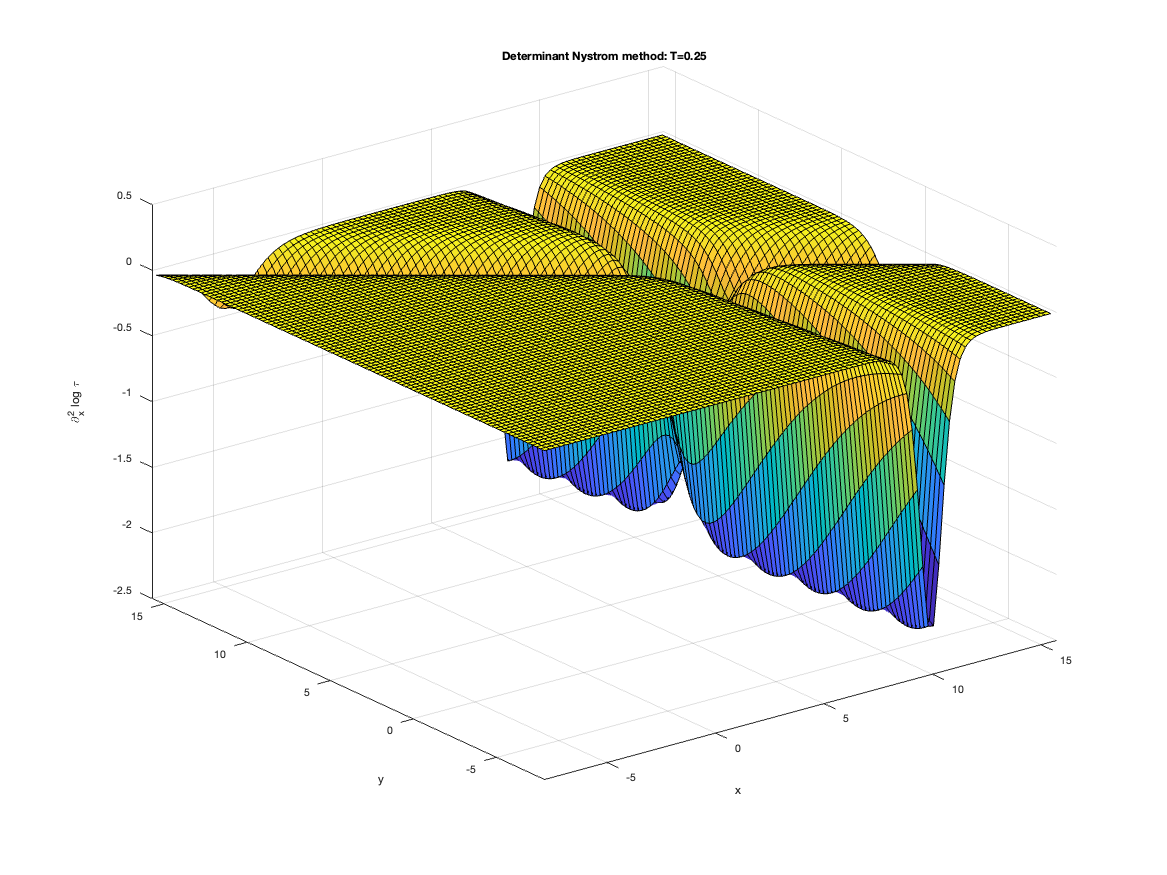}
    \includegraphics[width=7cm,height=6cm]{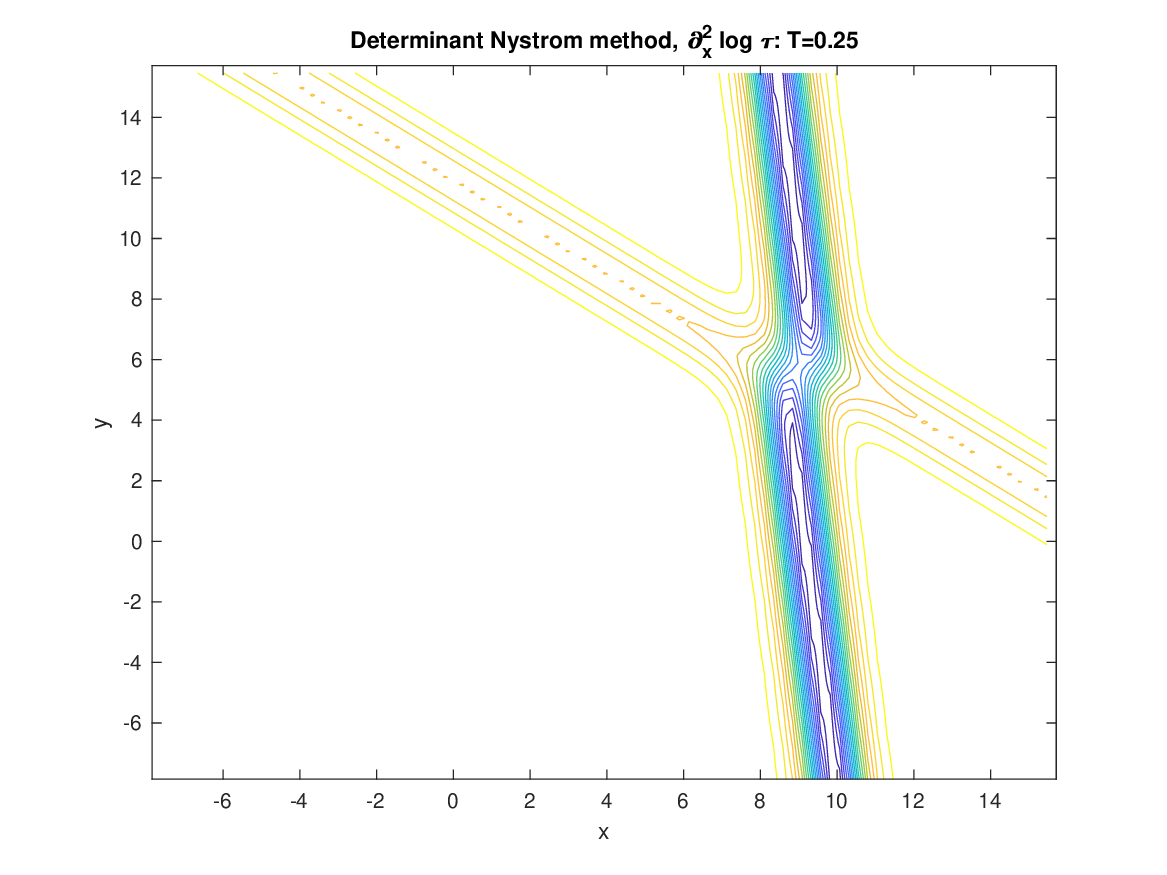}}\\
    \mbox{\includegraphics[width=7cm,height=6cm]{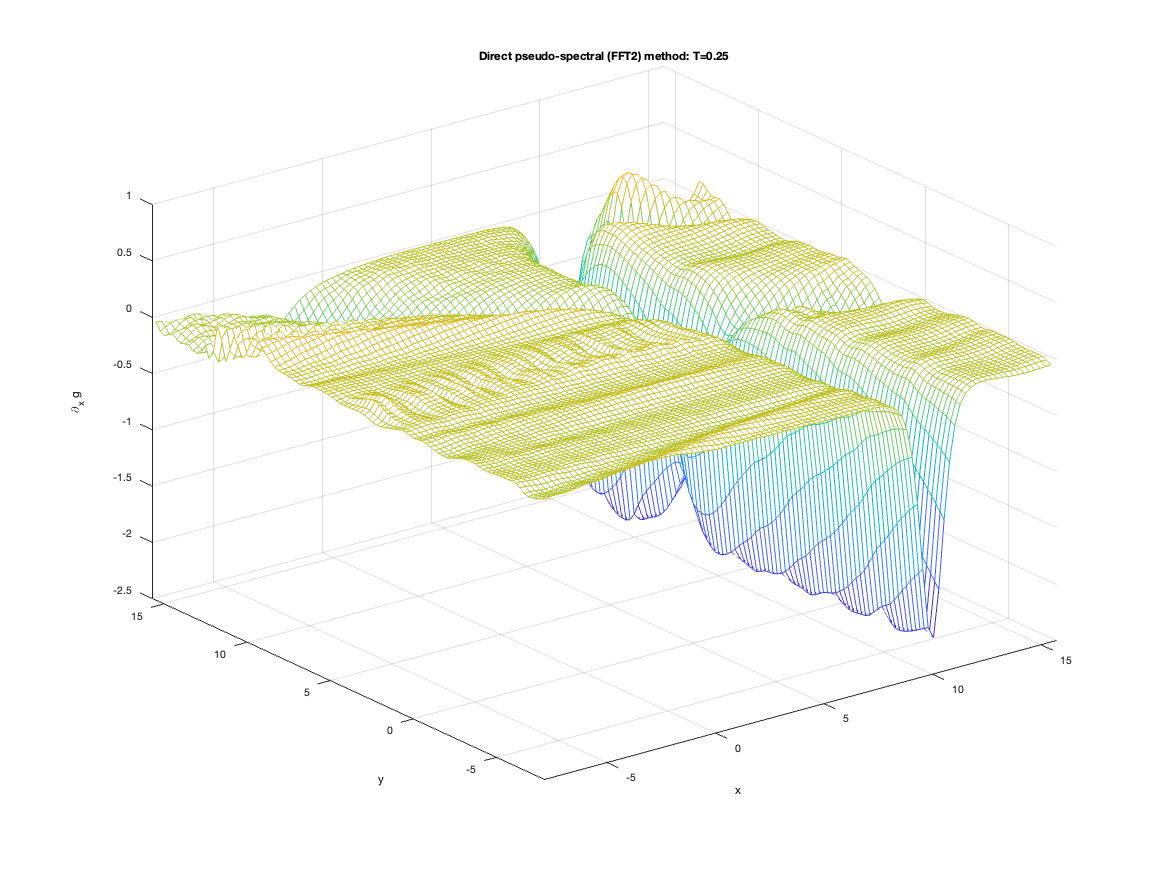}
    \includegraphics[width=7cm,height=6cm]{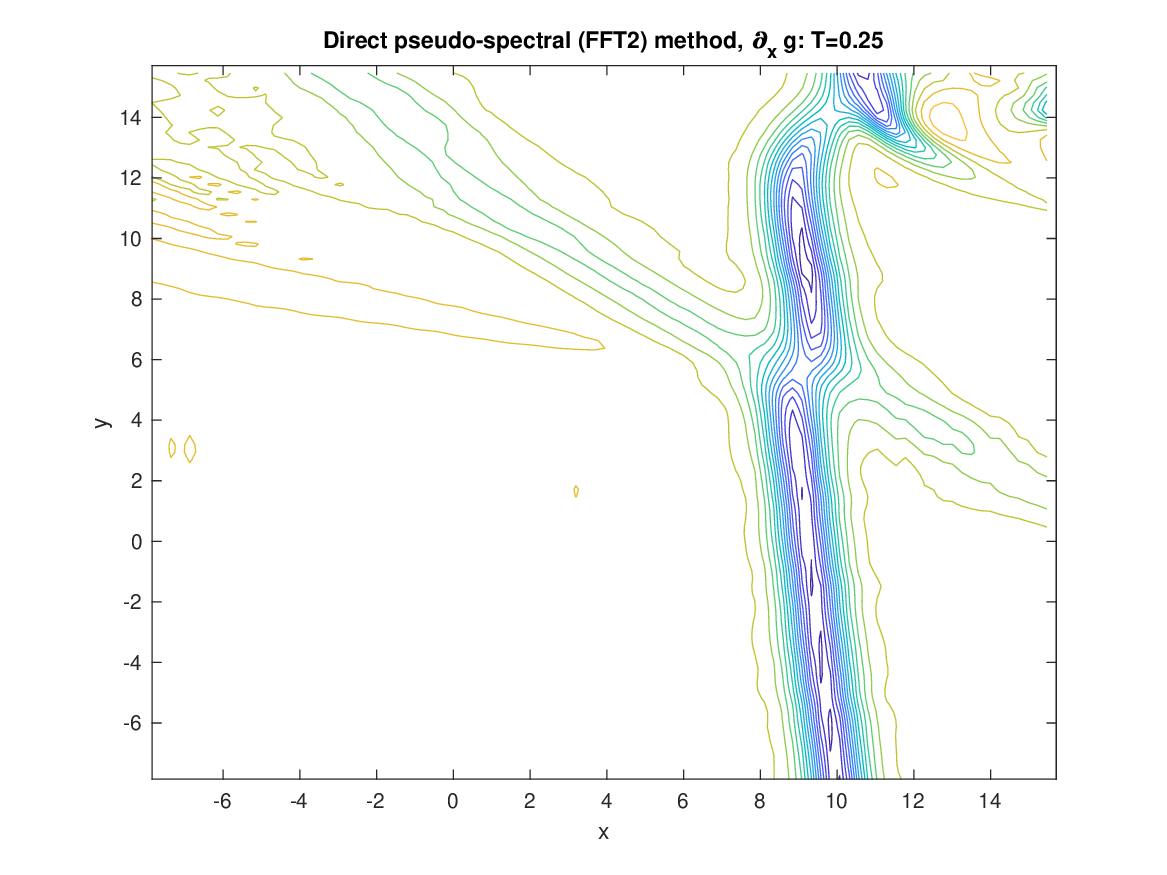}}
  \end{center}
  \caption{We plot the two-soliton interaction solution outlined in Example~\ref{ex:twosoliton} at $t=0.25$.
    The top panel set shows the solution computed by numerically solving the GLM equation~\eqref{eq:GLMPoppe} using the GLM-CC method. 
    The middle panel set shows the solution computed using the Det-CC method.
    The bottom panel set shows the solution computed by direct numerical integration using the FFT2-exp method. 
    The right-hand panels show the corresponding contour plots.}
\label{fig:twowaveinteraction}
\end{figure*}

\begin{ex}[Two-soliton interaction]\label{ex:twosoliton}
  We use the four algorithms we have outlined above to generate the solution to the $KP$ equation corresponding to the following two-soliton interaction scattering data.
  Suppose $p_1$ is the one-soliton scattering data of the form~\eqref{eq:groundstate} generated by the values $a=1.55$ and $b=1.45$,
  while $p_2$ is the one-soliton scattering data~\eqref{eq:groundstate} generated by $a=1.3$ and $b=0$.
  In this example, we assume the two-soliton interaction scattering data, $p=p_1+p_2$.
  In our computations, we assume the truncated domain lengths to be $L_x=L_y=10\pi$.
  Figure~\ref{fig:initialdata} shows the solution $\pa_xg(0,0;x,y,t)$ computed at $t=0$, both using the GLM-CC method and the Det-CC method,
  corresponding to the two-soliton interaction scattering data $p$ with $t=0$.
  In both cases the number of Clenshaw--Curtis nodal points used was $M=2^7$, while the number of $(x,y)$ evaluation points $N_x\times N_y$ was $2^7\times 2^7$.
  Figure~\ref{fig:twowaveinteraction} shows the solution $\pa_xg(0,0;x,y,t)$ computed at $t=0.25$, using the GLM-CC, Det-CC and FFT2-exp methods.
  In the case of the GLM-CC and Det-CC methods, the two-soliton interaction scattering data $p$ with $t=0.25$ was used.
  In both cases, the number of Clenshaw--Curtis nodal points and $(x,y)$ evaluation points used was the same as for the $t=0$ case.
  In the case of the FFT2-exp method, the number of $(x,y)$ evaluation points $N_x\times N_y$ used, was $2^7\times 2^7$, while $10^4$ timesteps where used on $[0,0.25]$.
  The initial data used for the FFT2-exp method was the output of the GLM-CC method at $t=0$ shown in the top panel of Figure~\ref{fig:initialdata}.
  We observe from Figure~\ref{fig:twowaveinteraction} that the accuracy of the FFT2-exp method does not match that of the GLM-CC and Det-CC methods.
  Note that in both Figures~\ref{fig:initialdata} and \ref{fig:twowaveinteraction} we only display the domain region $(x,y)\in[-L_x/4.L_x/2]\times[-L_y/4,L_y/2]$
  and the $x$- and $y$- coordinates of $p$ are shifted so the interaction occurs in this region.
\end{ex}

\begin{figure*}
  \begin{center}
     \mbox{\includegraphics[width=7cm,height=6cm]{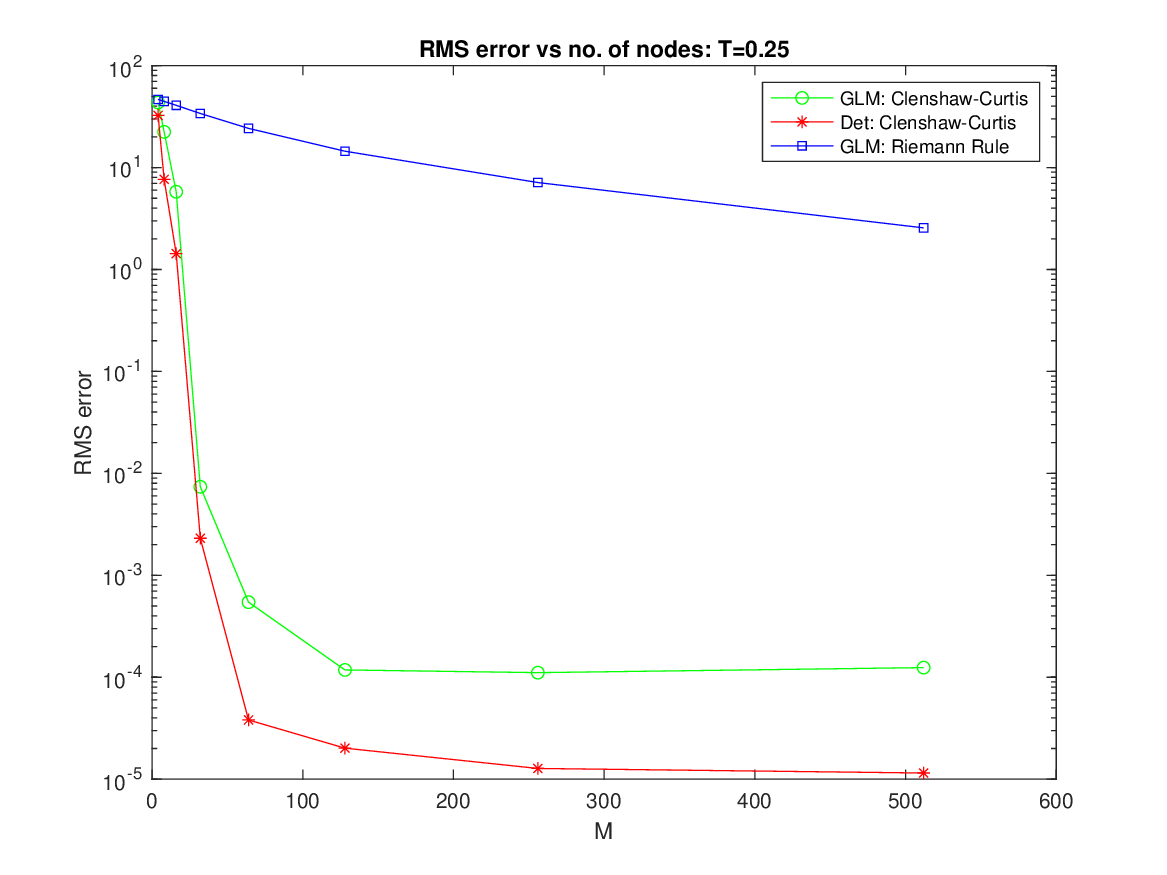}
     \includegraphics[width=7cm,height=6cm]{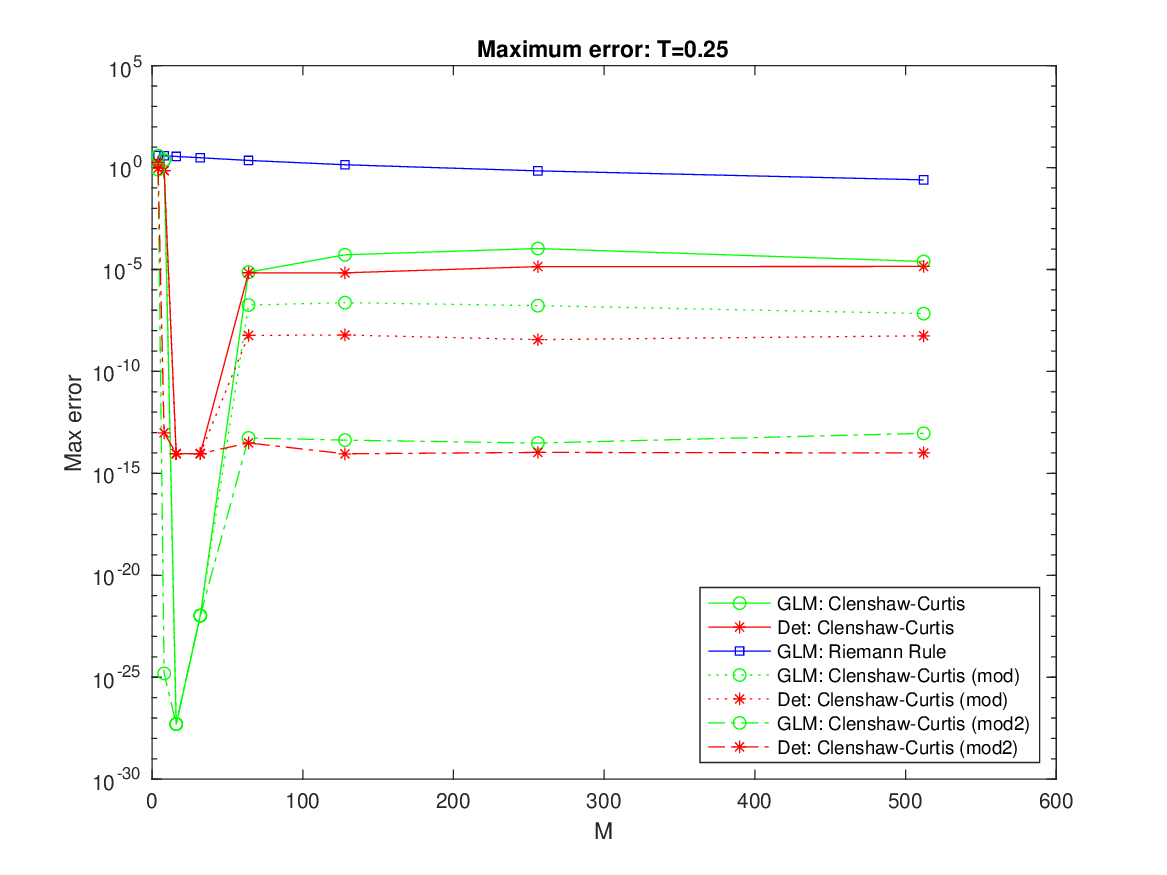}}\\ 
     \mbox{\includegraphics[width=7cm,height=6cm]{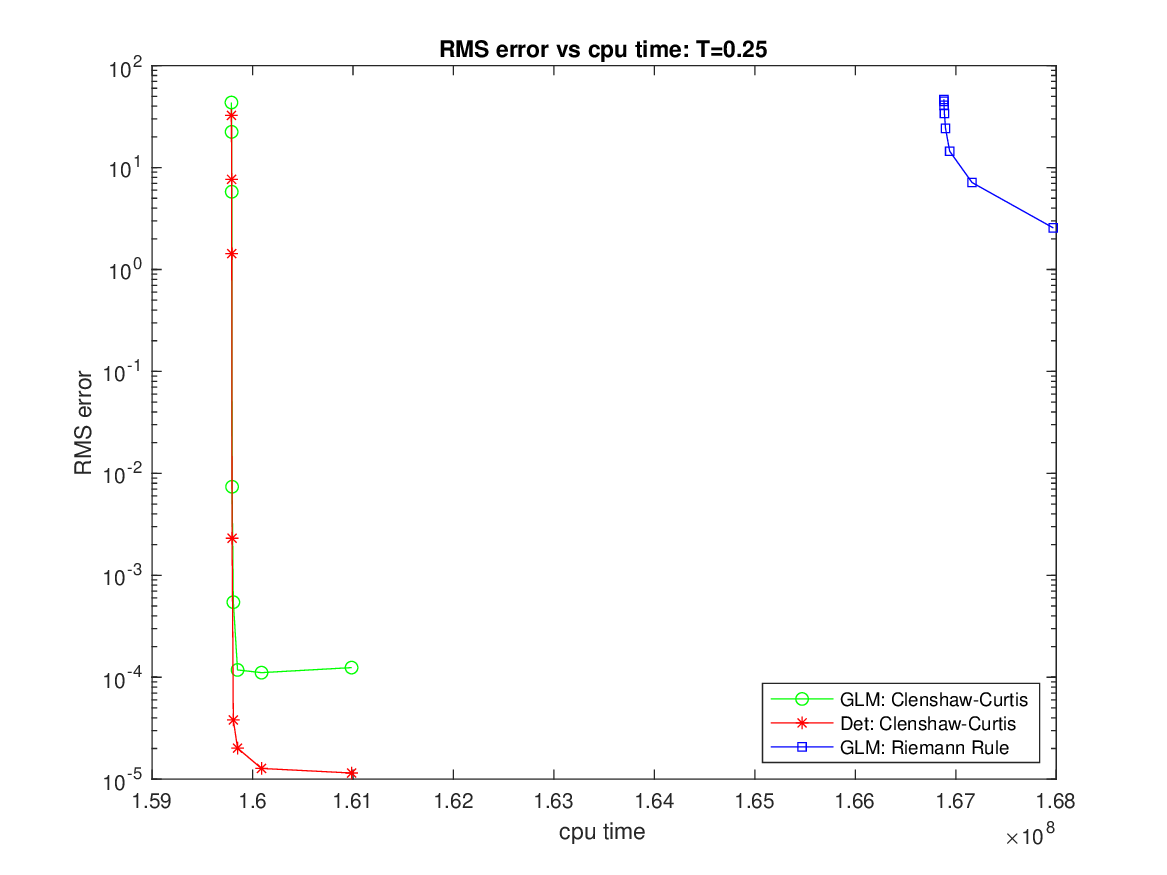}
     \includegraphics[width=7cm,height=6cm]{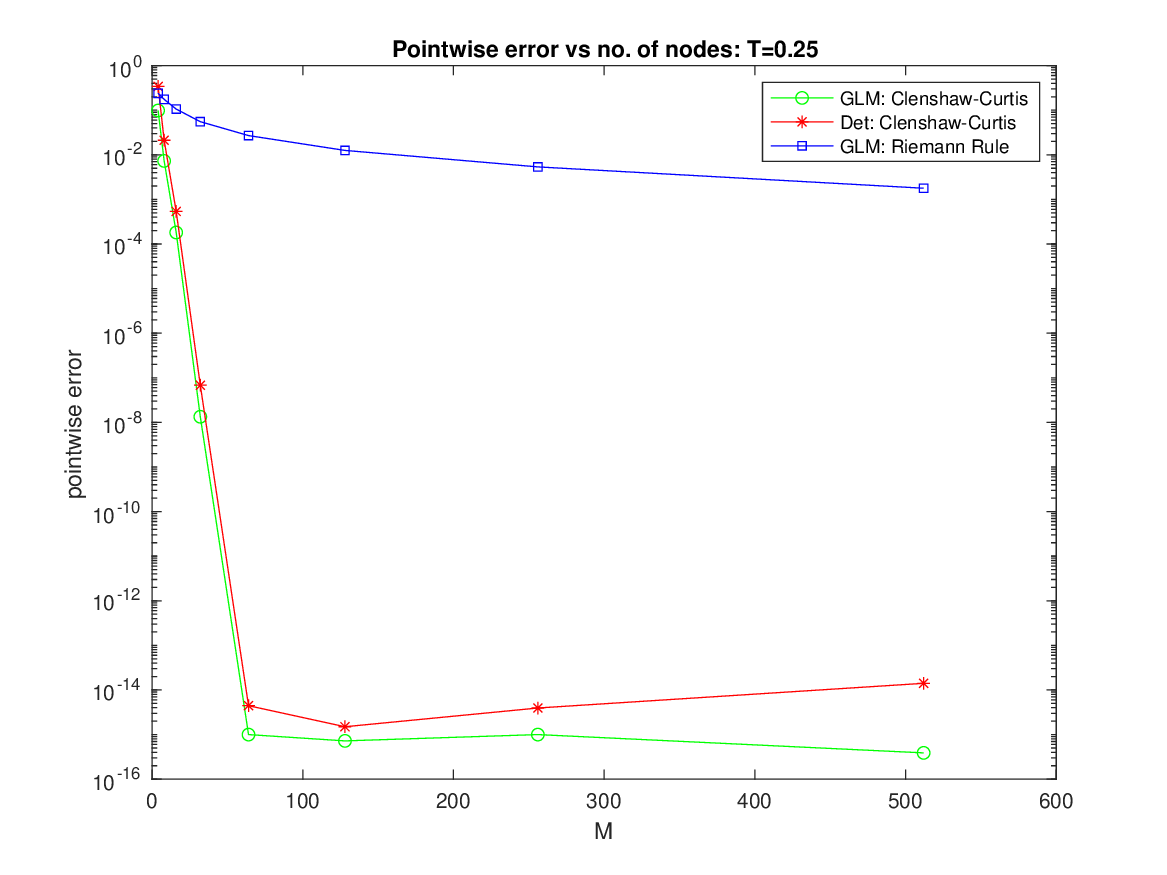}}
  \end{center}
  \caption{We plot the errors associated with the two-soliton interaction solution outlined in Example~\ref{ex:twosoliton} at $t=0.25$; 
    as shown in Figure~\ref{fig:twowaveinteraction}.
    The top panels show the root-mean square error (left panel) and the maximum error (right panel) versus the number of nodal points $M$ used
    in the Clenshaw--Curtis or Riemann Rule quadrature to compute the solutions at each point $(x,y)\in[-L_x/2,L_x/2]\times[-L_y/2,L_y/2]$.
    The bottom left panel shows the root-mean square error versus the cputime required to compute the solution, corresponding to the top left panel plot.
    The bottom right panel shows the pointwise error (right panel) versus the number of nodal points $M$.
    A generic point was chosen, in this case $x=y=6.4$, to compute the pointwise error.}
\label{fig:errors}
\end{figure*}

\begin{figure*}
  \begin{center}
     \mbox{\includegraphics[width=7cm,height=6cm]{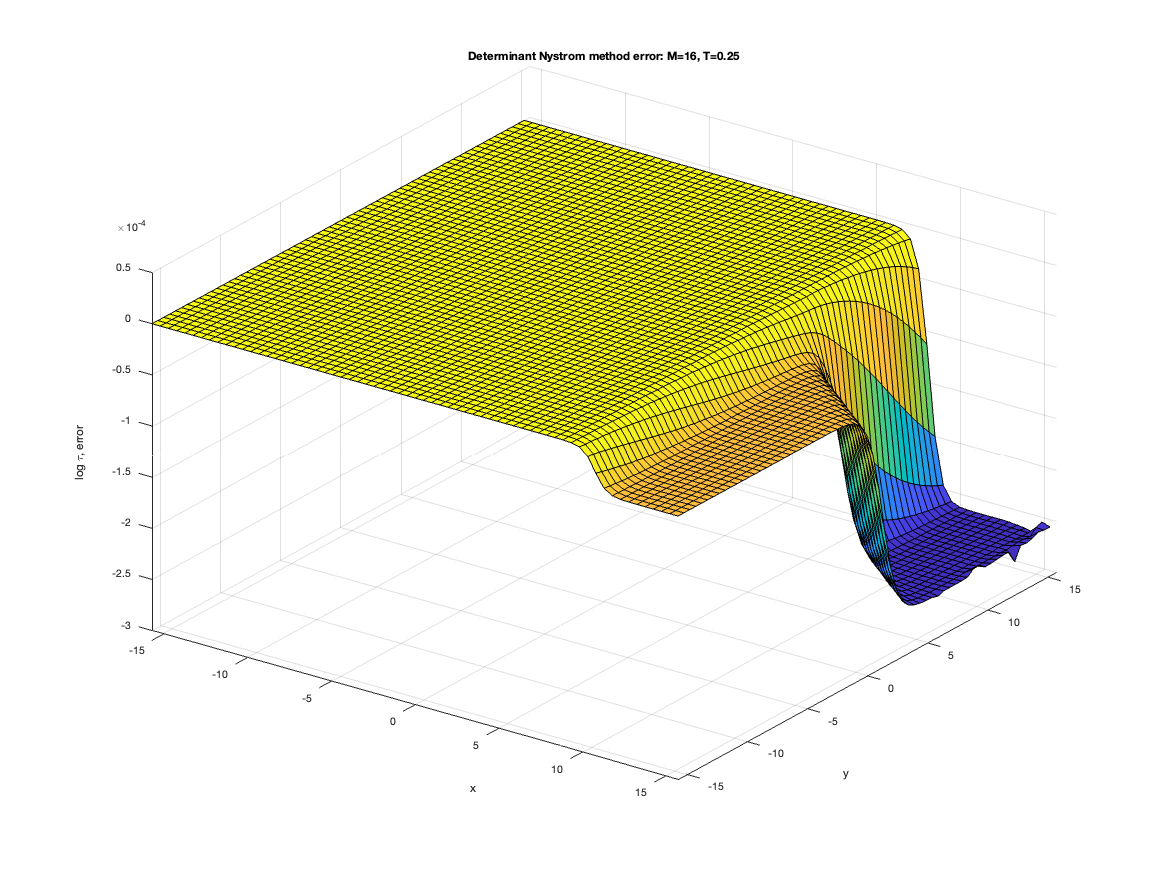}
     \includegraphics[width=7cm,height=6cm]{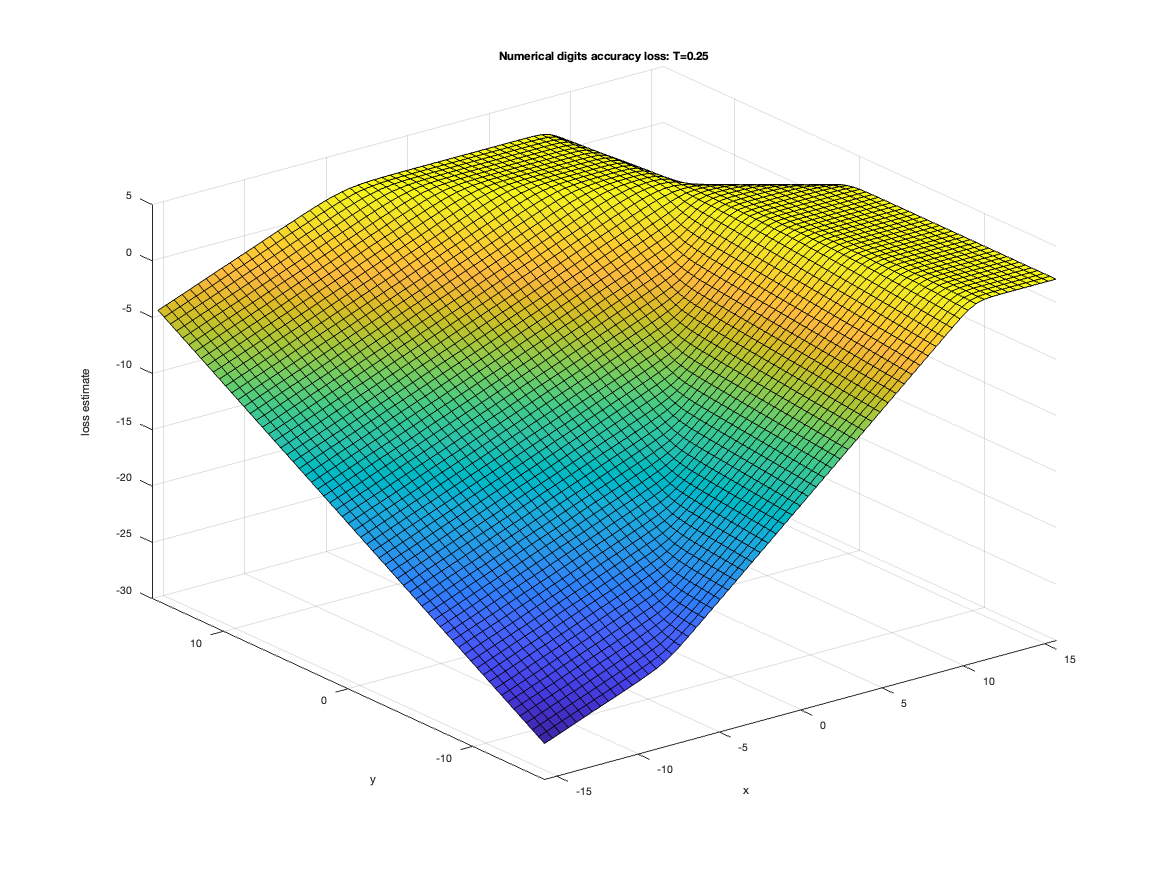}}
  \end{center}
  \caption{The left panel shows, for the two-soliton interaction solution in Example~\ref{ex:twosoliton} at $t=0.25$ computed with the Det-CC method, the difference between
    the solution computed using $M=2^4$ Clenshaw--Curtis nodal points and the solution computed using the maximum number of such nodal points that we used, namely $M=2^{10}$.
     The right panel shows the estimate~\eqref{eq:numericallossestimate} for the number of digits of accuracy lost in the Det-CC method.}
\label{fig:errorsexplained}
\end{figure*}
\section{Exponential convergence}\label{sec:expconv}
Herein, we compare the performance of the numerical methods we used in Section~\ref{sec:NumericalSimulations}, in particular the GLM-RR, GLM-CC and Det-CC methods.
See Figures~\ref{fig:errors} and \ref{fig:errorsexplained}.  
The remarkable properties of the use of Clenshaw--Curtis quadrature to compute Fredholm determinants are comprehensively outlined in Bornemann~\cite{Bornemann}.
In the left panels in Figure~\ref{fig:errors}, we estimated the errors involved in computing using the GLM-RR, GLM-CC and Det-CC methods as follows.
We computed the corresponding solution at $N_x\times N_y$ evaluation points with $N_x=N_y=2^6$.
In each case we used $2^m$ quadrature nodal points as $m$ varied from $2$ through to $10$.
In the case of the GLM-CC computations, we evaluated the numerical error by considering the difference of the solution $g(0,0;x_n,y_{n^\prime},t)$
computed using $2^2$, $2^3$, $\ldots$, $2^9$ Clenshaw--Curtis nodal points, to the corresponding solution $g(0,0;x_n,y_{n^\prime},t)$ computed using $2^{10}$ Clenshaw--Curtis nodal points.
For the GLM-RR method, instead of Clenshaw--Curtis quadrature nodal points, we just use the corresponding number of uniform quadrature nodal points.
In the case of the Det-CC method we computed the difference, between the $\tau$ function $\tau\coloneqq\mathrm{det}(\id-P)$ in \eqref{eq:DetCC},
i.e.\/ $\tau(x_n,y_{n^\prime},t)$ computed using $2^2$, $2^3$, $\ldots$, $2^9$ Clenshaw--Curtis nodal points, and the corresponding $\tau$ function computed using $2^{10}$ Clenshaw--Curtis nodal points.
For all methods, the time was set to be $t=0.25$.
We estimated the root-mean square (RMS) error in the left two panels by computing the Frobenius norm scaled by $(L_xL_y/N_xN_y)^{1/2}$, of all the differences across the evaluation points $(x_n,y_{n^\prime})$
for all $n\in\{0,1,\ldots,N_x\}$ and $n^\prime\in\{0,1,\ldots,N_y\}$.\par
\indent The top left panel shows a log-linear plot of the RMS error versus the number of Clenshaw--Curtis nodal points $M$ while the lower left panel shows a log-linear plot of the RMS error
versus the CPU time required to compute the solution for the corresponding number of Clenshaw--Curtis nodal points.
The superior error of the GLM-CC and Det-CC methods compared to the GLM-RR method is immediately apparent, though their computation times are only slightly better.  
Indeed, as Bornemann~\cite[p.~891]{Bornemann} points out, for analytic kernels, we expect \emph{exponential convergence} for the Det-CC method,
and the exponential two-soliton interaction form for $p$ we have assumed here is analytic. 
In the bottom right panel in Figure~\ref{fig:errors}, in a log-linear plot, we show the convergence of the GLM-RR, GLM-CC and Det-CC methods at a specific generic point, in this case $x=y=6.4$.
We observe that the convergence of both the Det-CC and GLM-CC methods is exponential, and indeed, it hits an error of order $10^{-15}$ relatively rapidly. 

This exponential convergence can also be seen in the top panels in Figure~\ref{fig:errors}, though the convergence flattens off at roughly $10^{-5}$ for $M$ beyond $2^5$.
This can be explained as follows.
In the left panel in Figure~\ref{fig:errorsexplained}, for the two-soliton interaction solution in Example~\ref{ex:twosoliton} at $t=0.25$ computed with the Det-CC method,
we plot the difference between the solution computed using $M=2^4$ Clenshaw--Curtis nodal points and the reference `exact' solution we computed using $M=2^{10}$ such nodal points.
We observe that this error estimate is smooth everywhere except for the region where $x$ and $y$ are large, close to their maximum values of $5\pi$. 
We observe some ``wrinkles'' in this error plot of order $10^{-5}$ in this region.
All Clenshaw--Curtis nodal points $(\xi_{m^\prime},\zeta_m)$, which are based on Chebyshev nodal points, are interior to the boundary of the domain $[-L_x/2.L_x/2]\times[-L_y/2,L_y/2]$.
When $M=2^m$ is small, say with $m\leqslant5$, the largest Clenshaw--Curtis nodal points are still relatively far from the boundary.
However, when $m>5$, they do become close.
Our scattering data kernel $p$ grows exponentially for large and positive values of $\xi$, $\zeta$, $x$ and $y$.
Indeed, when these values are close to the boundary, $5\pi$, then $p$ is of order $10^{27}$.
The accuracy of the linear algebra computations that underlie the GLM-CC and Det-CC methods is compromised in this situation and delivers an accuracy of $10^{-5}$ consistent with
the order of magnitude of the identified ``wrinkles''. We observe this in the top right panel in Figure~\ref{fig:errors}.
Therein, for the GLM-CC and Det-CC methods, we observe the exponential convergence of the maximum norm of the error for small $M\leqslant2^5$, but once $M$ is larger than this,
the maximum norm of the error over the whole domain reverts to an error of order $10^{-5}$.
In the top right panel in Figure~\ref{fig:errors}, we also computed the maximum norm of the error for the GLM-CC and Det-CC methods over the
restricted domains for which $x\in[-5\pi,10.8]$, indicated by the `mod' label, and for $x\in[-5\pi,0]$, indicated by the `mod2' label.
In the first case we observe that the maximum norm errors flatten out at about $10^{-8}$, while in the second case it flattens out at roughly $10^{-14}$.
This indicates that the influence of accuracy loss due to very large values of $p$ in our computations naturally recedes as $x_n$ decreases.
And this is consistent with our pointwise error estimate in the bottom right panel in Figure~\ref{fig:errors}, for which $x=y=6.4$.
Since the root-mean square error estimate is global, this also explains why this error estimate in the two left panels in Figure~\ref{fig:errors} flattens out at roughly $10^{-5}$.

\begin{rem}[Other sources of accuracy loss]
Bornemann~\cite[p.~884]{Bornemann} outlines that if $\det(\id-P)\ll\|P\|_{\mathcal{L}^2}$, where $\|P\|_{\mathcal{L}^2}$ is the Hilbert--Schmidt norm of $P$,
then a ``conservative estimate'' predicts a loss of some digits of accuracy, when computing the Fredholm determinant using Nystr\"om--Clenshaw--Curtis quadrature, of at most, 
\begin{equation}\label{eq:numericallossestimate}
\log_{10}\biggl(\frac{\sqrt{M}\cdot\|P\|_{\mathcal{L}^2}}{\det(\id-P)}\biggr),
\end{equation}
decimal places.
Note that in practice, given the matrix approximation $\widehat{Q}$ of $P$, we used the Frobenius norm scaled by $(L_xL_y/N_xN_y)^{1/2}$, to approximate $\|P\|_{\mathcal{L}^2}$.
Thus, as an additional check, in the right panel in Figure~\ref{fig:errorsexplained} we plot the estimate \eqref{eq:numericallossestimate} for the nodal points $(x_n,y_{n^\prime})$ we used in the domain region.
Nodal points where the surface shown is positive, indicates the potential number of decimal places lost at that nodal point.
Across the whole domain region, the maximum this estimate reaches is roughly $2$, indicating that this phenomenon is only a minor contribution in our error estimates and analysis.
\end{rem}

For analytic scattering data kernels, the exponential convergence of the GLM-CC and Det-CC methods, based on Bornemann's \cite{Bornemann} work,
marks these methods out as extremely powerful tools in the simulation of $KP$ solutions. These methods warrant further investigation.
Further, Bornemann's comprehensive analysis reveals that the order of convergence is linearly related to the smoothness of the scattering kernel.
Of interest in this direction is, given initial data for $g$ with only a certain degree of smoothness, is to solve the scattering problem to generate the corresponding initial scattering data. 
This can be evolved forward in time to $t>0$ in Fourier space via fast Fourier transform. This would result in a scattering data kernel at time $t$, of limited smoothness.
Then we would apply the GLM-CC and/or Det-CC methods to generate the corresponding $KP$ solution at time $t>0$. This is future work.

\section{Acknowledgement}
The authors acknowledge the helpful contribution of the late Henry McKean, who drew attention to the development potential of P\"oppe's work.



\end{document}